\documentclass[12pt]{amsart}
\usepackage[utf8]{inputenc}
\usepackage{amsmath}
\usepackage{amssymb}
\usepackage{fullpage}
\usepackage{graphicx}
\usepackage{enumerate}
\usepackage{color}
\usepackage{times,wrapfig}
\usepackage{booktabs,marvosym,pifont}
\usepackage{tikz,makecell,multirow}

\usepackage[colorlinks = true,
linkcolor = blue,
urlcolor  = blue,
citecolor = blue,
anchorcolor = blue]{hyperref}
\usepackage{booktabs,xcolor,siunitx,wrapfig}
\usepackage[skip=0pt]{caption}
\usepackage{mathrsfs,longtable}
\usepackage[textfont={color=black,it},figurename=Figure,labelfont={it}]{caption}

\allowdisplaybreaks[2]

\usepackage{tikz}
\usetikzlibrary{positioning}
\usetikzlibrary{arrows.meta}
\tikzset{%
	>={Latex[width=2mm,length=2mm]},
	base/.style = {
		rectangle, 
		rounded corners, 
		draw=black,
		minimum width=6cm, 
		minimum height=1.5cm,
		text centered
	},
	start/.style = {base, fill=blue!20},
	stop/.style = {base, fill=red!20},
	process/.style = {base, minimum width=6.5cm, fill=orange!15},
	every picture/.style = {line width=1pt}
}

\newcounter{mycount}
\theoremstyle{plain}
\newtheorem{theorem}[mycount]{Theorem}
\newtheorem{corollary}[mycount]{Corollary}

\newtheorem{lemma}[mycount]{Lemma}
\newtheorem{proposition}[mycount]{Proposition}

\newtheorem{remark}{Remark}
\theoremstyle{definition}
\newtheorem{definition}{Definition}

\theoremstyle{example}

\theoremstyle{remark}
\numberwithin{equation}{section}
\numberwithin{figure}{section}

\def\dd{\,\mathrm{d}}
\def\le{\leqslant}
\def\ge{\geqslant}
\def\tr#1{\lfloor #1\rfloor}

\def\cl#1{\lceil #1\rceil}

\def\lpa#1{\bigl({#1}\bigr)}
\def\Lpa#1{\Bigl({#1}\Bigr)}

\def\llpa#1{\biggl({#1}\biggr)}
\def\ve{\varepsilon}
\DeclareRobustCommand{\Eulerian}{\genfrac\langle\rangle{0pt}{}}

\makeatletter 
\def\l@subsection{\@tocline{2}{0pt}{2pc}{6pc}{}} 
\makeatother

\makeatletter
\renewcommand*\env@matrix[1][\arraystretch]{%
	\edef\arraystretch{#1}%
	\hskip -\arraycolsep
	\let\@ifnextchar\new@ifnextchar
	\array{*\c@MaxMatrixCols c}}
\makeatother

\graphicspath{{./rf_figs/}}

\def\with{\mbox{\quad with\quad}}
\def\and{\mbox{\quad and \quad}}

\begin{document}

\title[Asymptotics of Fishburn matrices and their 
generalizations]{Asymptotics and statistics on Fishburn matrices and 
their generalizations}

\let\origmaketitle\maketitle
\def\maketitle{
  \begingroup
  \def\uppercasenonmath##1{} 
  \let\MakeUppercase\relax 
  \origmaketitle
  \endgroup
}

\author{Hsien-Kuei Hwang and Emma Yu Jin}
\thanks{The research of the first author is partially supported by
FWF-MOST (Austria-Taiwan) grant MOST 104-2923-M-009-006-MY3, and by
an Investigator Award from Academia Sinica under the Grant
AS-IA-104-M03. The second author was partially supported by FWF-MOST 
(Austrian-Taiwanese) Project I 2309-N35 and FWF Project P 32305.}

\address{Institute of Statistical Science, Academia Sinica, Taipei,
115, Taiwan}
\email{\textbf{hkhwang@stat.sinica.edu.tw}}

\address{Fakult\"at f\"ur Mathematik, Universit\"{a}t Wien, 1090
Vienna, Austria}
\email{\textbf{yu.jin@univie.ac.at}}

\maketitle

\begin{abstract}
A direct saddle-point analysis (without relying on any modular forms or functional equations) is developed to establish the
asymptotics of Fishburn matrices and a large number of other variants
with a similar sum-of-finite-product form for their (formal) general
functions. In addition to solving some conjectures, the application
of our saddle-point approach to the distributional aspects of
statistics on Fishburn matrices is also examined with many new limit
theorems characterized, representing the first of their kind for such
structures.
	
\end{abstract}


\section{Motivations and Background}

Fishburn matrices, introduced in the 1970s in the context of interval
orders (in order theory) and directed graphs (see \cite{Andresen1976,
Fishburn1985, Greenough1976, Greenough1976a}), are nonnegative,
upper-triangular ones without zero row or column. They have been
later found to be bijectively equivalent to several other
combinatorial structures such as ${\bf (2+2)}$-free posets, ascent
sequences, pattern-avoiding permutations, pattern-avoiding inversion
sequences, Stoimenow matchings, and regular chord diagrams; see, for
instance, \cite{Bousquet-Melou2010, Dukes2010, Fu2020, Jelinek2012,
Levande2013} and Section~\ref{S:comb} for more information.

In addition to their rich combinatorial connections and modeling
capibilities, the corresponding asymptotic enumeration and the finer
distributional properties are equally enriching and challenging, as
we will explore in this paper. In particular, while the asymptotics
of some classes of Fishburn matrices were known (see, for example,
\cite{Bringmann2014, Zagier2001}), the stochastic aspects of the
major characteristic statistics on random Fishburn matrices have 
remained open up to now.

Zagier, in his influential paper \cite{Zagier2001} on Vassiliev
invariants and quantum modular forms, derived the asymptotic
approximation to the number of Fishburn matrices whose entries sum 
to $n$
\begin{align}\label{E:zag}
	\begin{split}
	    [z^n]\sum_{k\ge0}
		\prod_{1\le j\le k}\lpa{1-(1-z)^j}
		&= c \rho^{n}\, n^{n+1}
		\left(1+O\lpa{n^{-1}}\right),
	\end{split}
\end{align}
(see OEIS \cite{oeis2019} sequence
\href{https://oeis.org/A022493}{A022493}, the Fishburn numbers),
where $(c,\rho):=\lpa{\tfrac{12\sqrt{6}} {\pi^2}\,
e^{\frac{\pi^2}{12}}, \tfrac6{e\pi^2}}$. Here $[z^n]f(z)$ denotes the
coefficient of $z^n$ in the (formal) Taylor expansion of $f$. For
conciseness of notation and readability, \emph{all constant pairs
$(c,\rho)$ throughout this paper are generic and may not be the same
at each occurrence}; their values will be locally specified.

That the asymptotic approximation \eqref{E:zag} is remarkable can be
viewed in various perspectives. First, the Taylor coefficients of the
inner product on the left-hand side of \eqref{E:zag} alternate in
sign, so it is unclear if the coefficient of $z^n$ in the
sum-of-product expression is positive for all positive $n$, much less
its large factorial growth order shown on the right-hand side. Second,
since $\frac{6}{\pi^2}<1$, the right-hand side of \eqref{E:zag} is
exponentially smaller than $n!$, which equals $[z^n]\prod_{1\le j\le
n}\lpa{1-(1-z)^j}$. More precisely, we will prove that (see
Lemma~\ref{L:max-phi} and Proposition \ref{P:poly})
\[
    \max_{1\le k\le n}
    \Bigl|[z^n]\prod_{1\le j\le k}\lpa{1-(1-z)^j}\Bigr|
    = \max_{1\le k\le n}
    [z^n]\prod_{1\le j\le k}\lpa{(1+z)^j-1}
    =\Theta\lpa{n^{n+\frac12}\hat{\rho}^n},
\]
meaning that the left-hand side is exactly of order
$n^{n+\frac12}\hat{\rho}^n$ with $\hat{\rho}
:=\frac{12}{e\pi^2}=2\rho$. This shows that there is indeed a heavy
\emph{exponential cancellation} involved in the sum on the left-hand
side of (\ref{E:zag}). Third, in addition to the connection to
linearly independent Vassiliev invariants, the Fishburn numbers have
now been known to enumerate many different combinatorial objects; see
Section~\ref{S:comb}, OEIS sequence
\href{https://oeis.org/A022493}{A022493} and \cite{Andrews2014,
Bousquet-Melou2010, Bringmann2014, Claesson2011, Dukes2011,
Dukes2010, Fishburn1970, Fishburn1985, Jelinek2012, Jelinek2015,
Kitaev2011, Levande2013, Yan2011, Zagier2001} for more information.
Finally, Zagier's proof of \eqref{E:zag} relies crucially on an
unusual pair of identities
\begin{equation}\label{E:glaisher}
    \left\{
    \begin{split}
    	e^{-\frac{z}{24}}\sum_{k\ge0}
    	\prod_{1\le j\le k}\lpa{1-e^{-jz}}
    	&=\sum_{n\ge0}\frac{T_n}{n!}
    	\left(\frac{z}{24}\right)^n,\\
        \sum_{n\ge0}\frac{T_n}{(2n+1)!}\,z^{2n+1}
        &:=\frac{\sin (2z)}{2\cos (3z)}, 
    \end{split}
	\right.
\end{equation}
where the integers $T_n$'s, known as Glaisher's $T$-numbers (see
\href{https://oeis.org/A022493}{A002439}), are defined by the second
identity of (\ref{E:glaisher}). The first one, due to Zagier, is a
consequence of the relation between the generating function in
(\ref{E:zag}) and the ``half derivative'' of the Dedekind
eta-function, partial summation, Euler's pentagonal number theorem,
functional equations, Dirichlet series and Mellin transform
techniques; see \cite{Zagier2001, Zagier2010}. The asymptotics of
$T_n$ is then readily computed by, say using the singularity analysis
(see \cite{Flajolet2009}) on the right-hand side of the second
identity, which, unlike the formal nature of the first, is analytic
in $|z|<\frac{\pi}{6}$. What appears to be more important in
subsequent developments is that $T_n$ is essentially the value
defined by the analytic continuation of some Dirichlet series at
$-2n-1$, and the study of the identities \eqref{E:glaisher} is thus
closely connected to algebraic and analytic number theory, in
addition to their hypergeometric $q$-series nature and resurgence
aspect \cite{Costin2011}. Some similar pairs of relations such as 
\eqref{E:glaisher} are now known; see, for example, 
\cite{Andrews2001, Hikami2006a, Ono2004}. 

\emph{Here and throughout this paper, the asymptotic relation}
\begin{align}\label{E:a-cong-b}
    a_n = b_n \lpa{1+O\lpa{n^{-1}}}\;
    \text{ is abbreviated as }\;a_n \simeq b_n.
\end{align}


Subsequently in \cite{Bringmann2014}, Bringmann-Li-Rhoades established
the asymptotic approximation to the number of primitive row-Fishburn
matrices with entries summing to $n$,
\begin{align}\label{E:blr}
    [z^n]\sum_{k\ge0}
	\prod_{1\le j\le k}\lpa{(1+z)^j-1}
	\simeq c \rho^{n}\, n^{n+\frac12},
	\with (c,\rho):=\lpa{\tfrac{12}{\pi^{3/2}}\,
	e^{-\frac{\pi^2}{24}}, \tfrac{12}{e\pi^2}},
\end{align}
which confirmed a conjecture by Jel\'{i}nek (Conjecture 5.3 of
\cite{Jelinek2012}). Here a primitive row-Fishburn matrix is a binary
upper-triangular one without zero rows. Their proof to deriving
\eqref{E:blr} relies on various properties of the function
($\sigma(q)$ in \cite{Bringmann2014})
\begin{equation}\label{E:ramanujan}
    R(q) := \sum_{k\ge0}\frac{q^{\frac12k(k+1)}}
    {(1+q)\cdots(1+q^k)}
    = 1+\sum_{k\ge0}(-1)^kq^{k+1}
    \prod_{1\le j\le k}\lpa{1-q^j},
\end{equation}
first appeared in Ramanujan's lost notebook, with many unusual
properties discovered since Andrews's paper \cite{Andrews1986}; see 
\cite{Andrews1988,Cohen1988} and 
\href{https://oeis.org/A003406}{A003406} for more information. Very 
roughly, since  
\begin{align*}
    [z^n]\sum_{k\ge0}\prod_{1\le j\le k}\lpa{(1+z)^j-1}
    = \frac{(-1)^n}2[z^n]R(1-z),
\end{align*}
according to Equation (2.3) of \cite{Bringmann2014} and $e^{-z} = 1-z
+ O(|z|^2)$ for small $|z|$, the approach begins by working out the
asymptotics of $[z^n]R\lpa{e^{-z}}$. The bridge between $[z^n]R(1-z)$
and $[z^n]R\lpa{e^{-z}}$ can then be linked through a direct change of
variables and straightforward arguments because $z$ is very close to
zero (the arguments used in \cite{Zagier2001} and
\cite{Bringmann2014} relying instead on the asymptotics of the
Stirling numbers of the first kind); see Section~\ref{S:cov} for more
details.

The asymptotics of $[z^n]R\lpa{e^{-z}}$ is derived by first defining 
the Dirichlet series 
\[
    D(s) := \sum_{n\ge1}n^{-s} [q^{n-1}]R\lpa{q^{24}},
\]
which can be meromorphically continued into the whole $s$-plane. 
Since, by standard Mellin transform techniques (see, e.g., 
\cite{Flajolet1995}),
\[
    [z^n]e^{-z}R\lpa{e^{-24z}}
    = \frac{(-1)^n}{n!}\,D(-n),
\]
the crucial asymptotics of $D(-n)$ needed is then derived by the 
functional equation satisfied by certain function defined on $D(s)$ 
(similar to that satisfied by Riemann's zeta function); see
\cite{Bringmann2014,Cohen1988} and Section~\ref{S:ae} for more
details.

However, a lot more on the asymptotics and statistics of Fishburn
matrices and related structures have remained unknown, which include
several conjectures and open problems \cite{Bringmann2014,
Jelinek2012, Stoimenow1998, Zagier2001} that are of great interest to
the combinatorics and modular-form community. As apparently general
asymptotic techniques are still lacking, we aim to address this gap 
by developing a \emph{direct, self-contained} approach to deriving
\eqref{E:zag} and \eqref{E:blr} in a systematic way without relying
on any functional equations (satisfied by Dirichlet series) or
identities such as \eqref{E:glaisher} that are not available in more
general contexts with a similar sum-of-finite-product form for the
generating functions.

Our approach is based instead on a fine, double saddle-point analysis
and has the additional advantages of being applicable to a large
number of problems whose (formal) generating functions assume a
similar form, some of which are listed as follows. 
\begin{itemize}
	
\item[--] Derive the asymptotics of Fishburn and row-Fishburn
matrices whose entries belong to any multiset of
nonnegative integers containing particularly $0$; in particular,
\eqref{E:zag} and \eqref{E:blr} by Zagier \cite{Zagier2001} and
Bringmann-Li-Rhoades \cite{Brightwell2011}, respectively, are
reproved; our scheme is also applicable to more than two dozens of
other OEIS sequences; see Section \ref{S:A}.

\item[--] Prove a conjecture of Jel\'{i}nek in \cite{Jelinek2012} 
concerning the asymptotics of self-dual Fishburn matrices; see Corollary \ref{C:Jelinek}.

\item[--] Establish the limit distributions of some typical
statistics in random Fishburn matrices, which solves particularly an
open problem raised by Bringmann-Li-Rhoades \cite{Bringmann2014} and
Jel\'{i}nek \cite{Jelinek2012}; see Theorem \ref{T:lf-stat} and Section \ref{S:B}.

\item[--] Determine the typical shapes of random Fishburn matrices,
which exhibit an unexpected change of limit laws from normal to
Poisson when the smallest nonzero entry is $2$; see Theorem \ref{T:phase} and Section \ref{S:jelinek}.

\end{itemize}

Our approach is best illustrated through the prototypical (rational) 
sequence 
\begin{align}\label{E:s1}
	a_n
	:= [z^n]\sum_{k\ge0}
	\prod_{1\le j\le k}\lpa{e^{jz}-1}
	=\frac{(-1)^n}2[z^n]R\lpa{e^{-z}},
\end{align}
where $R$ is defined in \eqref{E:ramanujan}, for which we will show 
\emph{inter alia} that
\begin{align}\label{E:an}
	a_n
	= c\rho^{n}n^{n+\frac12}\Lpa{1+
	\frac{\nu_1}{n}+\frac{\nu_2}{n^2}
	+O\lpa{n^{-3}}},
	\with 
	(c,\rho) := \lpa{\tfrac{12}{\pi^{3/2}},
	\tfrac{12}{e\pi^2}},
\end{align}
where $\nu_1 = \frac{24-\pi^2}{288}$ and $\nu_2 = \frac12\nu_1^2$;
see \eqref{E:znAz-ae} for an asymptotic expansion. Here the integer
sequence $\{a_nn!\}$ corresponds to
\href{https://oeis.org/A158690}{A158690} in the OEIS.
\emph{Throughout this paper, we do not distinguish between ordinary
and exponential generating functions, and focus only on the large-$n$
asymptotics of the coefficients, so whether the sequence is integer
or not is immaterial for our purposes.}

Once the asymptotics \eqref{E:an} is available, we extend our 
approach to the sequences of the form
\begin{align}\label{E:gen-form}
	[z^n]\sum_{k\ge0}d(z)^{k+\omega_0}
	\prod_{1\le j\le k}\left(e(z)^{j+\omega}-1\right)^\alpha,
\end{align}
for $\alpha\in\mathbb{Z}^+$ and $\omega_0,\omega\in\mathbb{C}$. Here 
the generating functions $d(z)$ and $e(z)$ satisfy $d(0)>0, e(0)=1$ 
and $e'(0)\ne0$. 

Our result (Theorem~\ref{T:gen}) for the general form
\eqref{E:gen-form} will not only be applied to derive the large-$n$
asymptotics of many sequences in the literature and the OEIS (see 
Section~\ref{S:A}), but also be sufficient to re-derive (\ref{E:zag}) 
because of the identity (in the sense of formal power series) due to 
Andrews and Jel\'{i}nek \cite{Andrews2014}
\begin{align*}
    \sum_{k\ge0}\prod_{1\le j\le k}\lpa{1-(1-z)^j}
    =\sum_{k\ge0}(1-z)^{-k-1}\prod_{1\le j\le k} 
    \left((1-z)^{-j}-1\right)^2.
\end{align*}
This and other examples of similar types are collected in
Section~\ref{S:A}. 

In addition to its usefulness in univariate asymptotics, our
formulation \eqref{E:gen-form} is also effective in dealing with the
limiting distributions of various statistics (bivariate asymptotics)
on random Fishburn matrices with or without restriction on their
entries, which describe the typical shape of random Fishburn
matrices.

More precisely, we consider upper-triangular matrices whose entries
belong to $\Lambda$, a multiset of nonnegative integers with the
generating function $\Lambda(z) = 1+\lambda_1z +\lambda_2z^2+\cdots$.
We then define two classes of matrices: (i)
$\Lambda$\emph{-row-Fishburn} ones without zero row, and (ii)
$\Lambda$\emph{-Fishburn} ones without zero row or zero column. The
statistics examined and their limit laws are summarized in
Table~\ref{T:laws}, where we assume a uniform distribution on the set
of all possible such matrices with the same entry-sum.

\begin{table}[!ht]
\renewcommand*{\arraystretch}{1.7}	
\begin{center}
	\begin{tabular}{c|cc}\hline
		\multicolumn{1}{c}{}
        $\lambda_1>0$\qquad
		& Random $\Lambda$-row-Fishburn matrices 
        & Random $\Lambda$-Fishburn matrices \\ \hline
		First row sum & Zero-Truncated-Poisson$(\log 2)$ &
		Normal$(\log n,\log n)$  \\
		Diagonal sum & Normal$(\log n,\log n)$ 
		& Normal$(2\log n,2\log n)$\\
		$\frac12\lpa{n-\text{\#(1s)}}$ & 
		$\begin{cases}
			\mathrm{Poisson}\lpa{\frac{\lambda_2\pi^2}
			{12\lambda_1^2}}, & \text{if }\lambda_2>0\\
			\mathrm{degenerate}, & \text{if }\lambda_2=0
		\end{cases}$
		 & 
		$\begin{cases}
			\text{Poisson}\lpa{\frac{\lambda_2\pi^2}
			{6\lambda_1^2}}, & \text{if }\lambda_2>0\\
			\text{degenerate}, & \text{if }\lambda_2=0
		\end{cases}$
		 \\ \hline
	\end{tabular}
\end{center}
\medskip
\caption{The various asymptotic distributions of the three statistics
in large random $\Lambda$-row-Fishburn and $\Lambda$-Fishburn
matrices with entries belonging to a given multiset of nonnegative
integers $\Lambda$ (containing $0$ exactly once and $1$ at least
once). Here $n$ is the sum of all entries in the matrix.}
\label{T:laws}	
\end{table}

In particular, when $\Lambda$ is the set of nonnegative integers, the
first row-size in random Fishburn matrices also arises in many
different contexts under different guises, the first being in the
form of leftmost chord in regular linearized chord diagrams
\cite{Stoimenow1998}; see Section~\ref{S:sfm} for more information.
Our limit results thus have many different interpretations and
implications.

The proof of these limit laws requires the full power of our setting  
\eqref{E:gen-form} where some parameters or coefficients are 
themselves complex variables, as well as the Quasi-Powers Framework
(see \cite{Flajolet2009, Hwang1994, Hwang1998}), which is a simple 
synthetic scheme for deriving asymptotic normality and some of its 
quantitative refinements. 

From Table~\ref{T:laws} we see that in a typical random
$\Lambda$-Fishburn matrix (when all matrices of the same entry-sum
are equally likely), entries equal to $1$ are ubiquitous, those to
$2$ appear like a Poisson distribution, and the rest is
asymptotically negligible. Thus such random matrices have little
variation as far as the distribution of entries is concerned. In other
words, a random $\Lambda$-Fishburn matrix is asymptotically close to
its primitive counterpart in which only $0$ and $1$ are allowed as
entries. Regarding a Fishburn matrix of size $n$ as an integer
partition of $n$ arranged on upper-triangular square matrices without
zero row or column, we see that the number of $1$s in random Fishburn
matrices is very different from the number of $1$s in random integer
partitions, which has an exponential distribution in the limit.



What happens if we drop the omnipresent entry $1$, assuming that all
$\Lambda$-Fishburn matrices (of the same size) whose smallest nonzero
entries are $2$ are equally likely? The resulting random matrices
turn out to be more interesting, exhibiting less expected behaviors.
More precisely, we extend further our study in
Section~\ref{S:jelinek} to the situation when $\lambda_1=0$ and 
$\lambda_2>0$ in $\Lambda$-Fishburn matrices, which has a very
similar analytic context to the self-dual (or persymmetric) Fishburn
matrices when $\lambda_1>0$; the latter was considered in
\cite{Jelinek2012} for the cases when $\Lambda=\{0,1\}$ and $\Lambda =
\mathbb{Z}_{\ge0}$. We adopt the same framework \eqref{E:gen-form}
and address the asymptotics when $\lambda_1=0$ and $\lambda_2>0$.
Such a formulation is, as in the case of $\lambda_1>0$, not only
useful for the asymptotic enumeration of matrices of large size, but
also practical in characterizing the finer stochastic behaviors of
the random matrices, whether they are Fishburn with $\lambda_1=0$ and
$\lambda_2>0$ or self-dual Fishburn with $\lambda_1>0$.

\begin{table}[!ht]
\renewcommand*{\arraystretch}{1.7}	
\begin{center}
	\begin{tabular}{c|cc}\hline
        \multicolumn{1}{c}{}
		& \makecell{Random $\Lambda$-Fishburn matrices
        \\ with $\lambda_{2i-1}=0$ for $1\le i\le m$
        \\ and $\lambda_2,\lambda_{2m+1}>0$}
        & \makecell{Random self-dual $\Lambda$-Fishburn \\
        matrices with $1$s ($\lambda_1>0$)} \\ \hline
		First row sum & Normal$(\log n,\log n)$ &
		Normal$(\log n,\log n)$  \\
		Diagonal sum & Normal$(2\log n,2\log n)$ 
		& $2\,\cdot$ Normal$(\log n,\log n)$\\
        \makecell{\# smallest \\ nonzero entries}
		& $\begin{cases}
            \frac12n-\frac32\cdot\text{Normal}(\tau\sqrt{n},
            \tau\sqrt{n}),\\
            \qquad \text{if }m=1\\
            \frac12n-2\cdot
            \text{Poisson}(\frac{\lambda_4\pi^2}{6\lambda_2^2}),\\
            \qquad \text{ if }m\ge2, \lambda_4>0, n \text{ even}\\
            \frac12(n-2m-1)-2\cdot
            \text{Poisson}(\frac{\lambda_4\pi^2}{6\lambda_2^2}),\\
            \qquad \text{ if }m\ge2, \lambda_4>0, n \text{ odd}\\
			\text{degenerate, if }\lambda_3=\lambda_4=0
		\end{cases}$
		 & 
		$\begin{cases}
			n-2\cdot\text{Poisson}\lpa{\frac{\lambda_2}
			{\lambda_1}\log 2}\\
            \quad\; *\;4\cdot\text{Poisson}\lpa{\frac{\lambda_2\pi^2}
            {12\lambda_1^2}},\\
            \qquad \text{if }\lambda_2>0\\
			\text{degenerate, if }\lambda_2=0
		\end{cases}$
		 \\ \hline
	\end{tabular}
\end{center}
\medskip
\caption{The asymptotic distributions of the three statistics in large
random $\Lambda$-Fishburn and self-dual $\Lambda$-Fishburn matrices
with entries belonging to a given multiset of nonnegative integers
$\Lambda$ (containing $0$ exactly once and with or without $1$s). 
Here $n$ is the sum of all entries in the matrix, and $\tau := 
\frac{\lambda_3\pi}{2\sqrt{3}\lambda_2^{3/2}}$. The symbol $X*Y$ 
stands for the convolution of two distributions.}
\label{T:laws-2}	
\end{table}

While the logarithmic behaviors in the first row sum and the diagonal
sum are similar as in Table~\ref{T:laws}, the limit laws of the
occurrences of the smallest nonzero entries seem less predicted,
notably in the case when $2$ is the smallest nonzero entry. Roughly, 
the periodicity resulted from the prevalent factor $2$ in the class 
of $\Lambda$-Fishburn matrices without using $1$ as entries does 
change drastically the behavior from being bounded Poisson to normal 
with mean and variance both asymptotic to $\tau\sqrt{n}$ (or indeed a
Poisson distribution with unbounded mean $\tau\sqrt{n}$; see
Section~\ref{S:no-1}).

Our formulation and results include as a special case the asymptotic
approximation to self-dual Fishburn numbers (\ref{E:selfdualfish}),
confirming a conjecture in \cite{Jelinek2012}; see 
Sections~\ref{S:gf} and \ref{S:jelinek}.

This paper is structured as follows. In the next section, we outline
the background on the Fishburn matrices, and then derive the
generating functions that will be analyzed in later sections. Then we
describe the saddle-point method in detail in Section~\ref{S:basis}
which will then be used and modified throughout this paper, with the
finer asymptotic expansions briefly discussed in Section~\ref{S:ae}.
The general framework \eqref{E:gen-form} is examined in detail in
Section~\ref{S:gf} by extending the saddle-point analysis of
Section~\ref{S:basis}. Asymptotics of restricted Fishburn matrices
as well as other univariate examples are collected and discussed in
Section~\ref{S:A}. We then turn to bivariate asymptotics in
Section~\ref{S:B} and study the asymptotic distributions of the
statistics on random Fishburn matrices as those given in
Table~\ref{T:laws}. The extension of \eqref{E:gen-form} to the case
when $e_1=0$ and $e_2>0$ is examined in Section~\ref{S:jelinek},
together with univariate and bivariate applications (as shown in 
Table~\ref{T:laws-2}). We then conclude this paper in 
Section~\ref{S:conclusions} with some perspectives on how our 
approach may be further extended to other frameworks.

\textbf{Notations.} As mentioned at the beginning of this section, 
$(c,\rho)$ is used generically and will always be locally defined. 
Other generic and mostly local symbols include $c_i$, $c(\cdot)$, 
$\ve$, $f$, and $a_n$; their values will be specified whenever 
ambiguities may occur.

\section{Fishburn matrices and related combinatorial objects}
\label{S:comb}

We describe Fishburn matrices in this section, together with some of
their variants and generalizations. We also derive the bivariate
generating functions for some statistics that will be examined in more
detail in later sections.

Throughout this paper, the \emph{size} of a matrix is defined to be
the sum of all its entries. Similarly, we write the size of a row or
a column or the diagonal to represent their respective sum.

\begin{definition}[Fishburn matrix]
A Fishburn matrix is an upper-triangular square one with nonnegative
integer entries such that no row or no column consists solely of 
zeros.
\end{definition}
For example, all $15$ Fishburn matrices of size $4$ are depicted in 
Figure~\ref{F:fishburn4}.
\begin{figure}[!h]
	\begin{center}
		$$\begin{pmatrix}[0.8]
		4\\
		\end{pmatrix}
		\begin{pmatrix}[0.8]
		1 \, 2 \\
		0 \, 1
		\end{pmatrix}\begin{pmatrix}[0.8]
		2 \, 1 \\
		0 \, 1
		\end{pmatrix}\begin{pmatrix}[0.8]
		1 \, 1 \\
		0 \, 2
		\end{pmatrix}\begin{pmatrix}[0.8]
		2 \, 0 \\
		0 \, 2
		\end{pmatrix}\begin{pmatrix}[0.8]
		3 \, 0 \\
		0 \, 1
		\end{pmatrix}\begin{pmatrix}[0.8]
		1 \, 0 \\
		0 \, 3
		\end{pmatrix}$$
		$$\begin{pmatrix}[0.8]
		1 \, 1 \, 0\\
		0 \, 1 \, 0\\
		0 \, 0 \, 1
		\end{pmatrix}\begin{pmatrix}[0.8]
		1 \, 0 \, 1\\
		0 \, 1 \, 0\\
		0 \, 0 \, 1
		\end{pmatrix}\begin{pmatrix}[0.8]
		1 \, 0 \, 0\\
		0 \, 1 \, 1\\
		0 \, 0 \, 1
		\end{pmatrix}\begin{pmatrix}[0.8]
		2 \, 0 \, 0\\
		0 \, 1 \, 0\\
		0 \, 0 \, 1
		\end{pmatrix}\begin{pmatrix}[0.8]
		1 \, 0 \, 0\\
		0 \, 2 \, 0\\
		0 \, 0 \, 1
		\end{pmatrix}\begin{pmatrix}[0.8]
		1 \, 0 \, 0\\
		0 \, 1 \, 0\\
		0 \, 0 \, 2
		\end{pmatrix}\begin{pmatrix}[0.8]
		1 \, 1 \, 0\\
		0 \, 0 \, 1\\
		0 \, 0 \, 1
		\end{pmatrix}\begin{pmatrix}[0.8]
		1 \, 0 \, 0 \, 0\\
		0 \, 1 \, 0 \, 0\\
		0 \, 0 \, 1 \, 0\\
		0 \, 0 \, 0 \, 1
		\end{pmatrix}$$
	\end{center}   
	\caption{All $15$ Fishburn matrices of size $4$. The 
		occurrence of $1$ is seen to be predominant.}
	\label{F:fishburn4} 
\end{figure}

As a succinct representation tool for interval orders (see
\cite{Fishburn1985, Greenough1976}), Fishburn matrices (called
IO-matrices in \cite{Greenough1976}, characteristic matrices in
\cite{Fishburn1983, Fishburn1985}, and composition matrices in
\cite{Dukes2011a}) offer not
only algorithmic but also combinatorial advantages, and over the
years their study was largely enriched by the corresponding
developments in combinatorial enumeration and bijections, following
notably the paper by Bousquet-M\'elou-Claesson-Dukes-Kitaev
\cite{Bousquet-Melou2010}. In particular, the useful database OEIS
\cite{oeis2019} played a key role in linking the various structures
in different areas some of which will be briefly described later.

Closer to our interest here, the enumeration of Fishburn matrices of
a given \emph{dimension} was already investigated in the early papers
\cite{Andresen1976, Greenough1976}, and recursive algorithms were
later proposed for computing matrices of a given size (see 
e.g.~\cite{Haxell1987, Stoimenow1998}), culminating in the definitive
work of Zagier \cite{Zagier2001}, where, through the proper use of
generating functions, effective asymptotic approximations
\eqref{E:zag} for Fishburn matrices of large size are derived.

\subsection{Fishburn matrices and their variants}

Recall that the Fishburn numbers 
(\href{https://oeis.org/A022493}{A022493}) count Fishburn 
matrices of a given size and can be computed by the generating 
function
\begin{align*}
	\sum_{k\ge0}\prod_{1\le j\le k}
	\lpa{1-(1-z)^j}=1+z+2z^2+5z^3+15z^4+53z^5+217z^6+\cdots.
\end{align*}

From a combinatorial viewpoint, the Fishburn numbers also enumerate 
several seemingly unrelated structures some of which are listed as
follows; see \cite{Bousquet-Melou2010, Claesson2011, Dukes2011,
Dukes2010, Fishburn1970, Fishburn1985, Fu2020, Jelinek2012,
Kitaev2011, Levande2013, Yan2011} for the bijective and algebraic
proofs of these equinumerosity.

\begin{itemize}

\item \emph{Ascent sequences of length $n$}, which are sequences of
nonnegative integers $(x_1,x_2,\ldots,x_n)$ such that for each $i$,
$0\le x_i\le 1+|\{j:1\le j\le i-2, \, x_j<x_{j+1}\}|$.

\item \emph{$({\bf2-1})$-avoiding inversion sequences of length $n$}:
these are sequences $x=(x_1,x_2,\ldots,x_n)$ such that $0\le x_i<i$
and there exists \emph{no} $i<j$ such that $x_i=x_j+1$.

\item \emph{$(2|3\bar{1})$-avoiding permutations of $n$ elements},
which are permutations $\pi$ without subsequence
$\pi_i\pi_{i+1}\pi_j$ such that $\pi_i-1=\pi_j$ and $\pi_i<\pi_{i+1}$.

\item \emph{$(\bf 2+\bf 2)$-free posets of $n$ elements}: these are
posets $(P,\prec)$ with \emph{interval representations}, namely,
for each $x\in P$, a real closed interval $[\ell_x,r_x]$ is
associated to $x$ such that $x\prec y$ in $P$ exactly when
$r_x<\ell_y$.

\item \emph{Stoimenow matchings of length $2n$}: A matching of the
set $[2n]=\{1,2,\ldots,2n\}$ is a partition of $[2n]$ into subsets
(called \emph{arcs}) of size exactly two. A Stoimenow matching is one
without nested pair of arcs such that either the openers or the
closers are next to each other.

\item \emph{Regular linearized chord diagrams of length $2n$}: A
regular linear chord diagram is a fixed-point free involution $\tau$
on the set $[2n]$ such that if $[i,i+1]\subset [\tau(i+1),\tau(i)]$
whenever $\tau(i+1)<\tau(i)$. 

\end{itemize}

Two variants of Fishburn matrices, \emph{row-Fishburn matrices}
and \emph{self-dual Fishburn matrices} were studied by Jel\'{i}nek
\cite{Jelinek2012} during his study on refined enumeration of
self-dual interval orders. \emph{Row-Fishburn matrices} are
upper-triangular ones with non-negative integer entries such that
no row is composed solely of zeros. The corresponding 
generating function satisfies
\begin{align}\label{E:rowfish}
	\sum_{k\ge0}\prod_{1\le j\le k}
	\left((1-z)^{-j}-1\right)
	=1+z+3z^2+12z^3+61z^4+380z^5+2815z^6+\cdots,
\end{align}
where the coefficient of $z^n$ equals the number of row-Fishburn 
matrices of size $n$.

Furthermore, a matrix is \emph{primitive} if all entries are either
$0$ or $1$. Substituting $z$ by $\frac z{1+z}$ leads to the
generating function for the \emph{primitive row-Fishburn number}
(\href{https://oeis.org/A179525}{A179525})
\begin{align}\label{E:prirowfish}
	\sum_{k\ge0}\prod_{1\le j\le k}\left((1+z)^j-1\right)
	=1+z+2\,z^2+7\,z^3+33\,z^4+197\,z^5+1419\,z^6+\cdots.
\end{align}
Reversely, (\ref{E:rowfish}) can be obtained from
\eqref{E:prirowfish} by substituting $z$ with $\frac{z}{1-z}$. 

On the other hand, a Fishburn matrix is \emph{self-dual} if it is
persymmetric, or symmetric with respect to the northeast-southwest
diagonal. The generating function of primitive self-dual Fishburn matrices counted by the size is (see also \cite{Jelinek2012})
\[\sum_{k\ge0}(1+z)^{k+1}
\prod_{1\le j\le k}\left((1+z^2)^{j}-1\right)
=1+z+z^2+2\,z^3+3\,z^4+6\,z^5+13\,z^6+\cdots,
\]
and the one of all self-dual Fishburn matrices is 
\[
    \sum_{k\ge0}(1-z)^{-k-1}
    \prod_{1\le j\le k}\lpa{(1-z^2)^{-j}-1}
    = 1+z+2\,z^2+3\,z^3+7\,z^4+13\,z^5+33\,z^6+\cdots,
\]
so that $7$ out of the $15$ Fishburn matrices of size 4 are self-dual; see Figure~\ref{F:fishburn4}. 

\subsection{Fishburn matrices with entry restrictions}

We now extend the matrices by allowing more flexible entries. 
Let $\Lambda$ be a multiset of nonnegative integers with the 
generating function
\begin{align}\label{E:Lz}
	\Lambda(z) := 1+\sum_{\lambda\in\Lambda} z^\lambda
	= 1+\lambda_1z+\lambda_2z^2+\cdots.
\end{align}
Assume \emph{throughout this paper except in Sections~\ref{S:8-1}, 
\ref{S:8-3} and \ref{S:8-4} that $\lambda_1>0$}, so that 
$\{0,1\}\in\Lambda$. 

\begin{definition}[$\Lambda$-Fishburn matrix] 
An upper-triangular matrix is called \emph{$\Lambda$-Fishburn} if
every row and column has non-zero size, and all entries lie in the
set $\Lambda$.
\end{definition}

\begin{definition}[$\Lambda$-row-Fishburn matrix] 
An upper-triangular matrix is called \emph{$\Lambda$-row-Fishburn} if
all entries lie in the set $\Lambda$ without zero row. 
\end{definition}

\begin{proposition} \label{P:gf-f}
The number of $\Lambda$-row-Fishburn matrices of size $n$ is given
by
\begin{align}\label{E:gf-lrf}
	[z^n]\sum_{k\ge0}\prod_{1\le j\le k}\lpa{\Lambda(z)^j-1},
\end{align}
and that of $\Lambda$-Fishburn matrices by 
\begin{align}\label{E:gf-lf}
	[z^n]\sum_{k\ge0}\prod_{1\le j\le k}
	\lpa{1-\Lambda(z)^{-j}}
	&= [z^n]\sum_{k\ge0}\Lambda(z)^{k+1}
	\prod_{1\le j\le k}\lpa{\Lambda(z)^j-1}^2.
\end{align}
\end{proposition}
\begin{proof}
The first generating function \eqref{E:gf-lrf} follows from the 
definition of $\Lambda$-row-Fishburn matrices. For a 
$\Lambda$-row-Fishburn matrix of dimension $k$, and for any $j$, 
$1\le j\le k$, the generating function of the $(k-j+1)$-st row 
counted by the size (variable $z$) is $\Lambda(z)^j-1$. As a result, 
the generating function for $\Lambda$-row-Fishburn matrices of 
dimension $k$ is given by $\prod_{1\le j\le k}(\Lambda(z)^j-1)$. 
Summing over all $k$ leads to \eqref{E:gf-lrf}.

On the other hand, the generating function for primitive Fishburn 
matrices is given by (including the constant $1$ for the empty 
matrix; see \cite{Jelinek2012})
\begin{align*}
	\sum_{k\ge0}\prod_{1\le j\le k}\lpa{1-(1+z)^{-j}}.
\end{align*}
Substituting $1+z$ by $\Lambda(z)$ yields the generating function for 
$\Lambda$-Fishburn matrices in the general case.

The right-hand side of the identity \eqref{E:gf-lf} follows from the  
following $q$-identity due to Andrews and Jel\'{i}nek
\cite{Andrews2014} 
\begin{equation}\label{E:Andrews}
\begin{split}
	&\sum_{k\ge0}u^k\prod_{1\le j\le k}
	\left(1-\frac{1}{(1-s)(1-t)^{j-1}}\right)\\
	&\qquad=\sum_{k\ge0}(1-s)(1-t)^k
	\prod_{1\le j\le k}\lpa{\lpa{1-(1-s)(1-t)^{j-1})}
	\lpa{1-u(1-t)^j}},
\end{split}	
\end{equation}
by substituting $u=1$ and $s=t=1-\Lambda(z)$ on both sides.
\end{proof}

\subsection{Statistics on Fishburn matrices}
\label{S:sfm}

The study of statistics on the Fishburn structures traces back to the
work by Andresen and Kjeldsen \cite{Andresen1976} in the context of
transitively directed graphs (see also \cite{Jelinek2012}), where
they studied the numbers of primitive Fishburn matrices counted by
the dimension and by the size of the first row (with the notation
$\xi(n,k)$ in \cite{Andresen1976}).

Stoimenow \cite{Stoimenow1998} found a recursive formula for the
numbers of regular linearized chord diagram with a given length of
the leftmost chord. Subsequently, it was discovered
\cite{Bousquet-Melou2010, Fu2020, Jelinek2012, Kitaev2017,
Levande2013, Yan2011} that these numbers are equivalent to the
following ones:

\begin{itemize}

\item the sum of entries in the first row (or the last column) of
Fishburn matrices of size $n$;

\item the number of minimal (or maximal) elements in $(2+2)$-free
posets of size $n$;

\item the maximal entries (or right-to-left minimal entries, or the
number of zeros) in ascent sequences of length $n$;

\item the maximal entries in $({\bf2-1})$-avoiding inversion
sequences of length $n$;

\item the length of the initial run of openers in Stoimenow matchings
of length $[2n]$;

\item the length of the initial decreasing run in
$(2|3\bar{1})$-avoiding permutations of length $n$;

\item the number of left-to-right minima (or left-to-right maxima; or
right-to-left maxima) in $(2|3\bar{1})$-avoiding permutations of
length $n$.

\end{itemize}

These statistics are classified as Stirling statistics; see
\cite{Fu2020}. In parallel, the Eulerian statistics have also been
intensively studied but the corresponding limiting distributions have
remained open and will be addressed elsewhere.

\subsection{Bivariate generating functions for Fishburn
matrices with entry restrictions}

We study the asymptotic distributions of the following three
random variables on random $\Lambda$-Fishburn and
$\Lambda$-row-Fishburn matrices, assuming a uniform distribution on
the set of all size-$n$ matrices: the size of the first row, the size 
of the diagonal, and the number of occurrences of $1$.

The approach we use to characterize the corresponding limit laws
relies heavily on the corresponding bivariate generating functions
and our double saddle-point analysis. We derive the required 
generating functions in this subsection. We use the convention that 
$f(z,v)$ is the bivariate generating function for the quantity $X$ if 
$[z^nv^m]f(z,v)$ denotes the number of $\Lambda$-row-Fishburn 
matrices of size $n$ with $X=m$. 

\begin{proposition}[Statistics on $\Lambda$-row-Fishburn matrices] 
We have the following bivariate generating functions with $z$ marking 
the matrix size and $v$ marking respectively

\begin{enumerate}[(i)]

\item the size of the first row 
\begin{equation}\label{E:gf-fr}
	1+\sum_{k\ge0}\lpa{\Lambda(vz)^{k+1}-1}
	\prod_{1\le j\le k}\lpa{\Lambda(z)^j-1},
\end{equation}

\item the size of the (main) diagonal (or the last column)
\begin{equation}\label{E:gf-dg}
	\sum_{k\ge0}\prod_{1\le j\le k}
	\lpa{\Lambda(vz)\Lambda(z)^{j-1}-1}, \and 
\end{equation}
\item the number of $1$s
\begin{equation}\label{E:gf-1}
	\sum_{k\ge0}\prod_{1\le j\le k}
	\lpa{(\Lambda(z)+\lambda_1(v-1)z)^{j}-1}.
\end{equation}
\end{enumerate}
\end{proposition}
\begin{proof}
Given a $\Lambda$-row-Fishburn matrix of dimension $k+1$, the
generating function for the first row size (marked by $vz$) is
$\Lambda(vz)^{k+1}-1$, the remaining $k$ rows contributing
$\prod_{1\le j\le k}\lpa{\Lambda(z)^j-1}$, as in the proof of
\eqref{E:gf-lrf}. The proofs for the other two parameters are similar
and thus omitted.
\end{proof}

For $\Lambda$-Fishburn matrices, the proof is less straightforward
and we need a fine-tuned version of Jel\'\i nek's Theorem 2.1 in 
\cite{Jelinek2012} in order to enumerate both the first row and the 
diagonal sizes. 

Let $\mathcal{P}$ denote the set of primitive (binary) 
$\Lambda$-Fishburn matrices. Define first
\[
\begin{split}
    &G_k(t,u,v,w,x,y)\\
	&\qquad :=
    \sum_{(M_{i,j})_{k\times k}\in\mathcal{P}}
    t^{M_{k,k}} u^{\sum_{2\le j<k}M_{j,k}}
    v^{\sum_{2\le j<k}M_{j,j}}
    w^{\sum_{1<i<j<k}M_{i,j}}
    x^{M_{1,k}} y^{\sum_{1\le j<k}M_{1,j}},
\end{split}
\]
so that $t$ marks the lower-right corner (which is always $1$), $u$
the size of the last column except the two ends, $v$ the size of the
(main) diagonal except the two ends, $w$ the size of all interior
cells, $x$ the upper-right corner, and $y$ the size of the first row
except the upper-right corner.

\begin{lemma} The generating function $F(s,t,u,v,w,x,y)
    :=\sum_{k\ge2}G_{k}(t,u,v,w,x,y)s^k$ satisfies  
\begin{equation}\label{E:F-super-gf}
    \begin{split}
        &F(s,t,u,v,w,x,y)\\
        &\qquad =t\sum_{k\ge0}
        \frac{s^{k+2}y(1+x)(1+y)^k}{(1+u)(1+v)(1+w)^{k}-1}
        \prod_{0\le j\le k}\frac{(1+u)(1+v)(1+w)^j-1}
        {1+s\lpa{(1+u)(1+w)^j-1}}.
    \end{split}
\end{equation}    
\end{lemma}
\begin{proof} (Sketch)
By definition, it is clear that $G_1=x$ and $G_{k}(t,u,v,w,x,y)
=tG_{k}(1,u,v,w,x,y)$ for $k\ge 2$. By the recursive
construction used in \cite[Lemma 2.8]{Jelinek2012}, we derive the 
recurrence relation 
\begin{equation*} 
   \begin{split}
       &G_{k+1}(1,u,v,w,x,y)\\
       &\qquad=G_k(u+v+uv,u,v+w+vw,w,x+y+xy,y)
       -vG_k(1,u,v,w,x,y)\\
       &\qquad=(u+v+uv)G_k(1,u,v+w+vw,w,x+y+xy,y)
       -vG_k(1,u,v,w,x,y),
   \end{split} 
\end{equation*}
for $k\ge 2$. Then from this and the iterative arguments used in 
\cite{Jelinek2012}, we deduce \eqref{E:F-super-gf}; see 
\cite{Jelinek2012} for details. 
\end{proof}

\begin{proposition}[Statistics on $\Lambda$-Fishburn matrices] 
\label{P:lfm-stat}
We have the following bivariate generating functions with $z$ marking 
the matrix size and $v$ marking respectively

\begin{enumerate}[(i)]

\item the size of the first row (or the last column)
\begin{equation} \label{E:gf-f-fr} 
\begin{split}
    &\sum_{k\ge 0}\prod_{1\le j\le k}
    \lpa{1-\Lambda(vz)^{-1}\Lambda(z)^{1-j}}\\
	&\qquad=\Lambda(vz)\sum_{k\ge 0}\Lambda(z)^k
	\prod_{1\le j\le k}\lpa{\lpa{\Lambda(vz)\Lambda(z)^{j-1}-1}
	\lpa{\Lambda(z)^j-1}},
\end{split}\end{equation}

\item the size of the (main) diagonal
\begin{equation} \label{E:gf-f-dg} 
\begin{split}
	&1+\Lambda(vz)+(\Lambda(vz)-1)^2
	\sum_{k\ge 0}\prod_{1\le j\le k}
	\lpa{\Lambda(vz)-\Lambda(z)^{-j}}\\
	&\qquad =\Lambda(vz)
	\sum_{k\ge 0}\Lambda(z)^k
	\prod_{1\le j\le k}\lpa{\Lambda(vz)\Lambda(z)^{j-1}-1}^2, 
    \and 
\end{split}
\end{equation}
\item the number of $1$s
\begin{equation} \label{E:gf-f-1} 
\begin{split}
	&\sum_{k\ge0}\prod_{1\le j\le k}
	\lpa{1-(\Lambda(z)+\lambda_1(v-1)z)^{-j}}\\
	&\qquad=\sum_{k\ge0}(\Lambda(z)+\lambda_1(v-1)z)^{k+1}
	\prod_{1\le j\le k}
	\lpa{(\Lambda(z)+\lambda_1(v-1)z)^j-1}^2.
\end{split}\end{equation}
\end{enumerate}
\end{proposition}
\begin{proof}
	
\begin{enumerate}[(i)]
		
\item It is known that the generating function for the size of the
first row (marked by $v$) in primitive Fishburn matrices is given by (see \cite{Fu2020, Kitaev2011}) 
\begin{align*}
	\sum_{k\ge 0}\prod_{1\le j\le k}
	(1-(1+vz)^{-1}(1+z)^{1-j}).
\end{align*}
Substituting $1+vz$ by $\Lambda(vz)$ and $1+z$ by $\Lambda(z)$ gives
the left-hand side of (\ref{E:gf-f-fr}), which, equals, by
substituting $u=1$, $s=1-\Lambda(vz)$ and $t=1-\Lambda(z)$ in
(\ref{E:Andrews}), the same generating function on the right-hand
side of \eqref{E:gf-f-fr}. Alternatively, one can derive
(\ref{E:gf-f-fr}) by using \eqref{E:F-super-gf}, Andrew-Jel\'\i nek
identity \eqref{E:Andrews}, the identity \cite[Eq.\ (1)]{Andrews2014}
and substitutions.

\item For the size of the diagonal, we have, again, by 
\eqref{E:F-super-gf}, the generating function 
\begin{align*}
    1+vz+F(1,v^2z,z,vz,z,z,z)
	=1+vz+(vz)^2\sum_{k\ge0}
    \prod_{1\le j\le k}\lpa{1+vz-(1+z)^{-j}}.
\end{align*}
The same substitutions $1+vz\mapsto \Lambda(vz)$ and 
$1+z\mapsto\Lambda(z)$ give the left-hand side of \eqref{E:gf-f-dg}. 
Applying now \eqref{E:Andrews} with $u=1+vz$, $s=1-(1+vz)(1+z)$ and 
$t=-z$, and then using the same substitutions, we obtain the 
right-hand side of \eqref{E:gf-f-dg}. 

\item The generating functions \eqref{E:gf-f-1} for the number of $1$s
follow from substituting $\Lambda(z)$ by $\Lambda(z)
+\lambda_1(v-1)z$ in \eqref{E:gf-lf}.

\end{enumerate}
\end{proof}

\section{Asymptotics of the prototype sequence \href{https://oeis.org/A158690}{A158690}}\label{S:basis}

Consider the sequence $a_n := [z^n] A(z)$, where 
\begin{align}\label{A158690}
    A(z):=\sum_{k\ge0}A_k(z), \with 
    A_k(z):=\prod_{1\le j\le k}\lpa{e^{jz}-1},
\end{align}
which is used as the running and prototypical example of our analytic 
approach. The sequence $\{n!a_n\}_{n\ge0}$ equals 
\href{https://oeis.org/A158690}{A158690} and can be generated (in 
addition to \eqref{A158690} by many different forms (see 
\cite{Bringmann2014,Andrews1986}), showing partly the diversity and 
structural richness of the sequence 
\begin{align*}
    A(z) &= \sum_{k\ge0}
    \prod_{1\le j\le k}\lpa{1-e^{-(2j-1)z}}\\
    &= \sum_{k\ge0}e^{-(k+1)z}
    \prod_{1\le j\le k}\lpa{1-e^{-2jz}} \\
    &=\sum_{k\ge0}e^{(2k+1)z}
    \prod_{1\le j\le 2k}\lpa{e^{jz}-1}\\
    &=\frac12\llpa{1+\sum_{k\ge0} e^{(k+1)z}
	\prod_{1\le j\le k} \lpa{e^{jz}-1}}.
\end{align*}    
Among these series forms, we work on \eqref{A158690} because it 
is simpler and $A_k(z)$ contains only positive 
Taylor coefficients. 

\begin{theorem}\label{T:eg1}
Define $A(z)$ by (\ref{A158690}). Then as $n$ tends to infinity,
\begin{align}\label{E:A}
    a_n := [z^n]A(z)
    \simeq c\rho^{n} n^{n+\frac{1}{2}},
    \with  (c,\rho)
    := \lpa{\tfrac{12}{\pi^{3/2}},\tfrac{12}{e\pi^2}}.
\end{align}
\end{theorem}

Our approach consists in computing the asymptotics of $a_{n,k}$ by
the saddle-point method (see \cite{Flajolet2009}) for each $1\le
k<n$, and then summing $a_{n,k}$ over all $k$ (in turn involving
another application of saddle-point method); indeed, due to high
concentration near the maximum, only a small neighborhood of $k$ near
$\mu n$, $\mu := \frac{12}{\pi^2}\log 2\approx 0.84$, will contribute
to the dominant asymptotics \eqref{E:A}. Thus we are in the context
of a double saddle-point method.

More precisely, we begin with the expression
\begin{align*}
    a_n = \sum_{1\le k\le n}a_{n,k}
    = \sum_{1\le k\le n}\frac{r^{-n}}{2\pi i}
    \int_{-\pi}^{\pi} e^{-in\theta}A_k(re^{i\theta}) \dd \theta,
\end{align*}
and follow the procedures outlined below. 
\begin{itemize}

\item Find the positive solution pair $(k,r)$ of the equation 
$(\partial_kr^{-n}A_k(r),\partial_rr^{-n}A_k(r))=(0,0)$, so as to 
identify the terms $a_{n,k}$ that reach the maximum modulus for each 
fixed $n$; see Lemma~\ref{L:max-phi}. 

\item Once the range of $k\sim \mu n$ is identified, show, by simple 
saddle-point bound for Taylor coefficients, that the contribution to 
$a_n$ of $a_{n,k}$ from the range $|k-\mu n|\ge n^{\frac58}$ is 
asymptotically negligible; see Proposition~\ref{P:poly}. 

\item In the central range $|k-\mu n|\le n^{\frac 58}$, the 
integral $\int_{n^{-\frac38}\le |\theta|\le \pi}$ is asymptotically
negligible; see Proposition~\ref{P:central-o}. 

\item Then inside the ranges $|k-\mu n|\le n^{\frac58}$ and 
$\int_{|\theta|\le n^{-\frac38}}$, compute the asymptotic 
approximation \eqref{E:A} by local expansions and term-by-term 
integration; see Section~\ref{S:spm4}. 

\item These procedures can be applied to get refined expansions if 
desired. 
\end{itemize}

For all these purposes, it turns out that a precise asymptotic 
approximation to $\log A_k(r)$ will largely simplify the analysis. 
Since we will also need asymptotics of the derivatives of $\log 
A_k(r)$, we propose a complex-variable version so as to avoid 
repeated use of the Euler-Maclaurin formula. 

\subsection{Euler-Maclaurin formula and asymptotic expansions}

We apply the Euler-Maclaurin formula to approximate the various 
sums encountered in this paper, which for completeness is included as
follows.
\begin{lemma}[Euler-Maclaurin formula] Assume that $\varphi$ is 
$m$-times continuously differentiable over the interval $[a,b]$, 
$m\ge1$. Then
\begin{align}\label{E:emf}
    \begin{split}
        \sum_{j=a+1}^b \varphi(j)
        &= \int_a^b \varphi(t) \,\mathrm{d} t 
        +\frac{\varphi(b)-\varphi(a)}2
        +\sum_{\ell=1}^{\tr{m/2}}\frac{B_{2\ell}}{(2\ell)!}
        \lpa{\varphi^{(2\ell-1)}(b)-\varphi^{(2\ell-1)}(a)}  \\  
        &\qquad +\frac{(-1)^{m+1}}{m!}
        \int_a^b \varphi^{(m)}(t)
        B_m(\{t\})\mathrm{d} t,
    \end{split}
\end{align}
where $\{x\}$ denotes the fractional part of $x$, the $B_\ell$'s 
and the $B_n(t)$'s are Bernoulli numbers and polynomials, 
respectively. 
\end{lemma}
When $\varphi$ is infinitely differentiable (which is the case for 
all functions considered in this paper), we can push the expansion 
to any $m>0$ depending on the required error, making the error term 
under control.

The following expansion is crucial in our saddle-point analysis. 
Let 
\[
    L_k(z) := \log A_k(z) 
	= \sum_{1\le j\le k}\log\lpa{e^{jz}-1}.
\] 
\begin{proposition} For $k\to\infty$, we have 
\begin{equation}\label{E:Lkz}
\begin{split}
	L_k(z)
	&= k\log\lpa{e^{kz}-1}-\frac{I(kz)}{z}
	+\frac12\log\frac{2\pi\lpa{e^{kz}-1}}{z}
	+ \frac{z\lpa{e^{kz}+1}}
	{24\lpa{e^{kz}-1}}\\
	&\qquad+\sum_{2\le j<m}\frac{B_{2j}}{(2j)!}
	\cdot\frac{z^{2j-1}e^{-kz}E_{2j-2}
	\lpa{e^{-kz}}}{\lpa{1-e^{-kz}}^{2j-1}}
	+O\lpa{k^{1-2m}+|z|^{2m-1}},
\end{split}	
\end{equation}
uniformly for $k|z|\le 2\pi-\ve$ when $|\arg z|\le \pi-\ve$, where 
\begin{align*}
	I(z) := \int_0^z\frac{t}{1-e^{-t}}\,\dd t,
\end{align*}
and $E_n(x) = \sum_{0\le j<n}\Eulerian{n}{j}x^j$ denote the 
polynomials of Eulerian numbers.
\end{proposition}
\begin{proof}
For simplicity and for later use, we compute only the first few terms
by working out $m=2$, as the general form follows from the relation
\begin{align*}
    \partial_z^m \log\lpa{e^{xz}-1}
	= (-1)^{m-1}\frac{x^me^{-xz}E_{m-1}
	\lpa{e^{-xz}}}{\lpa{1-e^{-xz}}^m}
	\qquad(m\ge2);
\end{align*}
see \cite{Szekeres1953} for similar analysis.

Since $\log\lpa{e^{jz}-1}$ is undefined at $j=0$, we split the sum 
into two parts:
\begin{align*}
    L_k(z) =\log k!
	-\sum_{1\le j\le k} \log\frac{j}{e^{jz}-1}.
\end{align*}
By the Euler-Maclaurin formula \eqref{E:emf}, we find that
\begin{align*}
	\sum_{1\le j\le k}\log\frac{j}{e^{jz}-1}
	&=\int_{0}^k\log
	\frac{x}{e^{xz}-1}\dd x
	+\frac{1}{2}\log\frac{kz}{e^{kz}-1}\\
    &\quad \quad
	+\frac{1}{12}\Lpa{\frac{1}{k}
	+\frac{z}{2}-\frac{z}{1-e^{-kz}}}
    +O\lpa{k^{-2}+ |z|^2}.
\end{align*}
By an integration by parts, we see that  
\[
    \int_{0}^k\log\frac{x}{e^{xz}-1}\dd x
	= k\log\frac{k}{e^{kz}-1}-k
	+\frac{I(kz)}{z}. 
\]
The first few terms of \eqref{E:Lkz} then follow from this and 
Stirling's formula for $\log k!$.

For the error term, by \eqref{E:emf} with $m=2$ and $B_2(x) 
= x^2-x+\frac16$, we have
\begin{align*}
    R_2 
    &:= \int_0^k \Lpa{\frac{z^2e^{xz}}
    {(e^{xz}-1)^2}-\frac1{x^2}}
    \lpa{\{x\}^2-\{x\}+\tfrac16}\dd x.
\end{align*}
Since $k|z|\le 1$, we see that 
\[    
    R_2 
    = O\Lpa{\int_0^{1/|z|}|z|^2\dd x}
    =O(|z|).
\]
On the other hand, if $1\le k|z|\le 2\pi-\ve$, then
\begin{align*}
    R_2 = O\left(|z|+ \int_{1/|z|}^k
    \left(\frac{|z|^2e^{\Re(xz)}}
    {|e^{xz}-1|^2}+\frac1{x^2}\right)\dd x\right)
    = O\lpa{|z|+k^{-1}},
\end{align*}
as required, where $\Re(z)$ denotes the real part of $z$. 
This proves \eqref{E:Lkz}.
\end{proof}	

Note that 
\[
    I(z):= \int_0^{z}
    \frac{t}{1-e^{-t}}\dd t
	= \frac{z^2}2+\text{dilog}\lpa{e^{-z}},
\]
where $\text{dilog}(1-z):=\sum_{k\ge1}\frac{z^k}{k^2}$ denotes the 
dilogarithm function. Also $n![z^n]\text{dilog}\lpa{e^{-z}} = 
B_{n-1}$, the Bernoulli numbers. 

The main reason of stating this complex-variable version for $L_k(z)$
is that termwise differentiation with respect to $z$ is allowed by
analyticity in compact domain (or Cauchy's integral formula for
derivatives), leading to an asymptotic expansion for all higher
derivatives of $L_k(z)$; see, e.g., \cite{Olver1974,Wong2001}. In
this way, we obtain, for example, the following approximations, which
will be needed below.

\begin{corollary} Uniformly as $k\to\infty$ and $k|z|\le 2\pi-\ve$ 
    when $|\arg(z)|\le\pi-\ve$, 
\begin{equation}\label{E:Lk-1}
\begin{split}
    zL_k'(z)
    &= \sum_{1\le j\le k}\frac{jz}{1-e^{-jz}}
    = \frac{I(kz)}z+\frac{kz-1+e^{-kz}}
    {2\lpa{1-e^{-kz}}}+O\lpa{k^{-1}+|z|}.
\end{split}
\end{equation}
\end{corollary}

\begin{corollary} Let $m\ge2$. Then 
\begin{equation}\label{E:Lk-m}
\begin{split}
    z^mL_k^{(m)}(z)
    &= (-1)^{m-1}z^m\sum_{1\le j\le k}
	\frac{j^m e^{-jz}E_{m-1}
    \lpa{e^{-jz}}}{\lpa{1-e^{-jz}}^m}\\
    &= z^m\partial_z^{m-1}\left(\frac{I(kz)}{z^2} 
    +\frac{kz-1+e^{-kz}}
    {2z\lpa{1-e^{-kz}}}\right)+O\lpa{k^{-1}+|z|},
\end{split}
\end{equation}
uniformly as $k\to\infty$, $k|z|\le2\pi-\ve$ when $|\arg 
z|\le\pi-\ve$. 
\end{corollary}
In particular, we see that each $r^mL_k^{(m)}(r)$ is asymptotically 
of linear order when $kr=O(1)$. 

\subsection{Saddle-point method. I: Identifying the central range}

A very simple uniform estimate for $a_{n,k}$ is readily obtained 
by the saddle-point bound for positive Taylor coefficients (see 
\cite[Sec.~VIII.2]{Flajolet2009}).

\begin{lemma} For $1\le k<n$
\begin{align}\label{E:ank-sd1}
	a_{n,k} \le r^{-n} A_k(r),
\end{align}
where $r>0$ is chosen to be the saddle-point, namely, the unique 
positive solution of the equation
\begin{align}\label{E:saddle0}
    rL_k'(r) = \frac{rA_k'(r)}{A_k(r)}
    =\sum_{1\le j\le k}\frac{jr}{1-e^{-jr}} = n.
\end{align}
\end{lemma} 
Such an $r$ obviously exists for $n\ge1$ and $1\le k< n$ because 
$x/(1-e^{-x})\ge1$ is monotonically increasing with $x\ge0$. Also 
$r\to\infty$ when $k=o(n)$ and $r\to0$ when $k\to n$. In particular, 
$r=0$ when $k=n$. 

The simple bound \eqref{E:ank-sd1} is sufficient to give not only the
factorial term $n^n$ but also the right exponential one
$\lpa{\frac{12}{e\pi^2}}^n$ in \eqref{E:A}; see Lemma~\ref{L:max-phi}.

\begin{lemma} \label{L:max-phi} 
For $1\le k=qn<n$
\begin{align}\label{E:ank-sd2}
    a_{n,k} = O\lpa{n^{n+\frac12} e^{\phi(q,\varrho)n}},
    \with \phi(q,\varrho) := -\log\varrho + q \log(e^{q\varrho}-1)-1,
\end{align}    
subject to the condition $I(q\varrho) = \varrho$. The maximum value 
of $\phi(q,\varrho)$ for $q\in[0,1]$ is reached when
\begin{align}\label{E:mu-xi}
    (q,\varrho) 
    = (\mu,\xi) 
    := \Lpa{\frac{12}{\pi^2}\log 2, \frac{\pi^2}{12}}, 
    \with e^{\phi(\mu,\xi)} = \frac{12}{e\pi^2}.
\end{align}
\end{lemma}
\begin{proof}
We begin with the first-order approximation to $rL_k'(r)$ already 
derived in \eqref{E:Lk-1} with the saddle-point $|z|=r=\frac\varrho 
n$ and $k=qn$,
\[
    \frac{\varrho} nL_k'\Lpa{\frac{\varrho}{n}}
    \sim \frac n{\varrho}\,I(q \varrho),
\]
which leads to the approximate saddle-point equation $I(q\varrho) = \varrho$. Furthermore, we have, by \eqref{E:Lkz}, 
\begin{align*}
    \log\lpa{r^{-n}A_k(r)}
    &= n\log n +n \phi(q,\varrho) + \tfrac12\log n +O(1),
\end{align*}
where $ \phi(q,\varrho) := -\log\varrho + q \log(e^{q\varrho}-1)-1$,
in view of $I(q\varrho) = \varrho$. We next look for the pair of
$(q,\varrho)$ such that the maximum value of $\phi(q,\varrho)$ is
reached.

The only positive solution pair of the equations $(\partial_q \phi, 
\partial_\varrho \phi)=(0,0)$, or, equivalently,  
\[
    \log\lpa{e^{q\varrho}-1} = 0, \and 
	I(q\varrho))= \varrho,
\]
is given by the pair \eqref{E:mu-xi}, which is one of the sources of
the ubiquitous factor $\frac{\pi^2}{6}=\zeta(2)$ in this paper. It
remains to prove $\phi(q,\varrho)$ is maximal for $q \in [0,1]$ only
when (\ref{E:mu-xi}) occurs. Now, by viewing $w=w(q)$ as a function
of $q$, we see that, when $(q,w)$ satisfies the condition $I(qw)=w$,
\[
    \partial^2_q \phi(q,w)
	= \frac{w}{1-q^2w-e^{-qw}}.
\] 
We now prove that $\partial_q^2\phi(q,w)<0$ for all pairs $(q,w)$ 
such that $I(qw)=w$. First, the function $t\mapsto 
\frac{t}{1-e^{-t}}$ is motononically increasing for $t\ge0$; then, 
with $w\ne0$, 
\[
    w = \int_0^{qw}\frac{t}{1-e^{-t}}\,\dd t
	<\frac{qw}{1-e^{-qw}}\cdot qw = \frac{q^2w^2}{1-e^{-qw}},
\]
implying that 
\[
    w\Lpa{1-\frac{q^2w}{1-e^{-qw}}}
	= \frac{w(1-q^2w-e^{-qw})}{1-e^{-qw}}<0.
\]
Thus the function $q\mapsto \partial^2_q\phi(q,w)$ is always negative
for all pairs $(q,w)$ such that $I(qw)=w$, showing that
$q\mapsto\phi(q,w)$ is concave downward when $w$ satisfies $I(qw)=w$;
see Figure~\ref{F:concave}. This proves the lemma.
\end{proof}

\begin{figure}[!ht]
\begin{center}
\includegraphics[height=3.5cm]{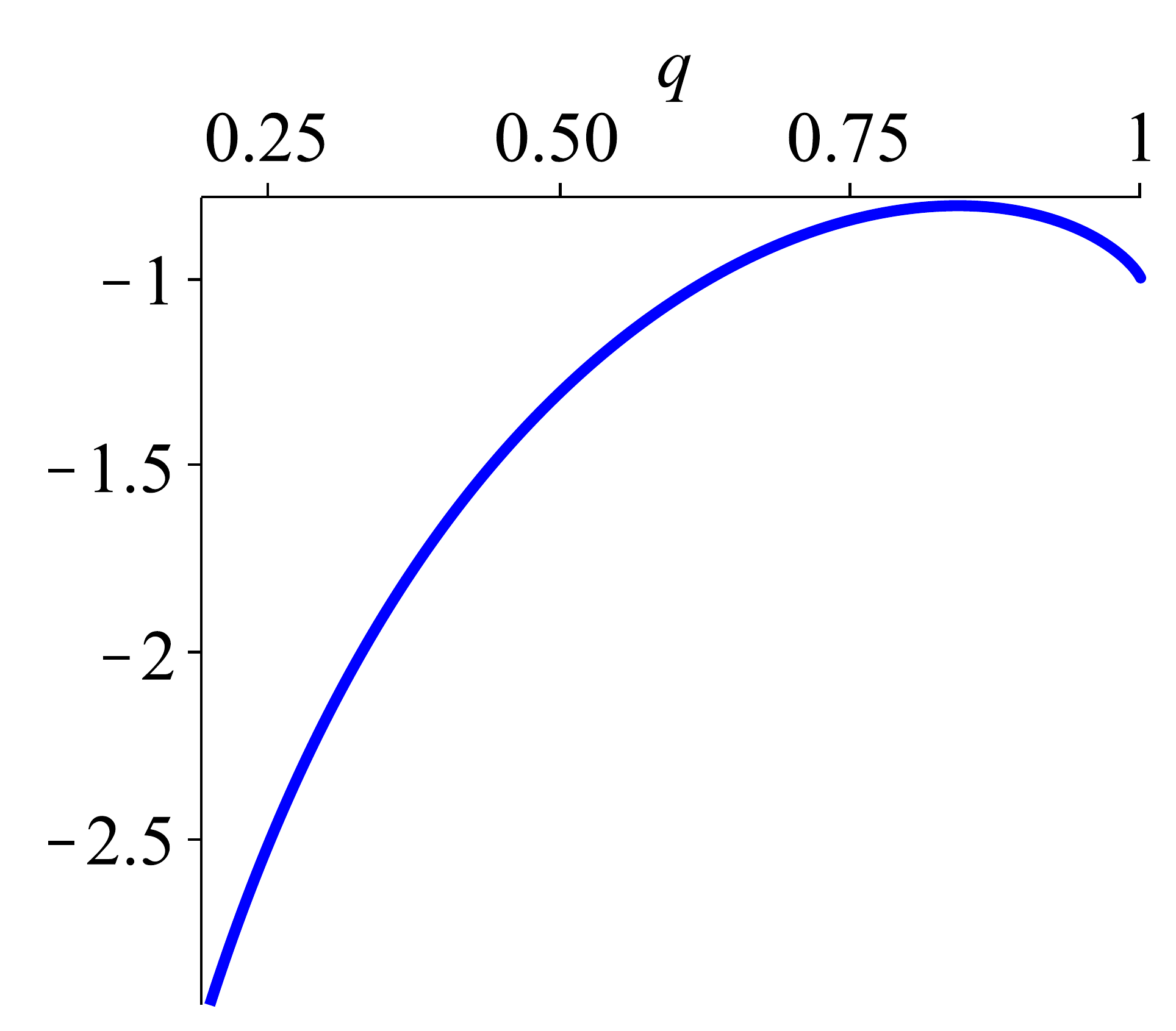}\;\;
\includegraphics[height=3.5cm]{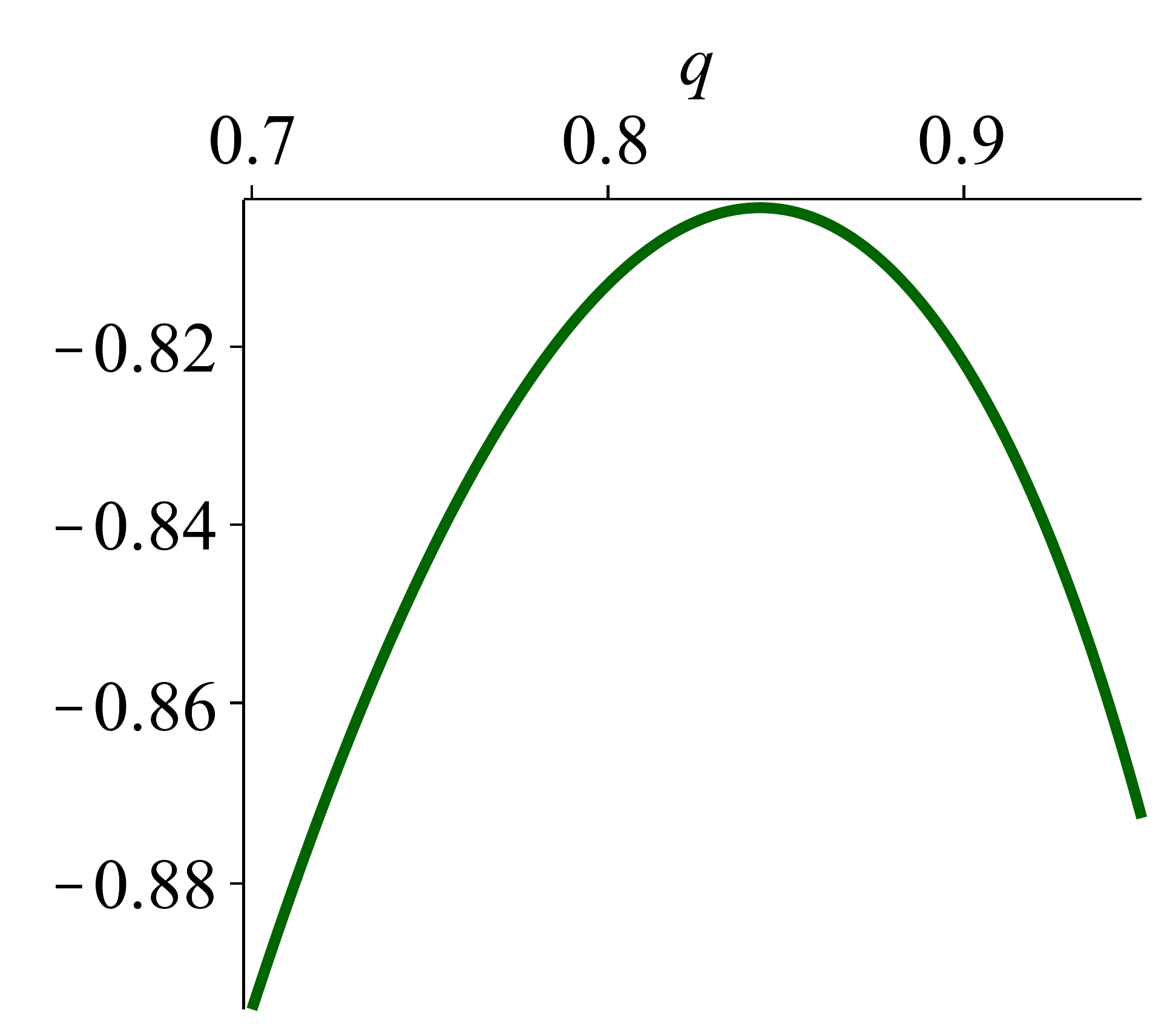} 	
\end{center}
\medskip
\caption{The concavity of the function $\phi(q,\varrho)$ when
$\varrho=\varrho(q)$ satisfies $I(q\varrho) = \varrho$ for
$q\in[0.2,1]$ (left) and $q\in[0.7,0.95]$ (right).} \label{F:concave}
\end{figure}


\subsection{Saddle-point method. II: Negligibility of summands 
outside the central range}

Define 
\begin{align}\label{E:sigma}
    \sigma := \pi^{-2}\sqrt{6\lpa{24(\log 2)^2-\pi^2}}
	\approx 0.31988.
\end{align}
\begin{proposition}\label{P:poly}
Write $k=\mu n+x\sigma\sqrt{n}$ where $\mu$ and $\sigma$ are given in
\eqref{E:mu-xi} and \eqref{E:sigma}, respectively. Then uniformly for
$x=o\lpa{n^{\frac16}}$,
\begin{align}\label{E:ank-sd4}
	a_{n,k} = O\lpa{\rho^n n^{n+\frac12}e^{-\frac12 x^2}},
	\with \rho:= \tfrac{12}{e\pi^2}.
\end{align}	
In particular when $k_\pm := \mu n \pm \sqrt{2}\,\sigma n^{\frac58}$, 
\begin{align}\label{E:negli-1}
    \llpa{\sum_{1\le k< k_-}
    +\sum_{k_+< k\le n}}\;a_{n,k}
    = O\lpa{\rho^n n^{n+\frac32} e^{-n^{\frac14}}}.
\end{align}
\end{proposition}

\begin{proof}    
Assume first 
\begin{align}\label{E:qx}
    q:=\mu + \frac{\sigma x}{\sqrt{n}},
\end{align}
where $\mu$ is defined in Lemma~\ref{L:max-phi} but the 
value of $\sigma$ given in \eqref{E:sigma} has remained unknown (and 
will be specified by the following procedure). Substituting this $q$ 
into the saddle-point equation \eqref{E:saddle0}, as approximated by 
\eqref{E:Lk-1}, and solving asymptotically for $\varrho$, we then 
obtain
\begin{align}\label{E:refw}
    \varrho =\xi +\frac{\xi_1 x}{\sqrt{n}}
    +\frac{\xi_2+\xi_3 x^2}{n}
    +O\Lpa{\frac{|x|+|x|^3}{n^{\frac32}}},
\end{align}
where, with $\tau:=2(\log 2)^2-\frac{\pi^2}{12}$ 
\begin{align}\label{E:xi1-3}
	\xi_1:=-\tfrac{\pi^4\log 2}{72\,\tau},
	\quad \xi_2:=-\tfrac{\pi^4(2\log 2-1)}{288\,\tau},
	\and 
	\xi_3:=\tfrac{\pi^6(288\tau^2
	+(\log 2)\pi^4+24\pi^2\tau-\pi^4)}{248832\,\tau^3}.
\end{align}
Then we substitute the expansions \eqref{E:qx} and \eqref{E:refw} 
into $ \phi(q,\varrho)$ (defined in Lemma \ref{L:max-phi}), giving
\begin{align*}
    \phi(q,\varrho) &=-\log\frac{\pi^2}{12}-1
    +\frac{1}{n}\Lpa{\frac{1}{2}-\log 2
    -\frac{\pi^4\sigma^2x^2}{144\tau}}
    +O\Lpa{\frac{|x|+|x|^3}{n^{\frac32}}}.
\end{align*}
So if we take $\sigma^2:=72\pi^{-4}\tau$ (which is identical to the 
expression \eqref{E:sigma}), then we see that 
\begin{align*}
    e^{n\phi(q,\varrho)}
    =\frac{\sqrt{e}}{2}
    \,\Lpa{\frac{12}{e\pi^2}}^n
    e^{-\frac12x^2}
    \left(1+O\Lpa{\frac{|x|+|x|^3}{\sqrt{n}}}\right),
\end{align*}
uniformly for $x=o\lpa{n^{\frac16}}$. This, together with 
\eqref{E:ank-sd2} and Lemma~\ref{L:max-phi}, proves 
\eqref{E:ank-sd4}.  

By monotonicity of $\phi(q,w)$ (see Lemma~\ref{L:max-phi}), the
left-hand side of (\ref{E:negli-1}) is bounded from above by
$(na_{n,k_-}+na_{n,k_+})$. In consequence, \eqref{E:negli-1} follows 
from \eqref{E:ank-sd4} with $x=\sqrt{2}\,n^{\frac18}$. 
\end{proof}

\subsection{Saddle-point method. III: Negligibility of integrals  
away from zero}

We now show that in the remaining sum ($k_\pm$ defined in Proposition 
\ref{P:poly})
\[
    \sum_{k_-\le k\le k_+}
    \frac{r^{-n}}{2\pi}\int_{-\pi}^{\pi}
    e^{-in\theta}A_k(re^{i\theta})\dd \theta,
\]
the integral over the range $\theta_0\le |\theta|\le\pi$, $\theta_0 
:= 6n^{-\frac38}$ is asymptotically negligible. Such a $\theta_0$ is 
always chosen so that $n\theta_0^2\to\infty$ and $n\theta_0^3\to0$; 
see \cite{Flajolet2009}. We begin with a uniform bound for 
$|A_k(z)|$.
\begin{lemma}
Let $\theta := \arg(z)$. Then, uniformly for $|z|>0$ and 
$|\theta|\le\pi$,
\begin{align}\label{E:Akz-ineq}
    |A_k(z)|
    \le A_k(|z|)\exp\left(-\frac{k(k+1)|z|\,\theta^2}
    {2\pi^2}\right), \quad(k=1,2,\dots).
\end{align}
\end{lemma}
\begin{proof}
The uniform bound \eqref{E:Akz-ineq} is a direct consequence of 
the inequality (see \cite[Appendix]{Pittel1986})
\begin{align}\label{E:pittel}
    \bigl| e^z-1 \bigr|
    \le \lpa{e^{|z|}-1}e^{-|z|\theta^2/\pi^2},
    \quad(|\theta|\le\pi).
\end{align}
This is proved as follows. First
\begin{align*}
    \bigl| e^z-1 \bigr| 
    &= \bigl| e^{\frac12z}\bigr| \bigl| 
    e^{\frac12z}-e^{-\frac12z} \bigr|\le e^{\frac12|z|\cos\theta}
    \lpa{e^{\frac12|z|}-e^{-\frac12|z|}}= \lpa{e^{|z|}-1}e^{-\frac12|z|(1-\cos\theta)},
\end{align*}
where the inequality results from the fact that $[t^n]
\lpa{e^t-e^{-t}} \ge 0$ for all $n\ge0$. Then \eqref{E:pittel}
follows from the elementary inequality $1-\cos\theta\ge
\frac 2{\pi^2}\theta^2$ for $|\theta|\le \pi$.
\end{proof}

\begin{proposition}\label{P:central-o}
Define $k_\pm := \mu n \pm \sqrt{2}\,\sigma n^{\frac58}$ as in 
Proposition \ref{P:poly} and $\theta_0:= 6n^{-\frac38}$. Then, 
\begin{align}\label{E:negli-2}
    \sum_{k_-\le k\le k_+}
    \frac{r^{-n}}{2\pi}\int_{\theta_0\le |\theta|\le\pi}
    e^{-in\theta}A_k(re^{i\theta})\dd \theta
    =O\lpa{\rho^n n^{n-\frac1{8}}
    e^{-n^{\frac14}}}, \with \rho:=\frac{12}{e\pi^2}.
\end{align}
\end{proposition}
\begin{proof}
By \eqref{E:Akz-ineq} with $z:=re^{i\theta}$,
\begin{align*}
    \sum_{k_-\le k\le k_+}
    \frac{r^{-n}}{2\pi}\int_{\theta_0\le |\theta|\le\pi}
	e^{-in\theta}A_k(re^{i\theta})\dd \theta
    &= O\llpa{\sum_{k_-\le k\le k_+}r^{-n}A_k(r)
    \int_{\theta_0}^{\infty}
    e^{-\frac{k^2r\theta^2}{2\pi^2}}\dd \theta}.
\end{align*}
Now, with $k\sim\mu n$ ($k_-\le k\le k_+$) and $rn\sim \xi$ (see 
\eqref{E:mu-xi}), we then have
\begin{align*}
    \int_{\theta_0}^{\infty}
    e^{-\frac{k^2r\theta^2}{2\pi^2}}\dd \theta
    &= O\Lpa{\frac{n^{\frac38}}{k^2r}\,
    e^{-\frac{k^2r}{2\pi^2 n^{3/4}}}}
    = O\lpa{n^{-\frac58}e^{-\frac{216(\log 2)^2}{\pi^4}
    (1+o(1))n^{\frac14}}}.
\end{align*}
Note that $\frac{216(\log 2)^2}{\pi^4} \approx 1.065>1$. 
Then \eqref{E:negli-2} follows from \eqref{E:ank-sd4}. 
\end{proof}

\subsection{Saddle-point method. IV: Proof of Theorem~\ref{T:eg1}}
\label{S:spm4}

From the two estimates \eqref{E:negli-1} and \eqref{E:negli-2}, we 
have, with $\rho = \frac{12}{e\pi^2}$ and $\theta_0:= 6n^{-\frac38}$,
\begin{align}\label{E:an-af}
    a_n = \sum_{k_-\le k\le k_+} r^{-n}A_k(r) 
    \cdot\frac{1}{2\pi}\int_{-\theta_0}^{\theta_0}
    e^{-in\theta}\frac{A_k(re^{i\theta})}{A_k(r)}\dd\theta
    +O\lpa{\rho^n n^{n+\frac32} e^{-n^{\frac14}}}.
\end{align}
We begin by evaluating asymptotically the integral . 

\begin{lemma} If $k=\mu n + x\sigma\sqrt{n}$, where $\mu$ and 
$\sigma$ are given in \eqref{E:mu-xi} and 
\eqref{E:sigma}, respectively, then
\begin{align}\label{E:JI-ae}
    J_I := \frac{1}{2\pi}\int_{-\theta_0}^{\theta_0}
    e^{-in\theta}\frac{A_k(re^{i\theta})}{A_k(r)}\dd\theta
    \simeq \frac{\sqrt{3}}{\pi^{3/2}\sigma}\,n^{-\frac12},
\end{align}
uniformly for $x=o\lpa{n^{\frac16}}$.
\end{lemma}
\begin{proof}
Expand $L_k(re^{i\theta})$ in $\theta$:
\begin{align*}
    \log \frac{A_k\lpa{re^{i\theta}}}{A_k(r)} 
	= L_k\lpa{re^{i\theta}}-L_k(r)
	:= \sum_{j\ge1}\frac{\upsilon_j(r)}{j!}(i\theta)^j.
\end{align*}
First of all, $\upsilon_1(r) = \frac{rA_k'(r)}{A_k(r)} = rL_k'(r) 
= n$ by our choice of $r$. Then by \eqref{E:Lk-m} with $q:=\frac kn$ 
and $\varrho := nr$ satisfying \eqref{E:qx} and \eqref{E:refw}, we 
obtain 
\[
    \upsilon_2(r) = r^2L_k''(r)+rL_k'(r)
	= \Lpa{\frac{24}{\pi^2}(\log 2)^2-1}n+O(1)
	= \frac{\pi^2}6\,\sigma^2n +O(1).
\] 
Furthermore, each $\upsilon_j(r)\asymp n$ by \eqref{E:Lk-m} when 
$k_-\le k\le k_+$. Thus $\upsilon_j(r)\theta_0^j \to0$ for 
$j=3,4,\dots$, and we then obtain 
\begin{align*}
    J_I &= \frac1{2\pi}\int_{-\theta_0}^{\theta_0}
	e^{-\frac12\upsilon_2(r) \theta^2-\frac16\upsilon_3(r)
    i\theta^3+O(n\theta^4) }\dd \theta\\
    &= \frac1{2\pi}\int_{-\infty}^{\infty}
    e^{-\frac12\upsilon_2(r) \theta^2}
    \left(1-\tfrac16{\upsilon_3(r)}
    i\theta^3+O(n\theta^4+n\theta^6)\right)\dd \theta
    +O\lpa{\upsilon_2^{-1}n^{\frac38}
    e^{-18\upsilon_2(r)n^{-\frac34}}}\\
	&= \frac1{\sqrt{2\pi\upsilon_2(r)}}
    \lpa{1+O\lpa{n^{-1}}} + O\lpa{n^{-\frac58}
    e^{-3\pi^2\sigma^2n^{\frac14}}},
\end{align*}
which proves \eqref{E:JI-ae}.
\end{proof}

\begin{proof}[Proof of Theorem~\ref{T:eg1}]
With \eqref{E:an-af} and \eqref{E:JI-ae} available, we can now 
complete the proof of Theorem \ref{T:eg1} by deriving the refined
expansion 
\[
    r^{-n}A_k(r)
	= c_0 \rho^n n^{n+\frac12}e^{-\frac12x^2}
    \left(1+\frac{g_1(x)}{\sqrt{n}}+O\lpa{n^{-1}(1+x^6)}\right),
\]
where $(c_0,\rho) := \lpa{\sqrt{\tfrac{24}{\pi}}, \tfrac{12}{e\pi^2}}$
and $g_1(x)$ is an odd polynomial in $x$ of degree three (whose 
expression being immaterial here). It follows that
\begin{equation}\label{E:ank-central}
    \begin{split}
        a_{n,k} &:= [z^n]A_k(z)
        = \frac{\sqrt{3}}{\pi^{3/2}\sigma}\,n^{-\frac12}
    	r^{-n}A_k(r)\lpa{1+O\lpa{n^{-1}}}\\
        &= \frac{c_1}{\sigma}\,\rho^n n^{n} e^{-\frac12x^2}
        \Lpa{1+\frac{g_1(x)}{\sqrt{n}}+O\lpa{n^{-1}(1+x^6)}},
    \end{split}
\end{equation}
uniformly for $k_-\le k\le k_+$, where $(c_1,\rho) 
:= \lpa{\tfrac{\sqrt{72}}{\pi^2},\tfrac{12}{e\pi^2}}$, where $\sigma$ 
is given in \eqref{E:sigma}. 

From this and the two estimates \eqref{E:negli-1} and 
\eqref{E:negli-2}, we obtain  
\[
    a_n = \frac{12}{\pi^{3/2}}\,\rho^n n^{n} 
	\sum_{k_-\le k\le k_+} 
	\frac{e^{-\frac12x^2}}{\sqrt{2\pi}\,\sigma}
	\left(1+\frac{g_1(x)}{\sqrt{n}}
	+O\lpa{n^{-1}(1+x^6)}\right) +O\lpa{\rho^n 
	n^{n+\frac32}e^{-n^{\frac14}}},
\] 
from which we deduce \eqref{E:A} by approximating the sum by 
integral. 
\end{proof}

\begin{remark}\label{R:llt}
We have proved more than the asymptotic estimate \eqref{E:A}; indeed, 
if we define the random variable $X_n$ by 
\[
    \mathbb{P}(X_n=k) := \frac{[z^n]A_k(z)}{[z^n]A(z)}
    \qquad(1\le k\le n), 
\]
then our asymptotic expansions \eqref{E:ank-central} and \eqref{E:A}
imply obviously the local limit theorem (in the form of moderate
deviations):
\[
    \mathbb{P}(X_n = \mu n + x\sigma\sqrt{n})
    = \frac{e^{-\frac12x^2}}{\sqrt{2\pi\sigma^2 n}}
    \left(1+O\Lpa{\frac{|x|+|x|^3}{\sqrt{n}}}\right),
\]
uniformly for $x=o\lpa{n^{\frac16}}$. 
\end{remark}

\section{Asymptotic expansions and change of variables}
\label{S:ae}

We examine briefly in this section two different ways to obtain 
asymptotics expansions for $a_n=[z^n]A(z)$ as defined in \eqref{E:A}, 
and then show how an argument based on change of variables leads to 
expansions for the coefficients under different parametrization of 
the underlying function. 

The first approach to deriving an expansion of the form 
\begin{align}\label{E:znAz-ae}
	[z^n]A(z) = c\rho^n n^{n+\frac12}
	\llpa{1+\sum_{1\le j<m}\nu_j n^{-j}+O\lpa{n^{-m}}}, 
\end{align}
for some computable coefficients $\nu_j$, is now straightforward 
following the same analysis detailed in the previous section. It 
consists in first computing an asymptotic expansion for $a_{n,k}$, 
which is of the form 
\begin{align*}
	a_{n,k} = \frac{c_1}{\sigma}
	\,\rho^n n^{n} e^{-\frac12x^2}
	\llpa{1+\sum_{1\le j<m}\frac{g_j(x)}{n^{\frac12j}}
	+O\lpa{n^{-\frac12m}}},
    \with  (c,\rho) 
    := \lpa{\tfrac{12}{\pi^{3/2}},\tfrac{12}{e\pi^2}}
\end{align*}
uniformly for $k=\mu n+x\sigma\sqrt{n}$, $x=o(n^{\frac16})$, where 
$g_j(x)$ is a computable polynomial in $x$ of degree $3j$ and 
contains only powers of $x$ with the same parity as $j$. From this we 
can then deduce \eqref{E:znAz-ae} by approximating the sum by an 
integral and extending the integration range to $\pm\infty$. We omit 
the details as they are more or less standard and all 
procedures can be readily coded in symbolic computation softwares.

\subsection{An asymptotic expansion via Dirichlet series}

For more methodological interest, we sketch here another approach,
based on that used in \cite{Bringmann2014}, to obtaining asymptotic
expansions for $a_n$ when more information is available.

\begin{proposition} The sequence $a_n$ in \eqref{E:A} satisfies the 
    asymptotic expansion
\begin{align}\label{E:A158690-ae}
	a_n
	= c \rho^{n}\, n!
	\llpa{1+\sum_{1\le j<m} \frac{c_j}{n(n-1)\cdots(n-j+1)}
	+ O\lpa{n^{-m}}},
\end{align}
for $m\ge2$, where $(c,\rho):=\lpa{\tfrac{6\sqrt{2}}{\pi^2}, 
\tfrac{12}{\pi^2}}$ and $c_j := \frac{1}{j!}
\lpa{-\frac{\pi^2}{288}}^j$ for $j\ge1$. 
\end{proposition}
In particular, 
\[	
	a_n  
	= c\rho^n n!\llpa{1-\frac{\pi^2}{288 n} + 
	\frac{\pi^4}{165888\, n(n-1)}+O\lpa{n^{-3}}}.
\]
The very simple form of the coefficients $c_j$ naturally suggests the
following approximation:
\[
    a_n
	= c \rho^n n! e^{-\frac{\pi^2}{288\,n}}
	\left(1+O\lpa{n^{-3}}\right),
\] 
which has obvious numerical advantages. 
\begin{proof}
We begin with \eqref{E:s1}. As in \cite{Bringmann2014}, we define the 
Dirichlet series
\[
    D(s) := \sum_{n\ge1}n^{-s} [q^{n-1}]R\lpa{q^{24}},
\]
which converges absolutely in $\Re(s)>1$ and can be analytically 
continued into the whole $s$-plane. Together with Mellin transform 
techniques, we now have the two relations \cite{Bringmann2014}
\begin{equation*}
    \left\{
        \begin{split} \displaystyle
            a_n &= \frac{(-1)^n}2[z^n]R\lpa{e^{-z}},\\
            \displaystyle
            b_n &:=[z^n]e^{-\frac1{24}z}R\lpa{e^{-z}}
            = \frac{(-1)^n D(-n)}{n!24^n}.
        \end{split}
    \right.
\end{equation*}
Then, by the functional equation derived in \cite{Cohen1988}, one has
\[
    D(-n) = c_0\rho_0^n
	n!^2\left(1+O\lpa{23^{-n}}\right), 
    \with (c_0,\rho_0) := \lpa{\tfrac{12\sqrt{2}}{\pi^2},
    \tfrac{288}{\pi^2}},
\]
for large $n$. This implies that 
\[
    b_n = c_0(-1)^n\rho^n
    n!\left(1+O\lpa{23^{-n}}\right), 
    \with 
    (c_0,\rho) := \lpa{\tfrac{12\sqrt{2}}{\pi^2},
    \tfrac{12}{\pi^2}}.
\]
From this, we have 
\begin{align}\label{E:ratio-bn}
    \frac{b_{n-j}}{b_n} 
    = \frac{(-\rho)^{-j}}{n(n-1)\cdots(n-j+1)}
	\left(1+O\lpa{23^{-n}}\right)
	\quad(j=0,1,\dots),
\end{align}
implying that the partial sum 
\begin{align*}
	a_n
	&= \frac{(-1)^n}2\,b_n \sum_{0\le j\le n}
	\frac{b_{n-j}}{j!24^jb_n}
\end{align*}
is itself an asymptotic expansion. In this way, we obtain 
\eqref{E:A158690-ae}. 
\end{proof}
 

\subsection{From $[z^n]R(e^z)$ to $[z^n]R(1+z)$
by a change of variables}
\label{S:cov}

We sketch here a different technique to derive the asymptotic
expansion \eqref{E:blr} for $[z^n]\sum_{k\ge0}\prod_{1\le j\le
k}\lpa{(1+z)^j-1}$ from that \eqref{E:A158690-ae} for $a_n$. The
original proof by Zagier in \cite{Zagier2001} and by
Bringmann-Li-Rhoades in \cite{Bringmann2014} uses the asymtotics
of the Stirling numbers of the first kind. We give a direct approach
via change of variables, which has the advantages of being easily
codable and widely applicable in more general contexts; see
Sections~\ref{S:gf} and \ref{S:jelinek}.

Define $R(q)$ by \eqref{E:ramanujan}. Since 
\begin{align*}
    R(q)=2\sum_{k\ge 0}\prod_{1\le j\le k}\lpa{q^j-1}
\end{align*}
is true to infinite order at every root of unity which includes the 
case $q=1$ (see \cite{Cohen1988}), by the change of variables 
$1+z=e^y$, we have,
\begin{align*}
    [z^n]\sum_{k\ge0}
	\prod_{1\le j\le k}\lpa{(1+z)^j-1}
    = \frac12[z^n]R(1+z)
	= [y^n]g(y)\lpa{e^{\frac1{24}y}R\lpa{e^y}},
\end{align*}
where
\begin{align*}
	g(y) &:= \frac12\left(\frac{y}{e^y-1}\right)^{n+1}
	e^{\frac{23}{24}y}
	= \frac12\exp\llpa{-\frac n2y+\frac{11}{24}y 
	-(n+1)\sum_{j\ge1}\frac{B_{2j}}{2j\cdot (2j)!}\, y^{2j}},
\end{align*}
for small $y$. Since $b_n$ (see \eqref{E:ratio-bn}) grows factorially
with $n$, and the Taylor coefficients of $g(y)$ are small when
compared to $b_n$, we expand $g$ at $y=\eta$, where $\eta$ is small
and to be determined soon, and then carry out term by term extraction
of the coefficients, yielding
\begin{align*}
    [y^n]g(y)\lpa{e^{\frac1{24}y}R\lpa{e^y}}
    &= \sum_{j\ge0}\frac{g^{(j)}(\eta)}{j!}
	[y^n](y-\eta)^je^{\frac1{24}y}R\lpa{e^y} \\
	&= g(\eta) \bar{b}_n + g'(\eta)\lpa{\bar{b}_{n-1}
    -\eta \bar{b}_{n}} + \cdots,
\end{align*}
where $\bar{b}_n := (-1)^n b_n = [y^n]e^{\frac1{24}y}R\lpa{e^y}$. 
So if we take (see \eqref{E:ratio-bn})
\[
    \eta:= \frac{\bar b_{n-1}}{\bar b_n} = \frac{\pi^2}{12n}
	\left(1+O\lpa{23^{-n}}\right),
\]
then the terms involving $g'(\eta)$ become zero, and we have
\[
    [y^n]\lpa{e^{\frac1{24}y}R\lpa{e^y}}g(y)
    = g(\eta) \bar{b}_n\llpa{1 + \frac{g''(\eta)}{2g(\eta)}
    \Lpa{\frac{\bar b_{n-2}}{\bar b_n}-
    \frac{\bar b_{n-1}^2}{\bar b_n^2}}+\cdots}.
\]
In general, by estimating the Taylor remainders, we deduce the 
expansion
\[
    [y^n]\lpa{e^{\frac1{24}y}R\lpa{e^y}}g(y)
    = g(\eta) \bar b_n \llpa{1+\sum_{2\le j\le 2m}
	\frac{g^{(j)}(\eta)}{j!g(\eta)}\,H_j(n)+
	O\lpa{n^{-m-1}}},
\]
for $m\ge1$, where the general terms are of order $n^{\cl{\frac12j}}$ 
because $g^{(j)}(\eta)=O(n^j)$ and 
\[
    H_j(n) := \sum_{0\le \ell\le j}
	\binom{j}{\ell}\left(-\frac{\pi^2}{12n}\right)^{j-\ell}
	\frac{\bar b_{n-\ell}}{\bar{b}_n}
	= \left(\frac{\pi^2}{12}\right)^{j}
	\sum_{0\le \ell\le j}\binom{j}{\ell}
	\frac{(-1)^{j-\ell}(n-\ell)!}{n^{j-\ell}n!}
	\left(1+O\lpa{23^{-n}}\right),
\]
which decays in the order $n^{-j-\cl{\frac12j}}$. In this way, we 
obtain 
\begin{align}\label{E:blr2}
    [z^n]\sum_{k\ge0}
	\prod_{1\le j\le k}\lpa{(1+z)^j-1}
	= c \rho^{n}\, n!
	\llpa{1+\sum_{1\le j<m} \frac{c_j}{n^j}
	+O\lpa{n^{-m}}},
\end{align}
where $(c,\rho):=\lpa{\tfrac{6\sqrt{2}}{\pi^{2}}\,
e^{-\frac{\pi^2}{24}}, \tfrac{12}{\pi^2}}$ and 
\begin{align*}
	c_1 &:= \frac{\pi^2(\pi^2+66)}{1728}\approx 0.43333, \quad 
	c_2:= \frac{\pi^4(\pi^4-12\pi^2-3420)}{5971968}
	\approx -0.05612,\\
	c_3 &:= -\frac{\pi^4(95\pi^8+9360\pi^6
	-232416\pi^4-27051840\pi^2+709171200)}{1238347284480}
	\approx -0.03378.
\end{align*}


\section{A framework for matrices with $1$s}
\label{S:gf}

We consider in this section generating functions of the form 
\begin{align}\label{E:sum-d-prod-c}
	\sum_{k\ge0}d(z)^{k+\omega_0}
	\prod_{1\le j\le k}\lpa{e(z)^{j+\omega}-1}^\alpha,
\end{align}
for $\alpha\in\mathbb{Z}^+$ and $\omega_0,\omega\in\mathbb{C}$, where
$d(z)$ and $e(z)$ are formal power series satisfying $d(0)>0$,
$e(0)=1$ and $e'(0)\ne0$. Then we discuss applications to
row-Fishburn, Fishburn matrices with entry restrictions and some OEIS
sequences.

Our approach consists in examining first the asymptotics of the 
simpler pattern 
\begin{align*}
	[z^n]\sum_{k\ge0}
	\prod_{1\le j\le k}\lpa{e^{(j+\omega)z}-1}^\alpha,
\end{align*}
for $\alpha\in\mathbb{Z}^+$ and $\omega\in\mathbb{C}$, and follows
closely the detailed analysis given in Section~\ref{S:basis} for the
sequence \href{https://oeis.org/A158690}{A158690}. Then the extension 
to \eqref{E:sum-d-prod-c} will rely on the change of variables 
argument of Section~\ref{S:cov}.

\begin{proposition} \label{P:eoza}
For any $\alpha\in\mathbb{Z}^+$ and $\omega\in\mathbb{C}$,
\begin{align}\label{E:proto1}
	[z^n]\sum_{k\ge0}
	\prod_{1\le j\le k}\lpa{e^{(j+\omega)z}-1}^{\alpha}
	\simeq c\rho^{n} n^{n+\alpha\omega
	+\frac12\alpha},
\end{align}
uniformly in $\omega$, where the notation ``$\simeq$'' is defined in 
\eqref{E:a-cong-b} and
\begin{align*}
	(c,\rho) := \llpa{\frac{\sqrt{6}}{\alpha\pi}
	\llpa{\frac{2\sqrt{6}}{\sqrt{\alpha\pi}\,
	\Gamma(1+\omega)} 
	\Lpa{\frac{12}{\alpha\pi^2}}^\omega}^{\alpha},
	\frac{12}{e\alpha\pi^2}}.
\end{align*}	
\end{proposition}
When $\omega\in\mathbb{Z}^-$, the leading constant $c$ is interpreted 
as zero because of $\Gamma(1+\omega)$ in the denominator, and the 
right-hand side of \eqref{E:proto1} becomes then a big-$O$ estimate. 
\begin{proof}
We sketch the major steps for obtaining the dominant term, as the error 
term follows from the same procedure with more refined 
calculations. 
\begin{itemize} 
    
	\item By the Euler-Maclaurin formula \eqref{E:emf}
	\begin{equation}\label{E:emf-de}
        \begin{split}
    	    \sum_{1\le j\le k}\log\lpa{e^{(j+\omega)z}-1}
    		&= k\log\lpa{e^{kz}-1}-\frac{I(kz)}z
    		+\Lpa{\omega+\frac12}
    		\log\frac{e^{kz}-1}{z}
			\\&\qquad -\log\Gamma(1+\omega)+ \frac{\log2\pi}2
            +O\lpa{|\omega|^2\lpa{k^{-1}+|z|}},
        \end{split}
	\end{equation}
    (compare \eqref{E:Lkz})
    which holds uniformly $k\to\infty$ and $k|z|\le2\pi-\ve$ in the 
    sector $|\arg z|\le \pi-\ve$. Here \eqref{E:emf-de} holds when 
    $\omega\ne\mathbb{R}^-$. But the asymptotic approximation, by 
    taking the exponential on both sides of \eqref{E:emf-de}, 
	\begin{align*}
        \begin{split}
    	    &\prod_{1\le j\le k}\lpa{e^{(j+\omega)z}-1}\\
    		&\qquad= \frac{\sqrt{2\pi}}{\Gamma(1+\omega)}
            \Lpa{\frac{e^{kz}-1}{z}}^{\omega+\frac12}
            \lpa{e^{kz}-1}^k e^{-I(kz)/z}
            \lpa{1+O\lpa{|\omega|^2\lpa{k^{-1}+|z|}}}
        \end{split}
	\end{align*}
    does hold for bounded $\omega$, provided we interpreted 
    the factor $\frac1{\Gamma(1+\omega)}$ as zero when 
    $\omega\in\mathbb{Z}^-$. 
    
    \item The saddle-point equation satisfies asymptotically, by the 
	same differentiation argument used for deriving \eqref{E:Lk-1}, 
    \[
        \frac\alpha r\,I(kr)
        +\frac\alpha 2(2\omega+1)
        \Lpa{\frac{kr}{1-e^{-kr}}-1}+O\lpa{k^{-1}+r}
        =n.
    \]
    Since the dominant term is independent of $\omega$, we deduce 
    that $k=qn$ with $q \sim \frac{\mu}{\alpha}$ and $rn\sim 
    \alpha\xi$, where $(\mu,\xi):=\lpa{\frac{12}{\pi^2}\log 2,
    \frac{\pi^2}{12}}$ is the same as in \eqref{E:mu-xi}.
   
   	\item Observe that for large $k\le n$ 
	\[
	    \prod_{1\le j\le k}\bigl|e^{(j+\omega)z}-1\bigr|
		= O\lpa{k^{\Re(\omega)}}
        \prod_{1\le j\le k}\bigl|e^{jz}-1\bigr|,
	\]
	when $|z|\asymp n^{-1}$ and $\omega=O(1)$. Then the smallness of 
	the sum 
    \[
        \sum_{|k-\frac\mu\alpha n|\ge \sqrt{2}\sigma n^{\frac58}} 
        [z^n]\prod_{1\le j\le k}\lpa{e^{(j+\omega)z}-1}^\alpha,
    \]
    as well as the corresponding sum of integrals
    $\sum_{|k-\frac\mu\alpha n|\le \sqrt{2}\sigma n^{\frac58}}
    \int_{6n^{-\frac38}\le|\theta|\le \pi}$ follows from the same
    bounding techniques used in the proofs of
    Propositions~\ref{P:poly} and \ref{P:central-o}.
    
    \item Inside the central range $\frac1\alpha k_-\le k\le
    \frac1\alpha{k_+}$, where $k_\pm := \mu n \pm \sqrt{2}\,
    \sigma n^{\frac58}$, write, as before, $q:=\frac1\alpha
    \lpa{\mu + \sigma \frac{x}{\sqrt{n}}}$, and solve the 
    saddle-point equation for $r$, giving
    \begin{align}\label{E:rn}
        rn = \alpha\xi + \frac{\alpha\xi_1x}{\sqrt{n}}
        +\frac{\alpha^2\xi_2(1+2\omega)+\alpha\xi_3x^2}{n}
        +O\Lpa{\frac{|x|+|x|^3}{n^{3/2}}},
    \end{align}
    where $\xi_i$ are defined in \eqref{E:xi1-3}.
	
	\item We then obtain
    \[
	    r^{-n}\prod_{1\le j\le k}\lpa{e^{(j+\omega)r}-1}^{\alpha}
		\sim c_0 \rho^n n^{n+\alpha(\frac12+\omega)},
	\]
	where 
	\[
		(c_0,\rho):=\llpa{ 
		\llpa{\frac{2\sqrt{6}}{\sqrt{\alpha\pi}\,
		\Gamma(1+\omega)} 
		\Lpa{\frac{12}{\alpha\pi^2}}^\omega}^{\alpha},
		\frac{12}{e\alpha\pi^2}}.
	\]
	
	\item The remaining saddle-point analysis is similar to that of 
	Theorem~\ref{T:eg1}. 
\end{itemize}	
\end{proof}
The uniformity in $\omega$ will be needed in Section~\ref{S:B}. We 
now consider the framework \eqref{E:sum-d-prod-c}. 

\begin{theorem}\label{T:gen}
Assume $\alpha\in\mathbb{Z}^+$ and $\omega_0, \omega\in \mathbb{C}$. 
For any two formal power series $d(z)$ and $e(z)$ satisfying 
$d(0)=e(0)=1$ and $e'(0)\ne0$, we have
\begin{align}\label{E:gen}
    [z^n]\sum_{k\ge0}d(z)^{k+\omega_0}
    \prod_{1\le j\le k}\lpa{e(z)^{j+\omega}-1}^\alpha
    \simeq c\rho^n n^{n+\alpha(\frac12+\omega)},
\end{align}
uniformly for bounded $\omega_0$ and $\omega$, where $d_j := 
[z^j]d(z)$, $e_j := [z^j]e(z)$, and 
\begin{align}\label{E:de-cre}
	(c,\rho) := \llpa{\frac{\sqrt{6}}{\alpha\pi}
	\llpa{\frac{2\sqrt{6}}{\sqrt{\alpha\pi}\,
	\Gamma(1+\omega)} 
	\Lpa{\frac{12}{\alpha\pi^2}}^\omega}^{\alpha}
	2^{\frac{d_1}{e_1}}e^{\frac{\alpha\pi^2}{12}
	\lpa{\frac{e_2}{e_1^2}-\frac12}},
	\frac{12e_1}{e\alpha\pi^2}}.
\end{align}   
\end{theorem}
The situation when $d(0)\ne1$ is readily modified. Also the error
term can be further refined if needed. The case when $e'(0)=0$ but
$e''(0)>0$ will be treated in Section \ref{S:jelinek} with
particular applications to self-dual Fishburn matrices.

We see that the exponential term depends on $\alpha$ and $e_1$, the
polynomial term on $\alpha$ and $\omega$, and the leading constant
$c$ on $\alpha,\omega,d_1,e_1$ and $e_2$. Furthermore, as far as the
dominant asymptotics of the coefficients is concerned, the difference
in \eqref{E:gen} and \eqref{E:proto1} is reflected via the first 
three terms $d_1, e_1, e_2$ in the Taylor expansions of $d(z)$ and 
$e(z)$, but not on $\omega_0$.
 
\begin{proof}
By Cauchy's integral formula
\begin{align*}
	a_n &:= \frac{1}{2\pi i}\oint_{|z|=r_0}z^{-n-1}
	\sum_{1\le k\le \frac n\alpha}d(z)^{k+\omega_0}
	\prod_{1\le j\le k}\left(e(z)^{j+\omega}-1\right)^\alpha\dd z,
\end{align*}
where $r_0>0$. Without loss of generality, we assume that both $d(z)$
and $e(z)$ are analytic at zero; otherwise, we truncate both formal
power series after the $n$th terms, resulting in two polynomials and
thus analytic functions at the origin. Since $e(z)= 1+e_1z+\cdots$
with $e_1\ne0$, the function is locally invertible and we can make 
the change of variables $e(z)=e^y$, giving
\begin{align*}
	a_n=\frac{1}{2\pi i}\oint_{|y|=r}\psi'(y)\psi(y)^{-n-1}
	\sum_{1\le k\le \frac n\alpha}d(\psi(y))^{k+\omega_0} A_k(y)
	\dd y,
\end{align*}
where $A_k(y) := \prod_{1\le j\le k}\lpa{e^{(j+\omega)y}-1}^\alpha$ 
and $\psi(y)$ satisfies $\psi(0)=0$ and $e(\psi(y))=e^y$. In 
particular, 
\begin{align}\label{E:d1d2}
	\psi_1=[y]\psi(y)=\tfrac1{e_1} \and 
	\psi_2=[y^2]\psi(y)=\tfrac1{e_1}\lpa{\tfrac{1}{2}
    -\tfrac{e_2}{e_1^2}}.
\end{align}
Observe first that for small $|y|$, 
\begin{align*}
    d(\psi(y))^{k+\omega_0}
    = \left(1+d_1\psi_1y+\lpa{d_1\psi_2+d_2\psi_1^2}y^2
    +\cdots \right)^k;
\end{align*}
on the other hand, from our saddle-point analysis above, the 
integration path $|y|=r$ is very close to zero with 
$r\asymp n^{-1}$, and most contribution to $a_n$ comes from  
terms with $k$ of linear order, so we see that $d(\psi(y))^k$ is 
bounded and close to $e^{d_1\psi_1ky}$ for large $n$. Similarly, by 
\eqref{E:d1d2},
\begin{align*}
	\psi'(y)\psi(y)^{-n-1}&=\lpa{\psi_1+2\psi_2y+O\lpa{|y|^2}}
	\left(\psi_1y+\psi_2y^2+O\lpa{|y|^3}\right)^{-n-1}\\
	&=e_1^n y^{-n-1} e^{-\frac{\psi_2}{\psi_1}ny}
    \left(1+O\lpa{|y|+n|y|^2}\right).
\end{align*}
Thus the same proof of Theorem~\ref{T:eg1} extends \emph{mutatis 
mutandis} to this case, and we then obtain the asymptotic 
approximation
\begin{align*}
    a_n &= \sum_{\frac{k_-}{\alpha}\le k\le \frac{k_+}{\alpha}}
    \frac1{2\pi i}\oint_{|y|=r}
    y^{-n-1} A_k(y)e^{-\frac{\psi_2}{\psi_1}ny+d_1\psi_1ky}
    \left(1+O\lpa{|y|+n|y|^2}\right)\dd y\\
	&\qquad + O\lpa{\rho^n n^{n+\alpha(\Re(\omega)+\frac12)}
	e^{-n^{\frac14}}},
\end{align*}
where $k_\pm := \mu n \pm \sqrt{2}\,\sigma n^{\frac58}$ and $r$ 
satisfies \eqref{E:rn}. Since $q=\frac kn$ satisfies 
$q=\frac1\alpha\lpa{\mu+\sigma\frac{x}{\sqrt{n}}}$, we 
then deduce \eqref{E:gen} by noting that 
\[
    e^{-\frac{\psi_2}{\psi_1}nr+d_1\psi_1kr}
    = e^{-\frac{\psi_2}{\psi_1}\alpha\xi+d_1\psi_1\mu\xi}
    \Lpa{1+\frac{\tilde{g}_1(x)}{\sqrt{n}}+
    \frac{\tilde{g}_2(x)}{n}+\cdots},
\]
for some polynomials $\tilde{g}_1(x)$ and $\tilde{g}_1(x)$, where 
$(\mu,\xi)$ is given in \eqref{E:mu-xi}. 
\end{proof}

 
\section{Applications I. Univariate asymptotics}
\label{S:A}

We group in this section various examples (mostly from the OEIS)
according to the pair $(\alpha,\omega)$. Some of them were already 
analyzed in the OEIS by Kot\v e\v sovec, but without proofs.

\subsection{$\Lambda$-row-Fishburn matrices and 
examples with $(\alpha,\omega):=(1,0)$}
\label{S:A1}

We derive a general asymptotic approximation to the number of  
$\Lambda$-row-Fishburn matrices and discuss some other examples. 

\subsubsection{$\Lambda$-row-Fishburn matrices}

From Theorem~\ref{T:gen}, it is clear that no matter how widely we 
choose the nonnegative integers as entries, the number of the 
resulting row-Fishburn matrices of size $n$ depends only on the 
numbers of $1$s and $2$s as far as the leading order asymptotics is 
concerned, provided that the generating function satisfies 
\eqref{E:Lambda-z}. 

\begin{corollary}\label{C:1}
Let $\Lambda$ be a multiset of nonnegative integers with the 
generating function
\begin{align}\label{E:Lambda-z}
    \Lambda(z):=1+\sum_{\lambda\in \Lambda}z^{\lambda}
    =1+\lambda_1z+\lambda_2z^2+\cdots.
\end{align}
If $\Lambda'(0)=\lambda_1>0$, then the number of 
$\Lambda$-row-Fishburn matrices of size $n$ satisfies 
\begin{align}\label{E:lrf-cor}
    [z^n]\sum_{k\ge 0}\prod_{1\le j\le k}\lpa{\Lambda(z)^j-1}
    \simeq c\rho^n n^{n+\frac{1}{2}}\with
    (c,\rho):=\Lpa{\tfrac{12}{\pi^{3/2}}\,
    e^{\frac{\pi^2}{12}
    \lpa{\frac{\lambda_2}{\lambda_1^2}-\frac{1}{2}}},
    \tfrac{12\,\lambda_1}{e\pi^2}}.
\end{align}
\end{corollary}
\begin{proof}
Apply Theorem~\ref{T:gen} with $(d(z),e(z)) := (1,\Lambda(z))$.    
\end{proof}

In particular, this corollary applies to the OEIS sequences in 
Table~\ref{T:lrf}. 

\begin{table}[!ht]
\begin{center}
\begin{tabular}{cccccc} \hline
OEIS & $\Lambda$ & $\Lambda(z)$ & $(\lambda_1,\lambda_2)$ 
& $c$  & $\rho$ \\ \hline  
\href{https://oeis.org/A179525}{A179525} & $\{0,1\}$
& $1+z$ & $(1,0)$ & $\frac{12}{\pi^{3/2}}\,
e^{-\frac{\pi^2}{24}}$ & $\frac{12}{e\pi^2}$ \\

\href{https://oeis.org/A289316}{A289316} & 
$\{0\}\cup\{2k-1:k\in\mathbb{Z}^+\}$ &
$\frac{1+z-z^2}{1-z^2}$ & $(1,0)$ & 
$\frac{12}{\pi^{3/2}}\,e^{-\frac{\pi^2}{24}}$
& $\frac{12}{e\pi^2}$ \\ \hline

\href{https://oeis.org/A207433}{A207433} & $\{0,1,2\}$ &
$\frac{1-z^3}{1-z}$ & $(1,1)$
& $\frac{12}{\pi^{3/2}}\,e^{\frac{\pi^2}{24}}$
& $\frac{12}{e\pi^2}$ \\

\href{https://oeis.org/A158691}{A158691} 
& $\mathbb{Z}_{\ge0}$ & 
$\frac1{1-z}$ & $(1,1)$ 
& $\frac{12}{\pi^{3/2}}\,e^{\frac{\pi^2}{24}}$
& $\frac{12}{e\pi^2}$ \\ \hline

\href{https://oeis.org/A289313}{A289313} 
& $\{0,1,1,2,2,\dots\}$ &
$\frac{1+z}{1-z}$ & $(2,2)$ &
$\frac{12}{\pi^{3/2}}$ & $\frac{24}{e\pi^2}$\\
\hline
\end{tabular}    
\end{center}  
\medskip
\caption{The large-$n$ asymptotics \eqref{E:lrf-cor} of some OEIS 
sequences that correspond to the enumeration of 
$\Lambda$-row-Fishburn matrices with 
different $\Lambda$. Here we split the pair $(c,\rho)$ for clarity 
and group the sequences with the same pair $(\lambda_1,\lambda_2)$.}
\label{T:lrf}
\end{table}
The last sequence of Table \ref{T:lrf} can also be interpreted as the
number of upper triangular matrices with integer entries (positive
and negative) whose sum of absolute entries is $n$, and no row sums
(in absolute entries) to zero.

\subsubsection{Some OEIS sequences}
Some other OEIS examples with $(\alpha,\omega):=(1,0)$ are compiled 
in Table~\ref{T:4}, where they all satisfy the asymptotic pattern
\begin{align}\label{de-ais1}
    [z^n]\sum_{k\ge0}d(z)^k
    \prod_{1\le j\le k}\lpa{e(z)^j-1}\simeq 
    c\rho^nn^{n+\frac12}.
\end{align}
\begin{table}[!ht]
\begin{center}
\begin{tabular}{ccccc}\hline
OEIS & $d(z)$ & $e(z)$ & $(d_1,e_1,e_2)$ & $(c,\rho)$ \\ \hline  
\href{https://oeis.org/A158690}{A158690} 
& $1$ & $e^z$ & $(0,1,\frac12)$ & $\lpa{\frac{12}
{\pi^{3/2}},\frac{12}{e\pi^2}}$ \\ 
\href{https://oeis.org/A196194}{A196194} 
& $\frac z{e^z-1}$ & $e^z$ & $\lpa{-\frac12,1,\frac12}$
& $\lpa{\frac{6\sqrt{2}}{\pi^{3/2}},\frac{12}{e\pi^2}}$ \\
\href{https://oeis.org/A207214}{A207214} 
& $e^z$ & $e^z$ & $(1,1,\frac12)$
& $\lpa{\frac{24}{\pi^{3/2}},\frac{12}{e\pi^2}}$ \\
\href{https://oeis.org/A207386}{A207386} 
& $1$ & $\frac{1+z}{1+z^3}$ & $(0,1,0)$ & 
$\lpa{\frac{12}{\pi^{3/2}}\,e^{-\frac{\pi^2}{24}}, 
\frac{12}{e\pi^2}}$ \\
\href{https://oeis.org/A207397}{A207397} 
& $1$ & $\frac{1+z}{1+z^2}$ & $(0,1,-1)$ &
$\lpa{\frac{12}{\pi^{3/2}}\,e^{-\frac{\pi^2}{8}},
\frac{12}{e\pi^2}}$ \\
\href{https://oeis.org/A207556}{A207556} 
& $1+z$ & $1+z$ & $(1,1,0)$ 
& $\lpa{\frac{24}{\pi^{3/2}}\,
e^{-\frac{\pi^2}{24}},\frac{12}{e\pi^2}}$ \\
\hline
\end{tabular}    
\end{center}
\medskip
\caption{Some OEIS examples with $(\alpha,\omega):=(1,0)$; they all 
satisfy the asymptotic pattern \eqref{de-ais1} with $(c,\rho)$ given 
in the last column. All $\rho$'s are the same because $e'(0)=1$.}
\label{T:4}
\end{table}

Note that the Taylor expansions of $e(z)$ in the two cases
\href{https://oeis.org/A207386}{A207386} and
\href{https://oeis.org/A207397}{A207397} of Table \ref{T:4} both
contain negative coefficients.


\subsubsection{Minor variants}

Consider the following sequence
(\href{https://oeis.org/A207652}{A207652}) whose generating function
does not have the same pattern \eqref{E:sum-d-prod-c}; yet this
sequence has the same leading order asymptotics as
\href{https://oeis.org/A179525}{A179525} (see Table~\ref{T:lrf}):
\[
    [z^n]\sum_{k\ge0}\prod_{1\le j\le k}
	\frac{(1+z)^j-1}{1-z^j}
	\simeq c\rho^n n^{n+\frac12},
	\with
	(c,\rho):=\lpa{\tfrac{12}{\pi^{3/2}}
	e^{-\frac{\pi^2}{24}}, \tfrac{12}{e\pi^2}}.
\]
This is because the extra product
\begin{equation}\label{E:int-part}
    \prod_{1\le j\le k}\frac1{1-z^j}
	= 1+z+O(|z|^2)
\end{equation}
is asymptotically negligible when $z\asymp n^{-1}$. Similarly, the
sequence \href{https://oeis.org/A207653}{A207653} satisfies
\begin{align}\label{E:A207653}
   [z^n]\sum_{k\ge0}\prod_{1\le j\le k}
	\frac{1-(1-z)^{2j-1}}{1-z^{2j-1}}
	\simeq c\rho^n n^{n+\frac12},
	\with
	(c,\rho):=\lpa{\tfrac{12}{\pi^{3/2}}
	e^{\frac{\pi^2}{24}}, \tfrac{12}{e\pi^2}}.
\end{align}
which has the same leading-order asymptotics as
\href{https://oeis.org/A158691}{A158691}. 

Another example is \href{https://oeis.org/A207434}{A207434}, which is 
defined by 
\[
    b_n := n[z^n]\log\llpa{\sum_{k\ge0}\prod_{1\le j\le k}
	\lpa{(1+z)^j-1}}.
\]
This is not of our format \eqref{E:sum-d-prod-c} but the leading
asymptotics can be quickly linked to that of 
\href{https://oeis.org/A79525}{A179525}, the number of primitive 
row-Fishburn matrices; see \eqref{E:blr}. Let $a_n := 
[z^n]\sum_{k\ge0} \prod_{1\le j\le k}\lpa{(1+z)^j-1}$. By the relation
\[
    b_n = na_n - \sum_{1\le j<n}b_j a_{n-j}
	\qquad(n\ge1),
\]
and the factorial growth of the coefficients in \eqref{E:blr}, we 
then deduce that 
\[
    b_n \simeq na_n
	\simeq c\rho^n n^{n+\frac32},
    \with  (c,\rho):=\lpa{\tfrac{12}{\pi^{3/2}}\,
	e^{-\frac{\pi^2}{24}}, \tfrac{12}{e\pi^2}}.
\]

\subsubsection{Recursive variants}

Consider first the sequence \href{https://oeis.org/A86737}{A186737} 
whose generating function is defined recursively by 
\begin{align*}
	f(z)=\sum_{k\ge0}
	\prod_{1\le j\le k}\lpa{(1+zf(z))^j-1}
    =1+z+3z^2+14z^3+82z^4+563z^5+\cdots.
\end{align*}
This is close to the framework \eqref{E:sum-d-prod-c}. While
Theorem~\ref{T:gen} does not apply, the proof there does. More
precisely, we truncate first all terms in the Taylor expansion of $f$
with powers $k>n$, so that the resulting series becomes a polynomial,
and then perform the change of variables $e^y=f(z)$. As a result, the
local expansion of the solution is given by
\begin{align*}
	z=y-\tfrac52y^2-\tfrac76y^3-\tfrac{65}{24}y^4+\cdots,
\end{align*}
and the remaining analysis follows the same procedure as the proof of 
Theorem~\ref{T:gen}, yielding 
\begin{align*}
	[z^n]\sum_{k\ge0}
	\prod_{1\le j\le k}((1+zf(z))^j-1)
	\simeq c \rho^n n^{n+\frac12},
	\with 
	(c,\rho) := \lpa{\tfrac{12}{\pi^{3/2}}
	e^{\frac{\pi^2}{24}}, \tfrac{12}{e\pi^2}},
\end{align*}
which is consistent with the expression derived by Kot\v e\v sovec on 
the OEIS page. 

Similarly, the sequence \href{https://oeis.org/A224885}{A224885} 
defined as the coefficients of the generating function
\[
    f(z) 
	= 1+z+\sum_{k\ge2}\prod_{1\le j\le k} \lpa{f(z)^j-1}
	= 1+z+2z^2+15z^3+143z^4+1552z^5+\cdots
\]
satisfies
\[
   [z^n]f(z) \simeq c\rho^n n^{n+\frac12},
    \with (c,\rho)
    = \lpa{\tfrac{12}{\pi^{3/2}}\,e^{\frac{\pi^2}{8}},
    \tfrac{12}{e\pi^2}}.
\]

\subsection{$\Lambda$-Fishburn matrices and examples with $(\alpha,\omega):=(2,0)$}
\label{S:A2}

We now consider the case when $(\alpha,\omega):=(2,0)$, beginning 
with the asymptotics of $\Lambda$-Fishburn matrices. 

\subsubsection{$\Lambda$-Fishburn matrices}

\begin{corollary}\label{C:2}
Let $\Lambda$ be a multiset of nonnegative integers with the 
generating function $\Lambda(z)$ defined as in \eqref{E:Lambda-z}. If 
$\lambda_1>0$, then the number of Fishburn matrices of size $n$ 
satisfies 
\begin{align*}
    [z^n] \sum_{k\ge0}\prod_{1\le j\le k}
	\lpa{1-\Lambda(z)^{-j}}
    \simeq c\rho^n n^{n+1} \with 
    (c,\rho):=\Lpa{\tfrac{12\sqrt{6}}{\pi^{2}}
    e^{\frac{\pi^2}{6}
    \lpa{\frac{\lambda_2}{\lambda_1^2}-\frac{1}{2}}},
    \tfrac{6\,\lambda_1}{e\pi^2}}.
\end{align*}
\end{corollary}
\begin{proof}
Use \eqref{E:gf-lf} and then apply Theorem~\ref{T:gen} with 
$d(z)=e(z)=\Lambda(z)$ and $\alpha=2$. 
\end{proof} 

A few OEIS examples to which this corollary applies are collected in 
Table~\ref{T:lf}. 
\begin{table}[!ht]
\begin{center}
\begin{tabular}{cccccc}\hline
OEIS & $\Lambda$ & $\Lambda(z)$ & $(\lambda_1,\lambda_2)$
& $(c,\rho)$ \\ \hline   

\href{https://oeis.org/A022493}{A022493} 
& $\mathbb{Z}_{\ge0}$ & $\frac1{1-z}$ & $(1,1)$ &
$\lpa{\tfrac{12\sqrt{6}}{\pi^2}\,e^{\frac{\pi^2}{12}},
\tfrac{6}{e\pi^2}}$ \\ 

\href{https://oeis.org/A138265}{A138265} 
& $\{0,1\}$ & $1+z$ & $(1,0)$ 
& $\lpa{\tfrac{12\sqrt{6}}{\pi^2}\,
e^{-\frac{\pi^2}{12}},\tfrac{6}{e\pi^2}}$  \\

\href{https://oeis.org/A289317}{A289317} 
& $\{0\}\cup \{2k-1:k\in\mathbb{Z}^+\}$ & 
$\frac{1+z-z^2}{1-z^2}$  & $(1,0)$  &
$\lpa{\tfrac{12\sqrt{6}}{\pi^2}\,
e^{-\frac{\pi^2}{12}},\tfrac{6}{e\pi^2}}$ \\

\href{https://oeis.org/A289312}{A289312} 
& $\{0\}\cup 2\mathbb{Z}^+$ & $\frac{1+z}{1-z}$ 
& $(2,2)$ & 
$\lpa{\tfrac{12\sqrt{6}}{\pi^2},\tfrac{12}{e\pi^2}}$\\
\hline
\end{tabular}
\end{center}  
\medskip
\caption{The large-$n$ asymptotics (of the form $c\rho^n n^{n+1}$) of 
some OEIS sequences that correspond to the enumeration of 
$\Lambda$-Fishburn matrices with different $\Lambda$.} 
\label{T:lf}
\end{table}

In particular, we see from Table~\ref{T:lf} that Zagier's result 
\eqref{E:zag} for the asymptotics of Fishburn numbers corresponds to 
\href{https://oeis.org/A022493}{A022493}. Also the result for 
\href{https://oeis.org/A138265}{A138265}
improves the crude bound given in \cite{Khamis2012}; see also 
\cite{Brightwell2011}. 

Corollary~\ref{C:2} also implies the asymptotics of $r$-Fishburn 
numbers \cite{Garvan2015}:
\[
    [z^n]\sum_{k\ge0}\prod_{1\le j\le k}
	\lpa{1-(1-z)^{rj}}
	\simeq c\rho^n n^{n+1} \with 
	(c,\rho):=\Lpa{\tfrac{12\sqrt{6}}{\pi^{2}}\,
    e^{\frac{\pi^2}{12r}}, \tfrac{6r}{e\pi^2}},
\]
and applies to the sequence studied in \cite{Dukes2011} with 
$\Lambda(z) = 1+z+\cdots+z^{m-1}$, $m\ge3$.

\subsubsection{Other OEIS examples}

We discuss three other OEIS sequences with $(\alpha,\omega)=(2,0)$. 
Consider first \href{https://oeis.org/A079144}{A079144}, which
enumerates labelled interval orders on $n$ points 
\cite{Brightwell2011} with $d(z)=e(z)=e^z$, and we obtain 
\begin{equation*}
	\begin{split}
		[z^n]\sum_{k\ge0}\prod_{1\le j\le k}\lpa{1-e^{-jz}}
		&=[z^n]\sum_{k\ge0}
		e^{(k+1)z}\prod_{1\le j\le k}\lpa{e^{jz}-1}^2\\
		&\simeq c\rho^n n^{n+1},
		\with (c,\rho):=\lpa{\tfrac{12\sqrt{6}}{\pi^2}
		,\tfrac{6}{e\pi^2}}.
	\end{split}
\end{equation*}
Alternatively, \eqref{E:glaisher} provides a different proof for 
this asymptotic estimate and a finer expansion; see \cite{Zagier2001}.

Consider now \href{https://oeis.org/A207651}{A207651}, the generating 
function of this sequence is different from 
\href{https://oeis.org/A022493}{A022493}, the Fishburn 
numbers, but they satisfy the same asymptotic relation (see 
\eqref{E:zag})
\[
	[z^n]\sum_{k\ge0}\prod_{1\le j\le k}
	\frac{1-(1-z)^j}{1-z^j}\simeq c\rho^nn^{n+1},
	\with  (c,\rho):= \lpa{\tfrac{12\sqrt{6}}{\pi^2}
	e^{\frac{\pi^2}{12}},\tfrac{6}{e\pi^2}},
\]
since the additional product is again asymptotically negligible; see 
\eqref{E:int-part}. 

The last sequence is \href{https://oeis.org/A035378}{A035378}: 
\[
	[z^n]\sum_{k\ge1}
	\prod_{1\le j\le k}\lpa{1-(z-1)^{j}}
    = [z^n]\sum_{k\ge0}(z-1)^{-k-1}
    \prod_{1\le j\le k}\left(1-(z-1)^{-j}\right)^2.
\]
Theorem~\ref{T:gen} does not apply directly but our approach does by 
rewriting the GF as (by grouping the terms in pairs)
\[
    \sum_{k\ge0}\frac1{(1-z)^{2k+1}}
    \left(\frac1{1-z}\left(1+
    \frac1{(1-z)^{2k+1}}\right)^2-1\right)
    \prod_{1\le j\le 2k}\left(\frac1{(1-z)^j}-1\right)^2;
\]
we then derive the approximation 
\[
    [z^n]\sum_{k\ge1}
	\prod_{1\le j\le k}\lpa{1-(z-1)^{j}}
    \simeq c\rho^n n^{n+1},
    \with 
    (c,\rho):= \lpa{\tfrac{48\sqrt{3}}{\pi^2}
    \,e^{\frac{\pi^2}{48}},
    \tfrac{24}{e\pi^2}},
\]
consistent with that provided on the OEIS webpage of 
\href{https://oeis.org/A035378}{A035378} by 
Kot\v e\v sovec; see also \cite[Sec.~5]{Zagier2001}. 

\subsection{Examples with $\omega\ne0$}
\label{S:A3}

We gather some examples in the following table, where we use the form 
\[
    a_n :=
    [z^n]\sum_{k\ge0} d_k(z)\prod_{1\le j\le k}\lpa{e_j(z)-1},
\]
with $(d_k(z),e_j(z))$ given in the second column. 
\begin{center}
\begin{tabular}{cccc}\hline
OEIS & $(d_k(z),e_j(z))$ & 
$a_n n^{-n}\simeq $ & $(c,\rho)$ \\ \hline
\href{https://oeis.org/A215066}{A215066} 
& $(1,e^{(2j-1)z})$ & $c\rho^nn^n$ 
& $\lpa{\tfrac{2\sqrt{3}}{\pi},\tfrac{24}{e\pi^2}}$ \\
\href{https://oeis.org/A209832}{A209832} 
& $(e^{(k+1)z}, e^{(2j-1)z})$ & $c\rho^nn^n$ 
& $\lpa{\tfrac{2\sqrt{6}}{\pi},\tfrac{24}{e\pi^2}}$ \\
\href{https://oeis.org/A214687}{A214687} 
& $\lpa{e^{2kz}, e^{(2j-1)z}}$ & $c\rho^nn^n$ 
& $\lpa{\tfrac{4\sqrt{3}}{\pi},\tfrac{24}{e\pi^2}}$ \\
\href{https://oeis.org/A207569}{A207569} 
& $\lpa{1,(1+z)^{2j-1}}$ & $c\rho^nn^n$ 
& $\lpa{\tfrac{2\sqrt{3}}{\pi} e^{-\frac{\pi^2}{48}}, 
\tfrac{24}{e\pi^2}}$ \\
\href{https://oeis.org/A207570}{A207570} 
& $\lpa{1,(1+z)^{3j-2}}$ & $c\rho^nn^{n-\frac16}$
& $\Lpa{\tfrac{\Gamma(\frac23)
        3^{5/6}}{2^{1/3}\pi^{7/6}}e^{-\frac{\pi^2}{72}}, 
        \tfrac{36}{e\pi^2}}$\\
\href{https://oeis.org/A207571}{A207571} 
& $\lpa{1,(1+z)^{3j-1}}$ & $c\rho^nn^{n+\frac16}$
& $\Lpa{\tfrac{12^{2/3}}
        {\pi^{5/6}\Gamma(\frac23)}
        e^{-\frac{\pi^2}{72}}, 
        \tfrac{36}{e\pi^2}}$\\ 
\hline        
\end{tabular}    
\end{center}  
In general, 
\[
    [z^n]\sum_{k\ge0}\prod_{1\le j\le k}
    \lpa{(1+z)^{pj-s}-1}\simeq c\rho^nn^{n+\frac12-\frac sp},
\]
for $0<s<p$ (not necessarily integers), where
\[    
    (c,\rho):=\llpa{\frac{\sqrt{\pi}}
    {\Gamma\lpa{1-\frac{s}{p}}}\,
	\Lpa{\frac{\pi^2}{12}}^{\frac sp-1}
    e^{-\frac{\pi^2}{24p}}, 
    \frac{12p}{e\pi^2}}.
\]

A minor variant of \href{https://oeis.org/A207569}{A207569} with the 
same asymptotic approximation is the sequence 
\href{https://oeis.org/A207654}{A207654}:
\[
    [z^n]\sum_{k\ge0}\prod_{1\le j\le k}
    \frac{(1+z)^{2j-1}-1}{1-z^{2j-1}}
    \simeq c\rho^nn^n,
    \with 
    (c,\rho)=\lpa{\tfrac{2\sqrt{3}}{\pi}e^{-\frac{\pi^2}{48}},
    \tfrac{24}{e\pi^2}},
\]
because the extra product is asymptotically negligible; see
also \eqref{E:A207653}.

The last example is \href{https://oeis.org/A207557}{A207557}: 
$[z^n]f(z)$ with
\[
    f(z):=\sum_{k\ge0}(1+z)^{-k(k-1)}
	\prod_{1\le j\le k}\lpa{(1+z)^{2j-1}-1},
\]
which can be transformed, by the Rogers-Fine identity (see 
\cite{Fine1988}) into 
\[
    f(z) = 1+z^{-1}\sum_{k\ge1}(1+z)^{2k+1}
	\prod_{1\le j\le k}\lpa{(1+z)^{2j-1}-1}^2.
\]
We can then apply Theorem~\ref{T:gen}, and obtain  
\begin{align*}
    [z^n]f(z)
    &= [z^{n+1}]\sum_{k\ge1}(1+z)^{2k+1}
	\prod_{1\le j\le k}\lpa{(1+z)^{2j-1}-1}^2\\
    &\simeq c_0\rho^{n+1}(n+1)^{n+1}
    \simeq c_0 e \rho \rho^n n^{n+1},
    \with  (c_0,\rho):= \lpa{\tfrac{2\sqrt{6}}{\pi}
	e^{-\frac{\pi^2}{24}}, \tfrac{12}{e\pi^2}}.
\end{align*}
Thus $c_0e\rho = \frac{24\sqrt{6}}{\pi^3}e^{-\frac{\pi^2}{24}}$,
consistent with the expression derived by Kot\v e\v sovec in the 
OEIS; see \cite[\href{https://oeis.org/A207557}{A207557}]{oeis2019}.

\section{Applications II. Bivariate asymptotics (asymptotic distributions)}\label{S:B}

We derive in this section the various limit laws arising from the
sizes of the first row and the diagonal, as well as the number of
$1$s in random Fishburn and row-Fishburn matrices, assuming that
all matrices of the same size are equally likely to be selected. We
begin with row-Fishburn matrices because they are technically simpler.

\subsection{Statistics on $\Lambda$-row-Fishburn matrices}
\label{S:B1}

By Proposition~\ref{P:gf-f}, the number of $\Lambda$-row-Fishburn 
matrices of size $n$ is given by (see \eqref{E:gf-lrf})
\[
    a_n := [z^n]\sum_{k\ge0}\prod_{1\le j\le k} 
    \lpa{\Lambda(z)^j-1},
\]
where $\Lambda(z)$ is the generating function of the multiset 
$\Lambda$; see \eqref{E:Lz}. The asymptotics of $a_n$ is already 
examined in Corollary~\ref{C:1}. 

Recall that the probability generating function of a Poisson 
distribution with mean $\tau>0$ is given by $e^{\tau(v-1)}$, while 
that of a zero-truncated Poisson (ZTP) distribution with parameter 
$\tau$ by 
\[
	\frac{e^{\tau v}-1}{e^\tau-1}
\]
whose mean and variance equal
\[
    \frac{\tau e^\tau}{e^\tau-1}
	\and 
	\frac{\tau e^\tau(e^\tau-1-\tau)}{(e^\tau-1)^2},
\]
respectively. When $\tau:=\log 2$, these become $2\log 2$ and $2(\log
2)(1-\log 2)$, respectively. Also $\mathscr{N}(0,1)$ denotes the
standard normal distribution. The notation $X_n\stackrel{d}{\to}X$
means convergence in distribution.

\subsubsection{Limit theorems}

\begin{theorem} \label{T:lrf-stat}
Assume $\lambda_1>0$ and that all $\Lambda$-row-Fishburn matrices of
size $n$ are equally likely to be selected. Then in a random matrix,
\begin{enumerate}[(i)]
    \item the size $X_n$ of the first row is distributed
    asymptotically as zero-truncated Poisson with parameter
    $\log 2$:
    \[
        X_n \stackrel{d}{\to} \mathrm{ZTP}(\log 2),
    \]
    
    \item the size $Y_n$ of the diagonal (or the last column) is
    asymptotically normally distributed with mean and variance both
    asymptotic to $\log n$,
    \begin{equation}\label{E:Yn-clt}
        \frac{Y_n-\log n}{\sqrt{\log n}} 
        \stackrel{d}{\to} \mathscr{N}(0,1), \and 
    \end{equation} 
    \item for the number $Z_n$ of $1$s, if $\lambda_2>0$, then 
	    \begin{align}\label{E:Zn-Poi}
	        \frac{n-Z_n}2 \stackrel{d}{\to} 
            \mathrm{Poisson}(\tau)\,\,
			\with \tau:=\frac{\lambda_2\pi^2}{12\lambda_1^2},
	    \end{align}
    and $\mathbb{P}(Z_n=n)\to1$ if $\lambda_2=0$. 
\end{enumerate}
\end{theorem}
For the diagonal size, we can also express the asymptotic 
distribution as $Y_n \sim \text{Poisson}(\log n)$, which implies  
\eqref{E:Yn-clt}. Finer approximations are given in 
\eqref{E:Xn-clt} and \eqref{E:Xn-mv}.

\begin{proof}
\begin{enumerate}[(i)]

\item For the first row size $X_n$, we begin with the generating 
function (see \eqref{E:gf-fr})
\[
    f_X(z,v) := \sum_{k\ge0}\lpa{\Lambda(vz)^{k+1}-1}
    \prod_{1\le j\le k}\lpa{\Lambda(z)^j-1}.
\]
By applying \eqref{E:gen} to $(d(z),e(z)):=(\Lambda(vz),\Lambda(z))$ 
and to $(d(z),e(z)):=(1,\Lambda(z))$, we deduce that 
\[
    \mathbb{E}\lpa{v^{X_n}}
    := \frac{[z^n]f_X(z,v)}{a_n}
    \simeq 2^v-1,
\]
for $v=O(1)$. This asymptotic estimate holds \emph{a priori}
pointwise for each finite $v$, but the same proof there gives indeed 
the uniformity of the error term in $v$ when $v=O(1)$. This implies 
the convergence in distribution to the zero-truncated Poisson (ZTP) 
law with parameter $\log 2$.

\item Consider now the generating polynomial for the diagonal size 
$Y_n$
\[
    [z^n]f_Y(z,v)
    :=[z^n]\sum_{k\ge1}\prod_{1\le j\le k}
    \lpa{\Lambda(vz)\Lambda(z)^{j-1}-1}.
\]
The generating function is not of the form \eqref{E:sum-d-prod-c}, 
but observe that
\[
    \Lambda(vz) = \Lambda(z)^v \lpa{1+O(|z|^2)},
\]
when $|z|$ is small. Then, when $k\asymp n$ and $|z|\asymp n^{-1}$ 
(taking logarithm and estimating the sum of errors), we have
\begin{align}\label{E:log-es}
    \prod_{1\le j\le k}\lpa{\Lambda(vz)\Lambda(z)^{j-1}-1}
    = \llpa{\prod_{1\le j\le k}
    \lpa{\Lambda(z)^{j+v-1}-1}}\lpa{1+O\lpa{|z|\log k}},
\end{align}
and we are in a position to apply Theorem~\ref{T:gen}, giving
\begin{align*}
    [z^n]f_Y(z,v)
    = c(v)\rho^n n^{n+v-\frac12}\lpa{1+O\lpa{n^{-1}\log n}},
\end{align*}
where
\begin{align*}
    (c(v),\rho):= \llpa{\frac{\sqrt{\pi}}
    {\Gamma(v)}\Lpa{\frac{12\lambda_1}{\pi^2}}^v
    e^{\frac{\pi^2}{12}
    \lpa{\frac{\lambda_2}{\lambda_1^2}-\frac12}},
    \frac{12\lambda_1}{e\pi^2}},
\end{align*}
uniformly for $v=O(1)$. Accordingly, the probability generating 
function of $Y_n$ satisfies
\begin{align*}
    \mathbb{E}\lpa{v^{Y_n}}
    =\frac{[z^n]f(z,v)}{a_n}
    =\frac1{\Gamma(v)}
    \left(\frac{12}{\pi^2}\right)^{v-1}
    \, e^{(v-1)\log n}\lpa{1+O\lpa{n^{-1}\log n}},
\end{align*}
uniformly for $v=O(1)$. This is of the form of Quasi-Powers (see
\cite{Flajolet2009,Hwang1994}), and we then deduce the asymptotic 
normality of $Y_n$ with optimal convergence rate:
\begin{align}\label{E:Xn-clt}
    \sup_{x\in\mathbb{R}}\biggl|\mathbb{P}
    \Lpa{\frac{Y_n-\log n}
    {\sqrt{\log n}}\le x}-\Phi(x) \biggr|
    = O\lpa{(\log n)^{-\frac12}},
\end{align}
together with the asymptotic approximations to the mean and the
variance:
\begin{equation} \label{E:Xn-mv}
\begin{split}
    \mathbb{E}(Y_n)
    &= \log n + \gamma+\log\tfrac{12}{\pi^2}
    +O\lpa{n^{-1}\log n},\\
    \mathbb{V}(Y_n)
    &= \log n + \gamma-\tfrac{\pi^2}{6}
    +\log\tfrac{12}{\pi^2}+O\lpa{n^{-1}(\log n)^2},
\end{split}
\end{equation}
where $\gamma$ denotes the Euler-Mascheroni constant and $\Phi(x)$
denotes the distribution function of the standard normal
distribution. For other types of Poisson approximation, see 
\cite{Hwang1999}.

\item Applying the same proof of Theorem~\ref{T:gen} to the 
generating function \eqref{E:gf-1} for the number of $1$s gives
\[
    [z^n]\sum_{k\ge1}\prod_{1\le j\le k} 
    \lpa{(\Lambda(z)+\lambda_1(v-1)z)^{j}-1}
    \simeq c(v)v^n \rho^n n^{n+\frac12},
\]
uniformly for bounded $v$, where $(c(v),\rho) 
    := \Lpa{\tfrac{12}
    {\pi^{3/2}}\,e^{\frac{\pi^2}{12}
    \lpa{\frac{\lambda_2}{\lambda_1^2v^2}-\frac12}},
	\tfrac{12\lambda_1v}{e\pi^2}}$.
This implies that if $\lambda_2>0$, then 
\begin{equation}\label{E:Zn-Poi2}
    \mathbb{E}\lpa{v^{\frac12(n-Z_n)}}
    \simeq e^{\tau(v-1)},\with 
	\tau := \frac{\pi^2\lambda_2}
    {12\lambda_1^2},
\end{equation}
and we then obtain the limit Poisson distribution with parameter 
$\tau$. If $\lambda_2=0$, then $\mathbb{E}\lpa{v^{n-Z_n}}\to 1$,
a Dirac distribution. Furthermore, by the uniformity of 
\eqref{E:Zn-Poi2} and Cauchy's integral representation, we obtain 
\eqref{E:Zn-Poi}; see \cite{Hwang1994}.

Similarly, for the number $Z_n^{[2]}$ of $2$s, we use the generating 
function
\[
    \sum_{k\ge1}\prod_{1\le j\le k} 
    \lpa{\lpa{\Lambda(z)+\lambda_2(v-1)z^2}^{j}-1},
\]
and deduce that $\mathbb{E}\lpa{v^{Z_n^{[2]}}}\simeq e^{\tau(v-1)}$,
with the same $\tau$ as in \eqref{E:Zn-Poi}. 

\end{enumerate}
\end{proof}
Stronger results such as local limit theorems can also be derived; 
see \cite{Hwang1994} for more information.

\subsubsection{Applications}

Consider first the case of primitive row-Fishburn matrices with
$\Lambda=\{0,1\}$. Then by Theorem~\ref{T:lrf-stat}, we see that 
in a random primitive Fishburn matrix the first row size is
asymptotically ZTP$(\log 2)$ distributed, the diagonal is
asymptotically normal, while the number of $1$ is obviously the same
as the size of the matrix. In particular, the distribution of the
diagonal size corresponds to sequence 
\href{https://oeis.org/A182319}{A182319}.


On the other hand, when $\Lambda(z) := \frac1{1-z}$, we have very 
similar behaviors for the sizes of the first row and the diagonal, 
but the number $Z_n$ of $1$s is asymptotically Poisson:
\[
	\mathbb{P}(n-Z_n=2k) \to \tfrac{\tau^k}{k!}\,e^{-\tau},
	\with  \tau = \tfrac{\pi^2}{12} \, \mbox{ for } k=0,1,\dots.
\]

\begin{figure}[!ht]
\begin{center}
\begin{tabular}{ccc}
    \includegraphics[height=3cm]{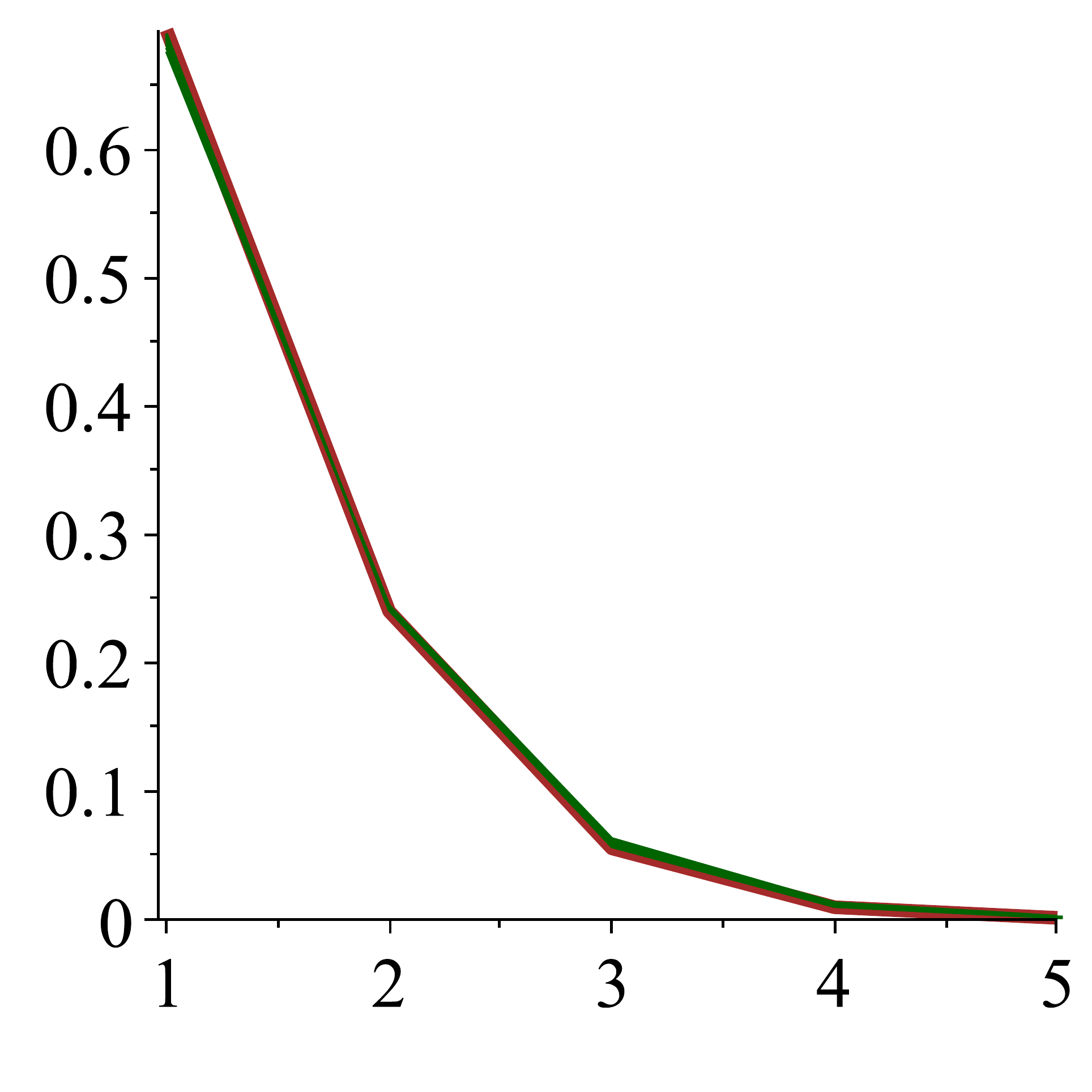}
    & \includegraphics[height=3cm]{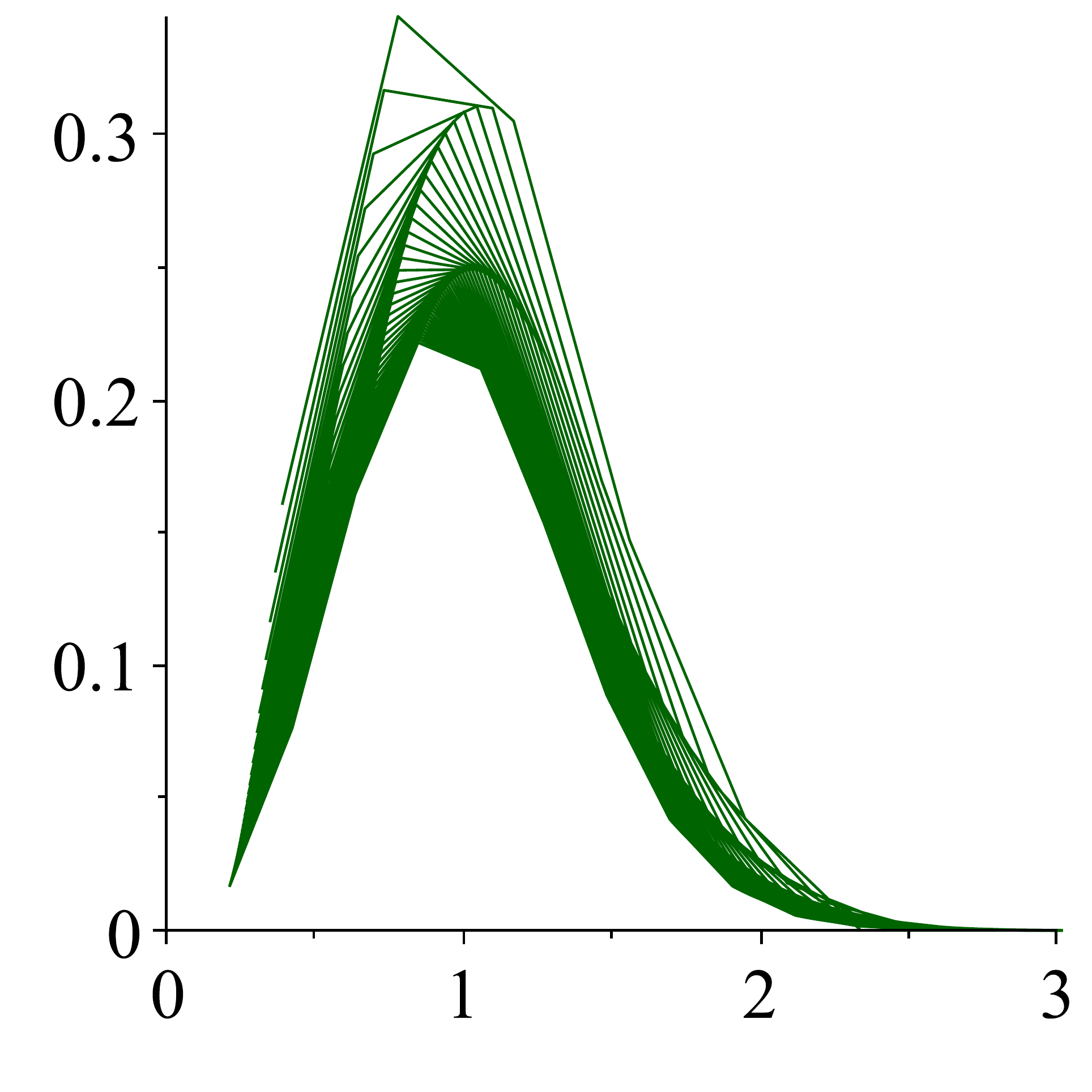} 
    & \includegraphics[height=3cm]{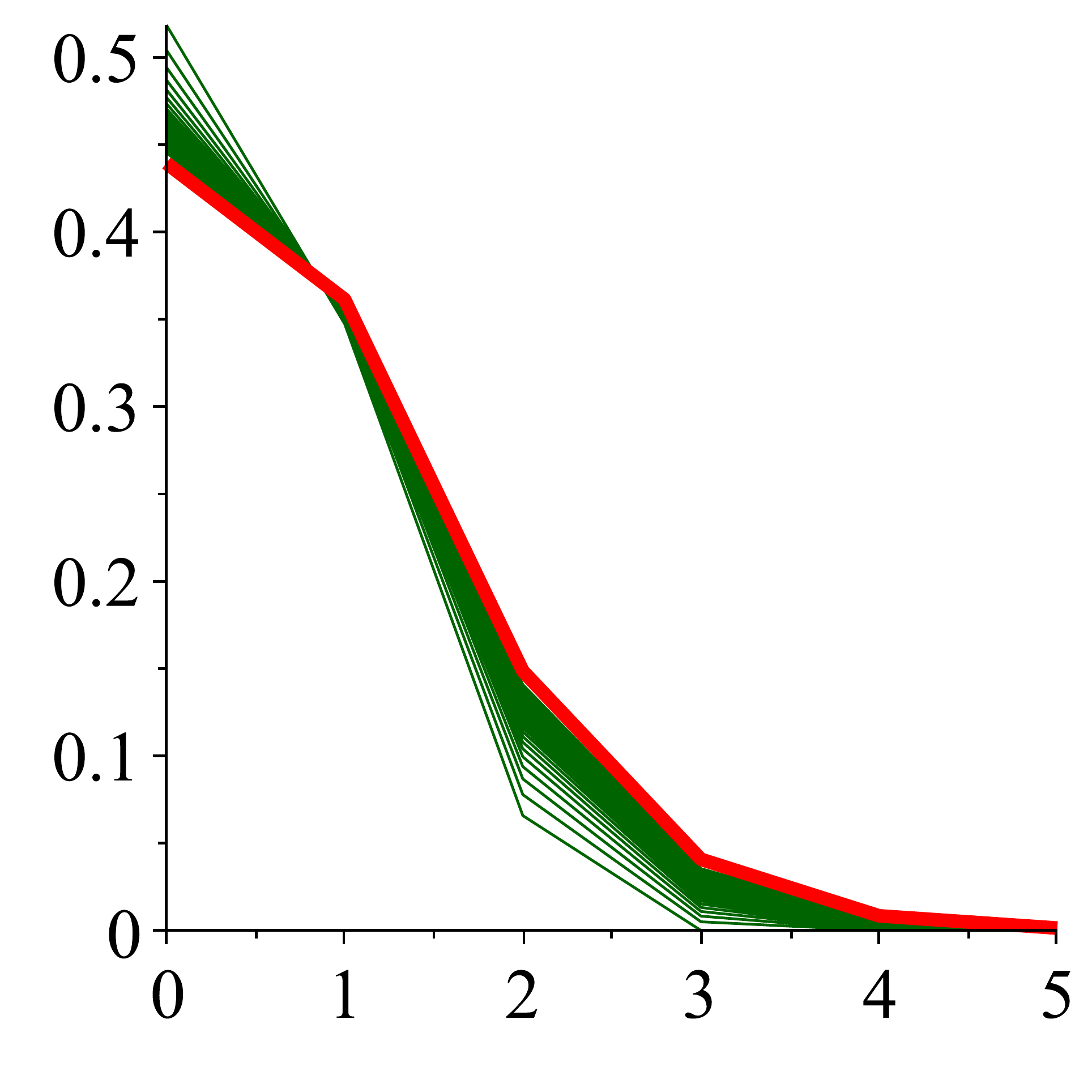} 
    \\
\end{tabular}
\end{center}
\caption{The histograms of $X_n$, $Y_n$ and $Z_n$ in the case of
row-Fishburn matrices ($\Lambda(z)=\frac1{1-z}$) for $n=6,\dots,50$:
$\mathbb{P}(X_n=k)$ (left), $\mathbb{P}(Y_n=\tr{t \mu_n})$ (middle),
and $\mathbb{P}(n-Z_n=2k)$ (right), where $\mu_n=\mathbb{E}(Y_n)$.
Then tendence to ZTP, normal and Poisson is visible in each case, as
well as the corresponding convergence rate.}
\end{figure}

\subsection{Statistics on Fishburn matrices} 
\label{S:B2}

We consider random $\Lambda$-Fishburn matrices in this subsection. By
Proposition~\ref{P:gf-f}, the number of $\Lambda$-Fishburn matrices
of size $n$ is given by (see \eqref{E:gf-lf})
\[
    a_n := [z^n]\sum_{k\ge0}\prod_{1\le j\le k} 
    \lpa{1-\Lambda(z)^{-j}},
\]
and an asymptotic approximation is already derived in 
Corollary~\ref{C:2}.

\subsubsection{Limit theorems}

\begin{theorem} \label{T:lf-stat}
Assume $\lambda_1>0$ and that all $\Lambda$-Fishburn matrices of size
$n$ are equally likely to be selected. Then in a random matrix, the
size $X_n$ of the first row (or the last column) and the diagonal
size $Y_n$ are both asymptotically normally distributed with
logarithmic mean and variance in the following sense
\begin{equation}\label{E:Yn-f-clt}
    \frac{X_n-\log n}{\sqrt{\log n}} 
    \stackrel{d}{\to} \mathscr{N}(0,1),
    \and 
    \frac{Y_n-2\log n}{\sqrt{2\log n}} 
    \stackrel{d}{\to} \mathscr{N}(0,1),
\end{equation}
and if $\lambda_2>0$, then the number $Z_n$ of $1$s is asymptotically 
Poisson distributed
\begin{align}\label{E:Znf-Poi}
    \frac{n-Z_n}2 \stackrel{d}{\to} 
    \mathrm{Poisson}(\tau)
	\with 
    \tau:=\frac{\lambda_2\pi^2}{6\lambda_1^2},
\end{align}
otherwise, $\lambda_2=0$ implies that $\mathbb{P}(Z_n=n)\to1$.
\end{theorem}

\begin{proof}
\begin{enumerate}[(i)]

\item We begin with the generating function (see \eqref{E:gf-f-fr}) 
for the first row size
\[
    f_X(z,v) := \Lambda(vz)\sum_{k\ge 0}\Lambda(z)^k
	\prod_{1\le j\le k}\lpa{\lpa{\Lambda(vz)\Lambda(z)^{j-1}-1}
	\lpa{\Lambda(z)^j-1}}.
\]
By \eqref{E:log-es} and the expansion $\Lambda(vz)=1+O(|z|)$ for 
small $|z|$, we have 
\[
    f_X(z,v) = \llpa{\sum_{k\ge 0}\Lambda(z)^k
	\prod_{1\le j\le k}\lpa{\lpa{\Lambda(z)^{j+v-1}-1}
	\lpa{\Lambda(z)^j-1}}}
    \lpa{1+O(|z|\log n)},
\]
when $|z|\asymp n^{-1}$. 

Similar to Theorem~\ref{T:gen}, we first derive, by the same methods 
of proof used in Proposition~\ref{P:eoza}, that 
\begin{align}\label{E:frwk-dr}
    [z^n]\sum_{k\ge0}e^{kz}\prod_{1\le j\le k}
    \lpa{e^{(j+\omega)z}-1}\lpa{e^{jz}-1}
    \simeq c_0(\omega)\rho^n n^{n+\omega+1},
\end{align}
where
\[
    (c_0(\omega),\rho)
    := \llpa{\frac{2\sqrt{6}}
    {\Gamma(1+\omega)}
    \Lpa{\frac{6}{\pi^2}}^{1+\omega},
    \frac{6}{e\pi^2}}.
\]
Briefly, $\alpha$ is almost $2$ in the proof of
Proposition~\ref{P:eoza}, and the largest terms occur when $k\sim \mu
n$ and $n|z|\sim \xi$ with $(\mu,\xi)$ as in \eqref{E:mu-xi}, so that
$e^{kz}$ contributes an extra factor $2$.

We now make the change of variables $\Lambda(z) = e^y$, and follow 
the same proof procedure of Theorem~\ref{T:gen}, yielding 
\begin{align*}
    [z^n]\sum_{k\ge 0}\Lambda(z)^k
	\prod_{1\le j\le k}\lpa{\lpa{\Lambda(z)^{j+v-1}-1}
	\lpa{\Lambda(z)^j-1}}
    \simeq c(v)\rho^n n^{n+v},
\end{align*}
where
\[
    (c(v),\rho)
    := \llpa{\frac{2\sqrt{6}}
    {\Gamma(v)}\Lpa{\frac{6}{\pi^2}}^{v}
    \,e^{\frac{\pi^2}{12}
    \lpa{\frac{\lambda_2}{\lambda_1^2}-\frac12}},
    \frac{6\lambda_1}{e\pi^2}}.
\]
We then deduce that 
\[
    \mathbb{E}\lpa{v^{X_n}}
    = \frac1{\Gamma(v)}\Lpa{\frac{6}{\pi^2}}^{v-1}
    \,n^{v-1}\lpa{1+O\lpa{n^{-1}\log n}},
\]
uniformly for $v=O(1)$, and the asymptotic normality of $X_n$ then
follows again from the Quasi-Powers theorem \cite{Flajolet2009,
Hwang1994} or a standard characteristic function argument. Finer
results such as \eqref{E:Xn-clt} and \eqref{E:Xn-mv} can also be
derived.

\item For the size of the diagonal $Y_n$, we now have the generating 
function (see \eqref{E:gf-f-dg})
\[
    f_Y(z,v) := \Lambda(vz)
    \sum_{k\ge 0}\Lambda(z)^{k}
    \prod_{1\le j\le k}\lpa{\Lambda(vz)\Lambda(z)^{j-1}-1}^2.
\]
By \eqref{E:log-es}, the same arguments used in (i) for $X_n$ and 
Theorem~\ref{T:gen}, we deduce that 
\[
    [z^n]f_Y(z,v)
    = c(v)\rho^n n^{n+2v-1}\lpa{1+O\lpa{n^{-1}\log n}},
\]
where
\[
    (c(v),\rho) := \llpa{\frac{2\sqrt{6}}
    {\Gamma(v)^2}\Lpa{\frac{6}{\pi^2}}^{2v-1}
    \,e^{\frac{\pi^2}{12}
    \lpa{\frac{\lambda_2}{\lambda_1^2}-\frac12}},
    \frac{6\lambda_1}{e\pi^2}}.
\]
It follows that 
\[
    \mathbb{E}\lpa{v^{Y_n}}
    = \frac1{\Gamma(v)^2}\Lpa{\frac{6}{\pi^2}}^{2(v-1)}
    \,n^{2(v-1)}\lpa{1+O\lpa{n^{-1}\log n}},
\]
uniformly for $v=O(1)$. The asymptotic normality then follows from 
Quasi-Powers Theorem. 

\item Since $\lambda_1>0$, we can apply Theorem~\ref{T:gen} to the 
generating function \eqref{E:gf-f-1} for the number $Z_n$ of $1$s, 
which is 
\[
    f_Z(z,v):= \sum_{k\ge0}(\Lambda(z)+\lambda_1(v-1)z)^{k+1}
    \prod_{1\le j\le k}
    \lpa{(\Lambda(z)+\lambda_1(v-1)z)^j-1}^2,
\]
and we deduce that $\mathbb{E}\lpa{v^{\frac12(n-Z_n)}} \simeq
e^{\tau(v-1)}$, where $\tau := \tfrac{\pi^2\lambda_2}{6\lambda_1^2}$,
which leads to a degenerate limit law when $\lambda_2=0$ and a
Poisson limit law otherwise. The number of $2$s follows the same law.

\end{enumerate}    
\end{proof}

The most widely studied parameter is the size $X_n$ of the first row of uniformly random Fishburn matrices, i.e., the case $\Lambda(z) := \frac1{1-z}$. It appeared
in Stoimenow's study \cite{Stoimenow1998} on chord diagrams, and 
later examined by Zagier in \cite{Zagier2001}. Then the limiting
distribution of $X_n$ was raised as an open question in
\cite{Bringmann2014, Jelinek2012}. The generating function $f_X(z,v)$
for the first row size has been derived in several papers; see, for
example, \cite{Andrews2014, Bousquet-Melou2010, Fu2020, Jelinek2012,
Yan2011}, \href{https://oeis.org/A175579}{A175579} and 
Section~\ref{S:comb} for several other quantities with the same 
distribution as $X_n$. See also Table~\ref{Tab:A175579} and 
Figure~\ref{F:XY} for the distribution of small $n$ and graphical 
renderings. 

\begin{table}[!ht]\small 
\def\arraystretch{0.9}
\begin{tabular}{c|ccccccc}
    $n\backslash k$
    & $1$ & $2$ & $3$ & $4$ & $5$ & $6$ & $7$\\ \hline
    $1$ & $1$ &&&&& \\
    $2$ & $1$ & $1$ &&&& \\
    $3$ & $2$ & $2$ & $1$ &&& \\
    $4$ & $5$ & $6$ & $3$ & $1$ && \\
    $5$ & $15$ & $21$ & $12$ & $4$ & $1$ & \\
    $6$ & $53$ & $84$ & $54$ & $20$ & $5$ & $1$ \\
    $7$ & $217$ & $380$ & $270$ & $110$ & $30$ & $6$ & $1$ \\
\end{tabular}\quad
\begin{tabular}{c|ccccccc}
$n\backslash k$
& $1$ & $2$ & $3$ & $4$ & $5$ & $6$ & $7$\\ \hline
$1$ & $1$\\
$2$ & $0$ & $2$\\
$3$ & $0$ & $1$ & $4$\\
$4$ & $0$ & $2$ & $5$ & $8$\\
$5$ & $0$ & $5$ & $14$ & $18$ & $16$\\
$6$ & $0$ & $15$ & $47$ & $67$ & $56$ & $32$\\
$7$ & $0$ & $53$ & $183$ & $287$ & $267$ & $160$ & $64$\\
\end{tabular}
\medskip
\caption{The number of Fishburn matrices of size $n$ with first row 
size equal to $k$ (left) and the diagonal size to $k$ (right) for 
$n=1,\dots,7$. The table on the left corresponds to 
\href{https://oeis.org/A175579}{A175579}.}
\label{Tab:A175579}
\end{table}

The mean and the variance of $X_n$ satisfy
\begin{equation*} 
\begin{split}
    \mathbb{E}(X_n)
    &= \log n + \gamma-\log\tfrac{\pi^2}{6}
    +O\lpa{n^{-1}\log n},\\
    \mathbb{V}(X_n)
    &= \log n + \gamma-\tfrac{\pi^2}{6}
    -\log\tfrac{\pi^2}{6}+O\lpa{n^{-1}(\log n)^2}.
\end{split}
\end{equation*}

\begin{figure}[!ht]
\begin{center}
\begin{tabular}{cccc}
    \includegraphics[height=3cm]{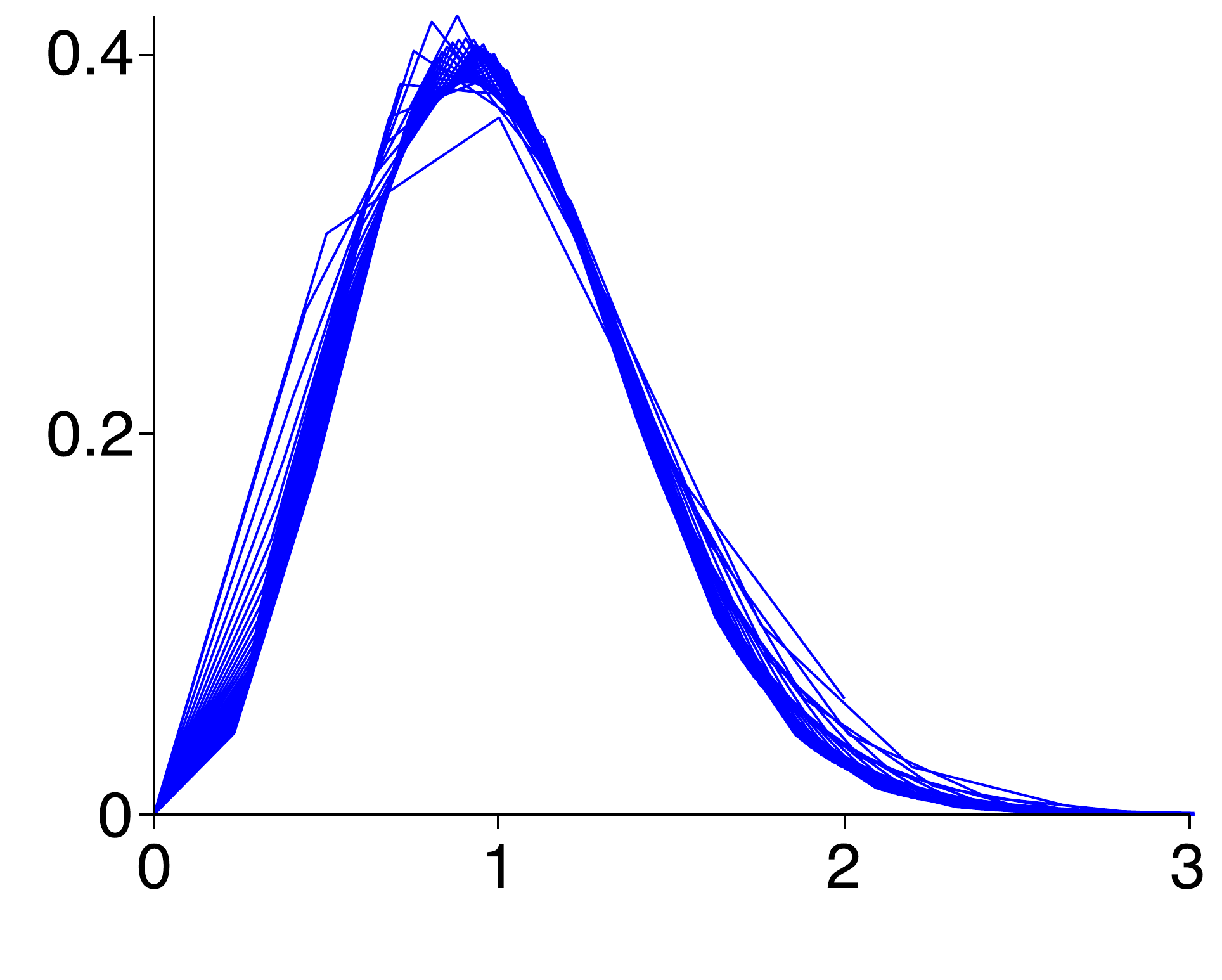}
    & \includegraphics[height=3cm]{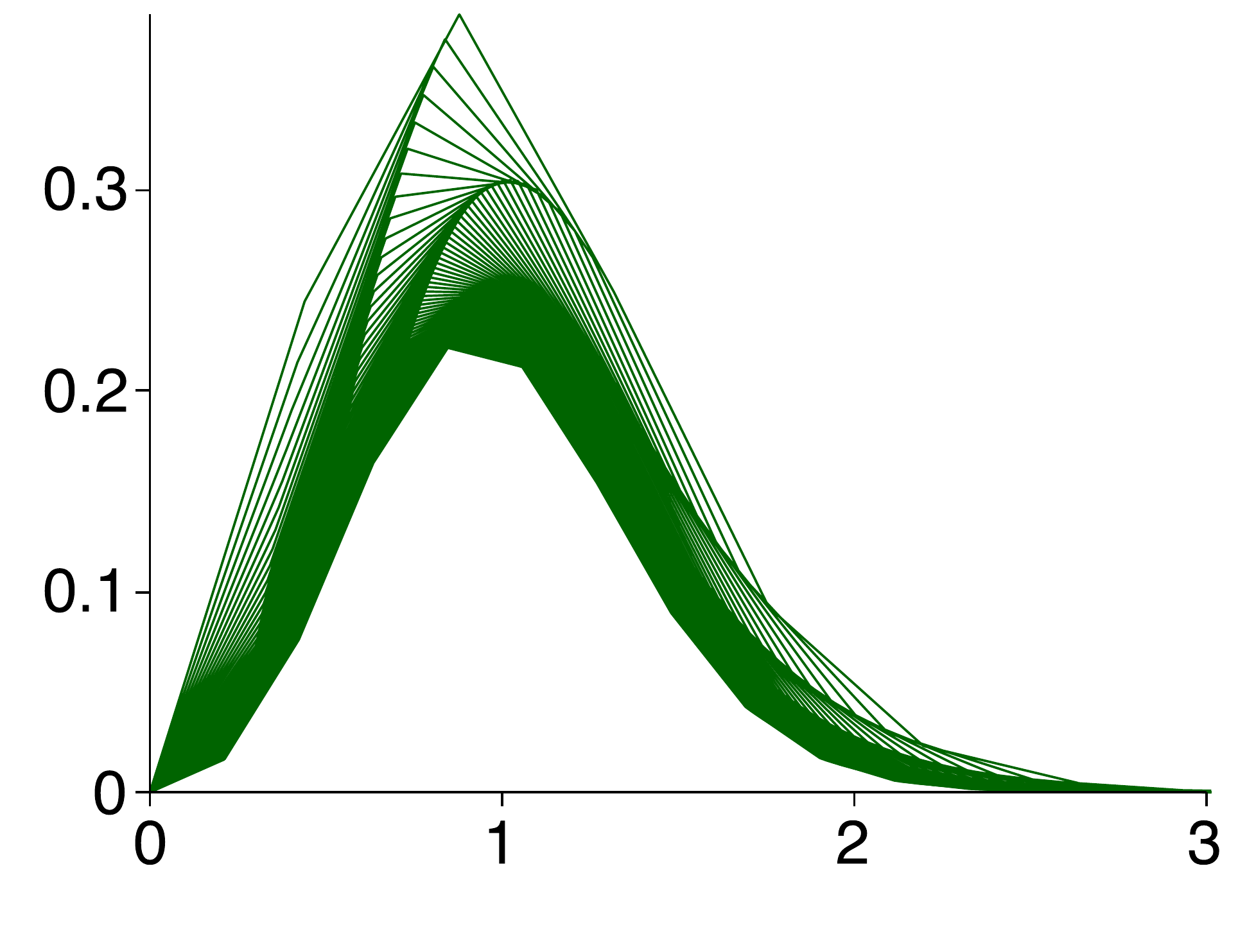} 
    & \includegraphics[height=3cm]{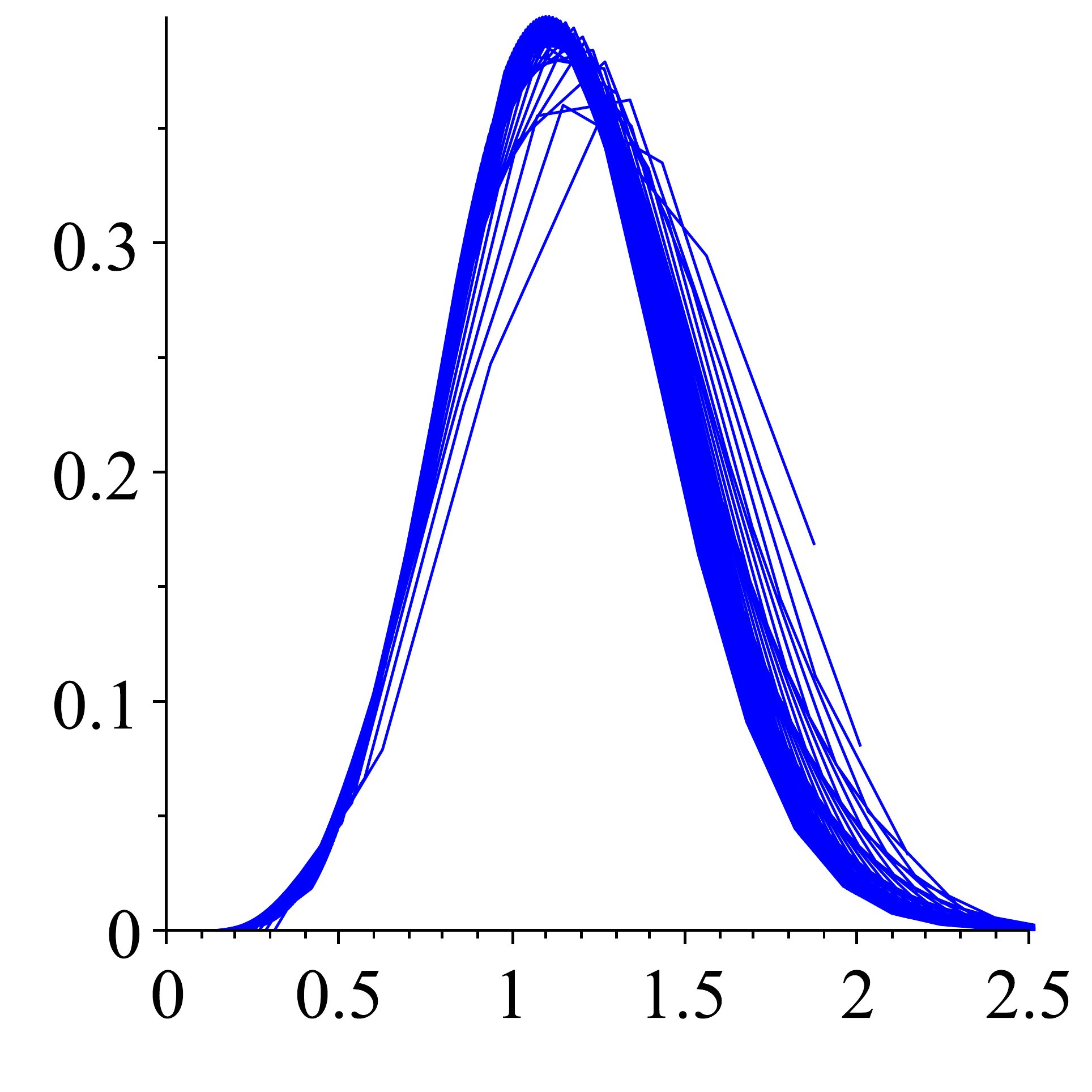}
    & \includegraphics[height=3cm]{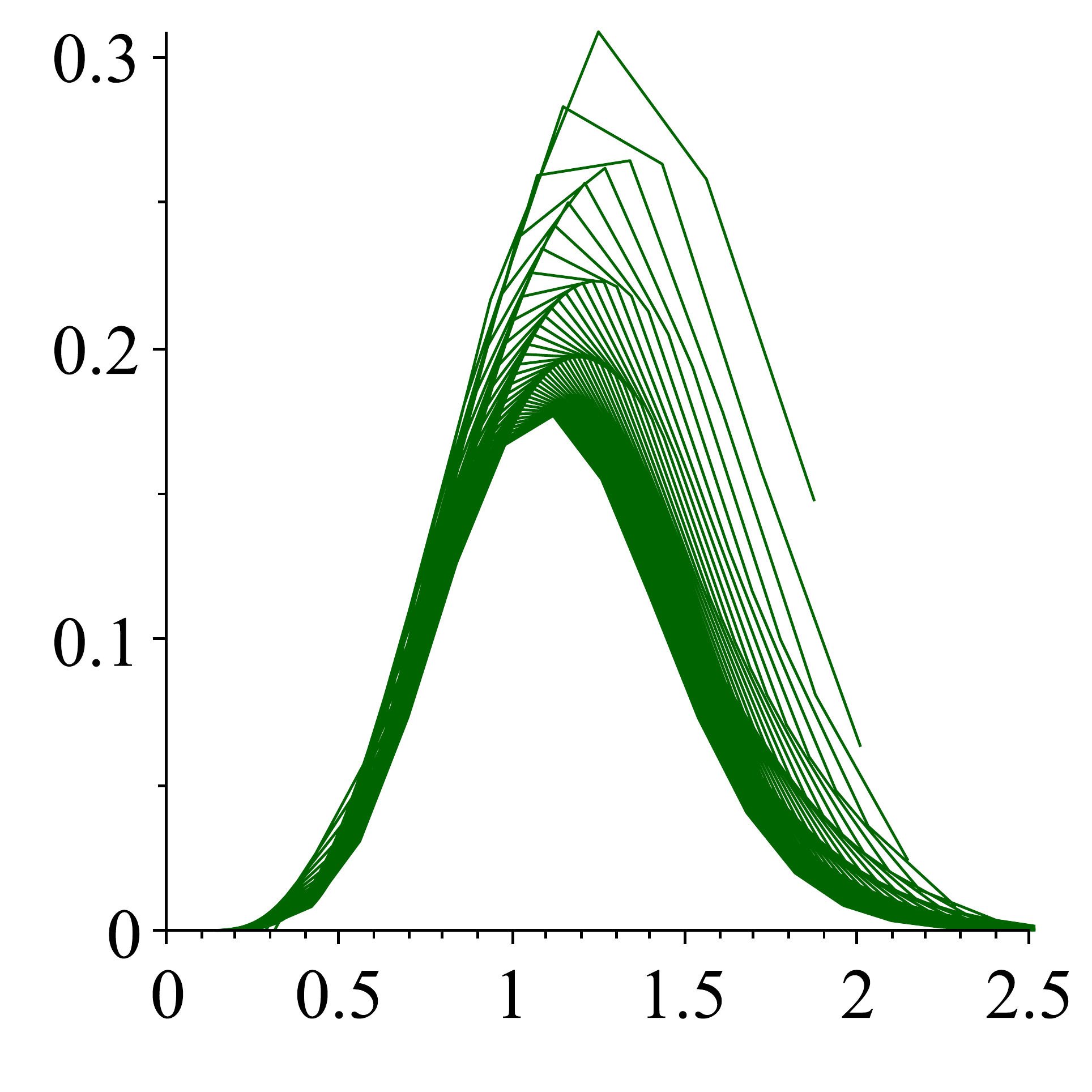}
    \\
\end{tabular}
\end{center}
\caption{The histograms of $X_n$ and $Y_n$ (Fishburn matrices) for
$n=6,\dots,100$: $\sigma_n(X)\mathbb{P}(X_n=\tr{t \mu_n(X)})$
(first), $\mathbb{P}(X_n=\tr{t \mu_n(X)})$ (second),
$\sigma_n(Y)\mathbb{P}(Y_n=\tr{t \mu_n(Y)})$ (third),
$\mathbb{P}(Y_n=\tr{t \mu_n(Y)})$ (fourth), where $\mu_n(W)$ and
$\sigma_n^2(W)$ denote the corresponding mean and variance of $W_n$,
respectively.}\label{F:XY}
\end{figure}

\section{A framework for matrices without $1$s and self-dual matrices}
\label{S:jelinek}

We discuss in this section the extension to the situation when
$e_1=0$ and $e_2>0$ of the general framework
\eqref{E:sum-d-prod-c}. The general asymptotic expressions
\eqref{E:gen} and \eqref{E:de-cre} certainly fail in such a case as
the leading constant involves $e_1$ in the denominator.

In addition to providing a better understanding of Fishburn matrices
in more general situations, our consideration of
\eqref{E:sum-d-prod-c} with $e_1=0$ and $e_2>0$ was also motivated by
asymptotic enumeration of the self-dual Fishburn matrices, a
conjecture raised by Jel\'{i}nek (Conjecture 5.4 of
\cite{Jelinek2012}). In particular, the asymptotic approximations
of non-primitive and primitive self-dual Fishburn matrices
(given in \eqref{E:selfdualfish} and \eqref{E:primselfdualfish}) will
follow readily from our general result Theorem~\ref{T:ext} or
Corollary~\ref{C:self-dual_lam1}. Furthermore, as in
Sections~\ref{S:A} and \ref{S:B}, our framework will be equally
useful in characterizing the asymptotic distributions of a few 
statistics in random self-dual Fishburn matrices, which we briefly
explore in this section.

While most proofs in this section follow similar ideas to the ones we
employed in Section \ref{S:basis}--\ref{S:B}, the technical details
in these proofs are more involved with generally lengthier
expressions. Thus we will indicate the major differences.

\subsection{Asymptotics of \eqref{E:sum-d-prod-c} with $e_1=0$ 
and $e_2>0$}\label{S:8-1}

\begin{theorem}\label{T:ext}
Assume $\alpha\in\mathbb{Z}^+$ and $\omega_0,\omega\in\mathbb{C}$. 
Given any two formal power series $e(z):=1+\sum_{j\ge1}e_jz^j$ and
$d(z):=1+\sum_{j\ge1}d_jz^j$ satisfying $e_1=0$, $e_2>0$, and 
\begin{align}\label{E:de-conds}
    \alpha e_3\pi^2+12 d_1e_2\log 2>0,
\end{align}
we have
\begin{equation}\label{E:ext8}
    [z^n]\sum_{k\ge0}d(z)^{k+\omega_0}\prod_{1\le j\le k}
    \lpa{e(z)^{j+\omega}-1}^\alpha
    = ce^{\beta\sqrt{n}}\rho^{\frac12n}
    n^{\frac12(n+\alpha)+\alpha\omega}
    \lpa{1+O\lpa{n^{-\frac12}}},
\end{equation}
the $O$-term holding uniformly for bounded $\omega_0$ and $\omega$, 
where $\beta:=\frac{\sqrt{6}d_1\log 2}{\sqrt{e_2\alpha}\,\pi}
+\frac{\sqrt{\alpha}\,e_3\pi}{2\sqrt{6}\,e_2^{3/2}}$, 
$\rho:=\frac{6e_2}{e\pi^2\alpha}$, and
\begin{align*}
    c:=\tfrac{\sqrt{3}}{\sqrt{2}\,\alpha\pi}
    \Lpa{\tfrac1{\Gamma(1+\omega)}
    \sqrt{\tfrac{12}{\alpha\pi}}
    \left(\tfrac{6}{\alpha\pi^2}\right)^\omega}^\alpha
    2^{-\frac{d_1^2}{2e_2}-\frac{3d_1e_3}{4e_2^2}+
    \frac{d_2}{e_2}}
    e^{-\frac{d_1^2}{4\alpha e_2}
    -\frac{\alpha\pi^2}{12}\lpa{\frac{7e_3^2}{8e_2^3}
    -\frac{e_4}{e_2^2}+\frac12}    
    +\frac{3d_1^2}{2e_2\alpha\pi^2}(\log2)^2}.
\end{align*}
\end{theorem}
Note specially the change of the dominant exponential part
$\rho^{\frac12n} n^{\frac12(n+\alpha)+\alpha\omega}$ in
\eqref{E:ext8}, as well as the presence of the extra factor $e^{\beta
\sqrt{n}}$ when compared to \eqref{E:zag} and \eqref{E:blr}. On the
other hand, $\beta>0$ is equivalent to the condition
\eqref{E:de-conds}. When $\beta=0$ (and $e_2>0$), asymptotic
periodicities emerge (depending on the parity of $n$), which
complicate the corresponding expressions. Instead of formulating a
general heavy result, we will be content with ourselves with the
study of Fishburn matrices with
$\lambda_{2i-1}=0$ for $1\le i\le m$ but
$\lambda_2,\lambda_{2m+1}>0$ in Section~\ref{S:130}.

The proof of Theorem~\ref{T:ext} is similar to that of 
Theorem~\ref{T:eg1} and Theorem \ref{T:gen}, beginning first with the
corresponding exponential version and following by the change of 
variables $e(z)=e^{y^2}$ (locally invertible).

\begin{proposition} \label{P:sec8-proto}
For large $n$, $\alpha\in \mathbb{Z}^+$, and $\omega\in\mathbb{C}$, 
\begin{align}\label{E:jelinek-exp}
    [z^n]\sum_{k\ge0}e^{kz}
    \prod_{1\le j\le k}\lpa{e^{(j+\omega)z^2}-1}^{\alpha}
    = c e^{\beta \sqrt{n}}
    \rho^{\frac12n} n^{\frac12(n+\alpha)+\alpha\omega}
    \lpa{1+O\lpa{n^{-\frac12}}},
\end{align}
the $O$-term holding uniformly for bounded $\omega$, 
where $\beta :=\tfrac{\sqrt{6}\log2}{\sqrt{\alpha}\pi}$, and
\begin{align*}
    (c,\rho):=\Lpa{\tfrac{\sqrt{3}}{\sqrt{2}\,\alpha\pi}
    \Lpa{\tfrac1{\Gamma(1+\omega)}
    \sqrt{\tfrac{12}{\alpha\pi}}
    \left(\tfrac{6}{\alpha\pi^2}\right)^\omega}^\alpha
    e^{-\frac1{4\alpha}+\frac3{2\alpha \pi^2}
    (\log 2)^2},\tfrac6{e\pi^2\alpha}}.
\end{align*}
\end{proposition}
\begin{proof} (Sketch)
Similar to the proof of Proposition~\ref{P:eoza}; note that $A_k(z^2)
= \prod_{1\le j\le k}\lpa{e^{jz^2}-1}$ has the same order of
magnitude as $A_k(r^2)$ when $z=re^{i\theta}$ with $\theta\sim\pi$,
but due to the presence of $e^{kz}$, the corresponding Cauchy
integral remains asymptotically negligible.
\end{proof}

%

Note that the proof of Theorem \ref{T:ext} can be extended to the
situation when $m$ ($m\ge 2$) is the smallest nonzero entry, that is,
$\lambda_j=0$ for $1\le j<m$ and $\lambda_m>0$, $m\ge2$.

\subsection{Self-dual $\Lambda$-Fishburn matrices with $\lambda_1>0$}
\label{S:no-1}

We now consider general self-dual $\Lambda$-Fishburn matrices with 
$\lambda_1>0$. 

\begin{lemma}\label{L:self}
The generating function for self-dual $\Lambda$-Fishburn matrices is 
given by ($z$ marking the matrix size)
\begin{align*}
    \sum_{k\ge 0}\Lambda(z)^{k+1}\prod_{1\le j\le k}
    \lpa{\Lambda\lpa{z^2}^{j}-1}.
\end{align*}
\end{lemma}
This lemma is a direct consequence of the case
$\Lambda=\mathbb{Z}_{\ge0}$ given in \cite{Jelinek2012}.

\begin{corollary}\label{C:self-dual_lam1}
If $\lambda_1>0$, then the number of self-dual $\Lambda$-Fishburn 
matrices of size $n$ satisfies
\begin{align*}
    [z^n]\sum_{k\ge0}\Lambda(z)^{k+1}
    \prod_{1\le j\le k}\lpa{\Lambda(z^2)^{j}-1}
    = c e^{\beta \sqrt{n}}
    \rho^{\frac12n} n^{\frac12(n+1)}
    \lpa{1+O\lpa{n^{-\frac12}}},
\end{align*}
where $\beta := \frac{\sqrt{6\lambda_1}}{\pi}\log 2$, and
$
    (c,\rho):=
    \Lpa{\tfrac{3\sqrt{2}}{\pi^{3/2}}\,
    2^{\frac{\lambda_2}{\lambda_1}-\frac{\lambda_1}{2}}
    e^{-\frac{\lambda_1}{4}-\frac{\pi^2}{24}
    +\frac{\pi^2\lambda_2}{12\lambda_1^2}
    +\frac{3\lambda_1}{2\pi^2}(\log 2)^2}
    , \tfrac{6\lambda_1}{e\pi^2}}.
$
\end{corollary}
\begin{proof}
Condition~\eqref{E:de-conds} holds because $d_1>0$ and $e_3=0$. Apply 
Theorem~\ref{T:ext} with $\omega_0=\alpha=1$, $\omega=0$, 
$d_1=e_2=\lambda_1$, $d_2=e_4=\lambda_2$. 
\end{proof}
This implies that if $\lambda_1$ is fixed, then no matter how many
copies other positive integers are used as entries, the resulting
asymptotic count of self-dual matrices of large size differs only in
the leading constant.

\begin{corollary}\label{C:Jelinek}(Conjecture 5.4 of
\cite{Jelinek2012}) The number of self-dual and primitive self-dual
$\Lambda$-Fishburn matrices of size $n$ are asymptotically given by
\begin{align}\label{E:primselfdualfish}
	[z^n]\sum_{k\ge0}(1+z)^{k+1}
	\prod_{1\le j\le k}\left((1+z^2)^{j}-1\right)
	&= \frac{c}{2}\, e^{-\frac{\pi^2}{12}} e^{\beta \sqrt{n}}
	\rho^{\frac12n} n^{\frac12(n+1)}
	\lpa{1+O\lpa{n^{-\frac12}}},\\
	\label{E:selfdualfish}
	[z^n]\sum_{k\ge0}(1-z)^{-k-1}
	\prod_{1\le j\le k}\lpa{(1-z^2)^{-j}-1}
	&= c e^{\beta \sqrt{n}}
	\rho^{\frac12n} n^{\frac12(n+1)}
	\lpa{1+O\lpa{n^{-\frac12}}}, 
\end{align}
where $\beta :=\frac{\sqrt{6}\log2}\pi$ and $ (c,\rho)
:=\Lpa{\frac{6}{\pi^{3/2}}\,e^{\frac{\pi^2}{24}-\frac{1}{4}
+\frac{3}{2\pi^2}(\log2)^2},\frac6{e\pi^2}}$.
\end{corollary}
\begin{remark}
The constant $c\approx 1.361951039$ (see Figure~\ref{F:sdf}) is given
in an approximate numerical form in \cite{Jelinek2012}. By comparing
these estimates with \eqref{E:zag}, we see that the proportion of
self-dual Fishburn matrices is asymptotically negligible (indeed
factorially small).
\end{remark}

\renewcommand{\arraystretch}{1.5}
\begin{figure}[!ht]
    \begin{center}
        \begin{tabular}{cc}
            $\frac{\text{LHS of \eqref{E:selfdualfish}}}
            {e^{\beta\sqrt{n}}
                \rho^{\frac12n} n^{\frac12(n+1)}}
            \left(1-\frac{0.2}{\sqrt{n}}\right)$
            & $\frac{\text{LHS of \eqref{E:primselfdualfish}}}
            {e^{\beta\sqrt{n}}
                \rho^{\frac12n} n^{\frac12(n+1)}}
            \left(1-\frac{1.5}{\sqrt{n}}\right)$ \\ [10pt]
            \includegraphics[height=3.5cm]{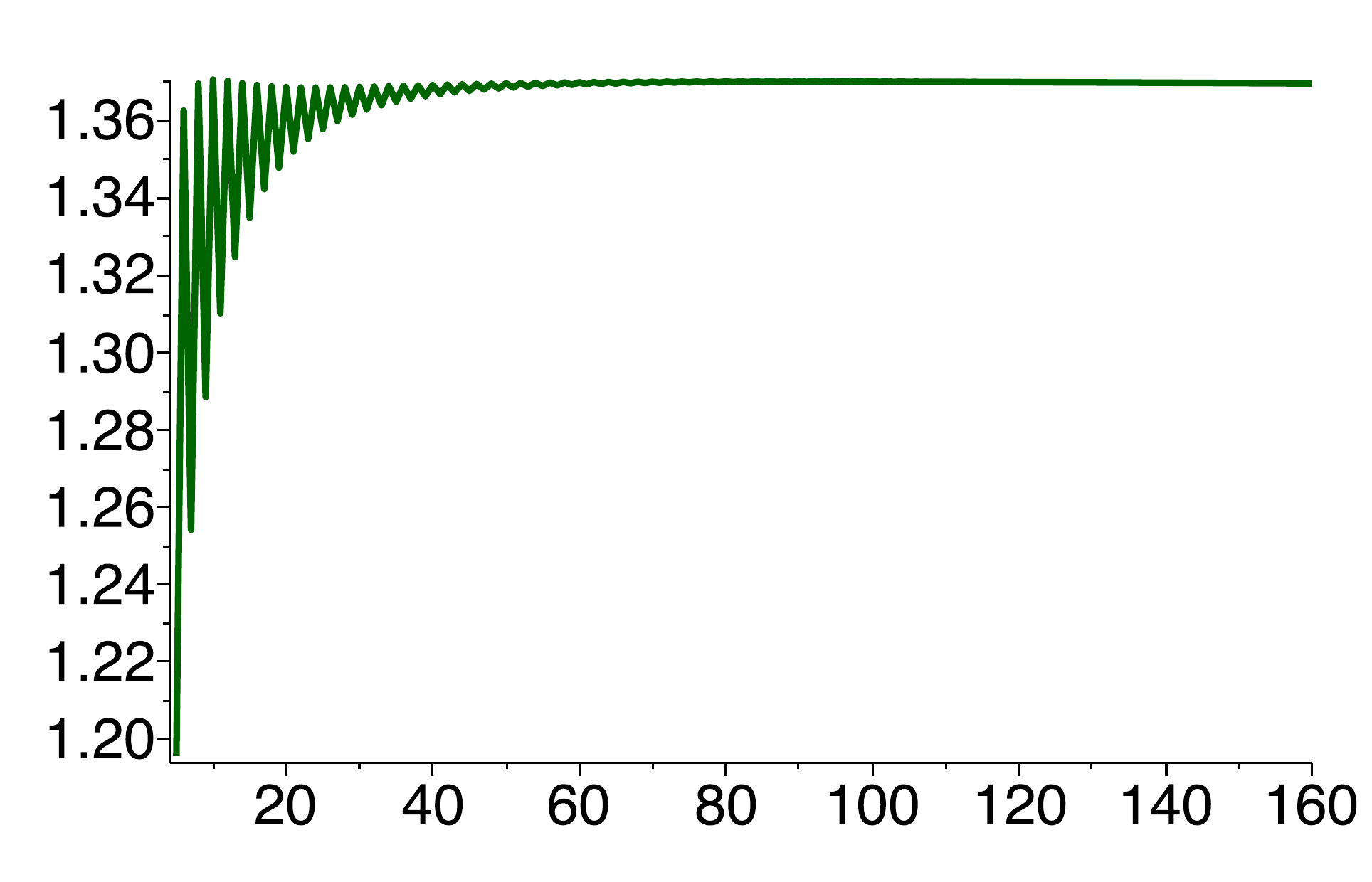} &
            \includegraphics[height=3.5cm]{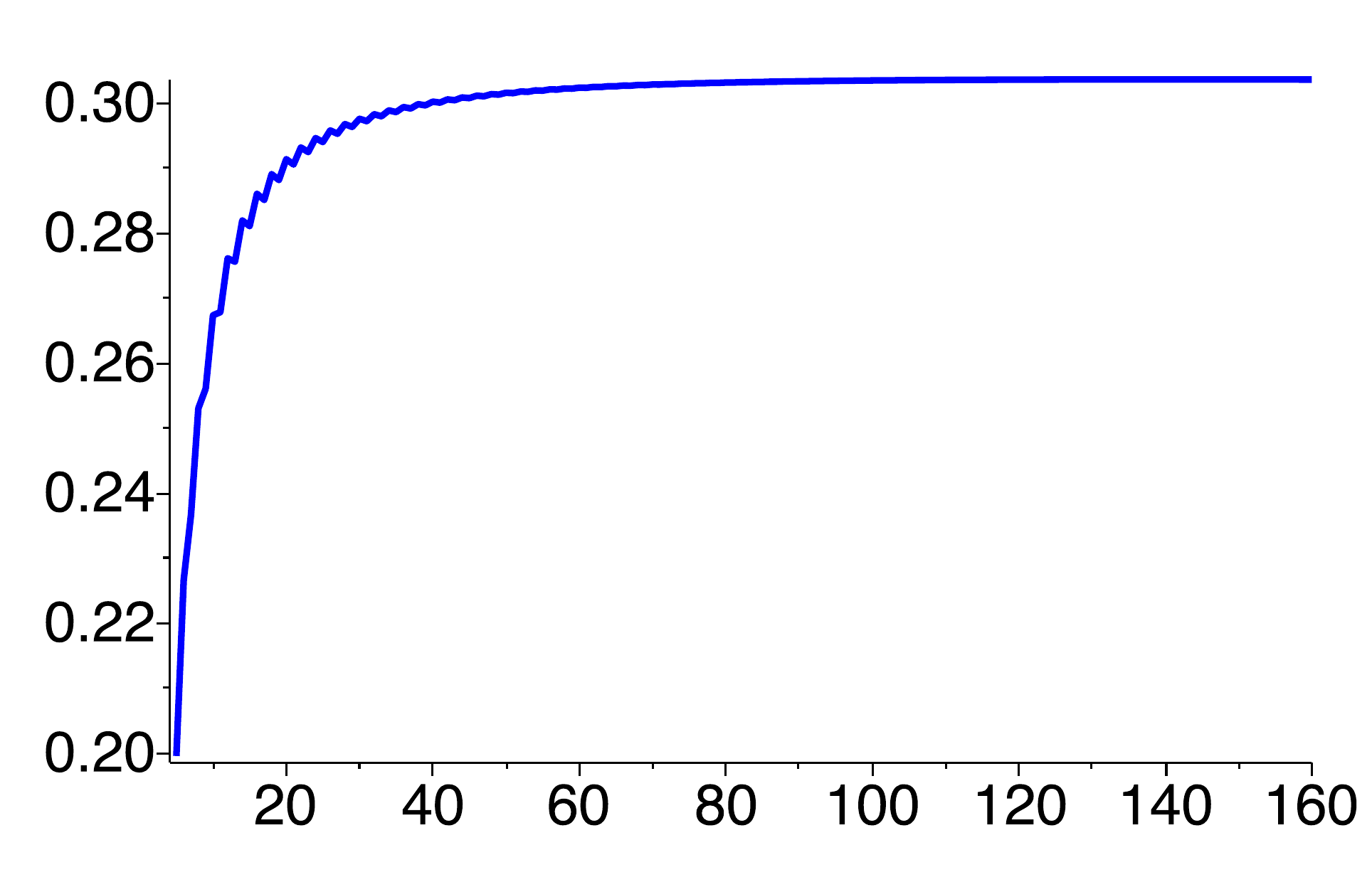} \\
            $\tfrac{6}{\pi^{3/2}}\,e^{\frac{\pi^2}{24}-\frac{1}{4}
                +\frac{3}{2\pi^2}(\log2)^2}\approx 1.362$ &
            $\tfrac{3}{\pi^{3/2}}\,e^{-\frac{\pi^2}{24}-\frac{1}{4}
                +\frac{3}{2\pi^2}(\log 2)^2}\approx 0.299$ \\
            &
        \end{tabular}
    \end{center}
    \vspace*{-.3cm}
    \caption{\emph{Numerical convergence of the two ratios
            $\frac{\text{LHS of \eqref{E:selfdualfish}}}
            {e^{\beta\sqrt{n}}
                \rho^{\frac12n} n^{\frac12(n+1)}}$
            and $\frac{\text{LHS of \eqref{E:primselfdualfish}}}
            {e^{\beta\sqrt{n}}
                \rho^{\frac12n} n^{\frac12(n+1)}}$
            (with proper corrections for the $O$-terms)
            to their respective limit $c$.}}\label{F:sdf}
\end{figure}

We now examine the three statistics (first row-size, diagonal sum,
and the number of $1$s) on random self-dual $\Lambda$-Fishburn
matrices, beginning with the corresponding bivariate generating
functions. For convenience, we include the empty matrix with size $0$.

\begin{proposition}[Statistics on self-dual $\Lambda$-Fishburn
matrices] For self-dual $\Lambda$-Fishburn matrices, we have the 
following bivariate generating functions with $z$ marking the matrix 
size and $v$ marking respectively
\begin{enumerate}[(i)]
\item the size of the first row
\begin{align}\label{E:sd-fr}
    \Lambda(vz)\sum_{k\ge 0}\Lambda(z)^k
    \prod_{1\le j\le k}\lpa{\Lambda\lpa{vz^2}
    \Lambda\lpa{z^2}^{j-1}-1},
\end{align}

\item the size of the diagonal
\begin{align}\label{E:sd-dg}
    \Lambda(vz)\sum_{k\ge 0}
    \Lambda(z)^{k}\prod_{0\le j<k}
    \lpa{\Lambda\lpa{v^2z^2}\Lambda\lpa{z^2}^{j} -1}, \and 
\end{align}

\item the number of $1$s 
\begin{align}\label{E:sd-1}
    \sum_{k\ge 0}(\Lambda(z)+\lambda_1(v-1)z)^{k+1}
    \prod_{1\le j\le k}
    \lpa{\lpa{\Lambda(z^2)+\lambda_1(v^2-1)z^2}^j -1}.
\end{align}
\end{enumerate}
\end{proposition}
This is in analogy to Proposition~\ref{P:lfm-stat}, using 
the same ideas in \cite{Jelinek2012} for counting self-dual 
matrices.

\begin{theorem}\label{T:sec8-sd}
Assume $\lambda_1>0$ and that all self-dual $\Lambda$-Fishburn
matrices of size $n$ are equally likely to be selected. Then in a
random matrix, the size $X_n$ of the first row (or the last column) and the half of the 
diagonal size $\frac12Y_n$ both satisfy a central limit theorem with 
logarithmic mean and variance:
\begin{align*}
    \frac{X_n-\log n}{\sqrt{\log n}}
    \xrightarrow{d} \mathscr{N}(0,1),
    \and
    \frac{\frac12Y_n-\log n}{\sqrt{\log n}}
    \xrightarrow{d} \mathscr{N}(0,1),
\end{align*}
and for the number $Z_n$ of $1$s, if $\lambda_2>0$, then $n-Z_n$ is 
the convolution of two Poisson variates:
\begin{align*}
    n-Z_n\xrightarrow{d} 
    2\mathrm{Poisson}\lpa{\tfrac{\lambda_2}{\lambda_1}\log 2}
    * 4\mathrm{Poisson}\lpa{\tfrac{\lambda_2\pi^2}{12\lambda_1}},
\end{align*}
while if $\lambda_2=0$, then $\mathbb{P}(Z_n=n)\to 1$.
\end{theorem}
\begin{proof}
(Sketch) The proofs for the random variables $X_n$ and $\frac{1}{2}Y_n$ rely on Theorem~\ref{T:ext}, in parallel of Theorem~\ref{T:lf-stat}. For the number of $1$s, Theorem~\ref{T:ext} does
not apply to \eqref{E:sd-1} because $e_2=\lambda_1v^2$ is a
complex number in general and $e_2>0$ may not hold. However, the
proof there does apply by considering $e(z/\sqrt{e_2})$, similar
to Theorem~\ref{T:gen}. The result is the same as if we apply
formally Theorem~\ref{T:gen} with $\omega_0=\alpha=1$,
$\omega=0$, $d_1=\lambda_1v$, $d_2=\lambda_2$, $e_2 =
\lambda_1v^2$, $e_3=0$, $e_4 = \lambda_2$, yielding
\[
\mathbb{E}\lpa{v^{n-Z_n}}
= 2^{\frac{\lambda_2}{\lambda_1}(v^2-1)}
e^{\frac{\lambda_2\pi^2}{12\lambda_1^2}(v^4-1)}
\lpa{1+O\lpa{n^{-\frac12}}},
\]
where the first term on the right-hand side is the probability
generating function of two Poisson distributions if
$\lambda_2>0$. The right-side becomes $1$ when $\lambda_2=0$.

\end{proof}

\subsection{Asymptotics of $\Lambda$-Fishburn matrices whose smallest nonzero entry is $2$}\label{S:8-3}

We consider Fishburn matrices whose smallest nonzero entry is $2$. 
We assume that there is at least an odd number in $\Lambda$, namely,
\begin{equation}\label{E:odd-entries}
    \lambda_{2k-1}=0, \mbox{ for } 1\le k\le m
    \and \lambda_2,\lambda_{2m+1}>0,
\end{equation}
for $m\ge1$. Otherwise, if $\Lambda$ contains only even numbers,
then, by dividing all entries by $2$, the corresponding asymptotics
and distributional properties can be dealt with by the same framework
of Section~\ref{S:gf}. It turns out that $m=1$ (that is,
$\lambda_1=0$ but $\lambda_2, \lambda_3>0$) and $m\ge2$ have
different asymptotic behaviors, and in the latter case the dependence
on the parity of $n$ is more pronounced, one technical reason being
that the condition \eqref{E:de-conds} fails when $m\ge2$, and the odd
case needs special treatment.

\begin{lemma} \label{L:Bz}
Given a formal power series $B(z) = \sum_{n\ge0}b_n z^n$ with $b_n
\simeq c_0 \rho_0^n n^{n+t}$, $\rho_0\ne0$, we have, for even $n$,
\begin{align*}
	[z^{\frac12n}] e^{\beta n z} B(z) 
	\simeq c\rho^{\frac12n} n^{\frac12n+t} ,
	\with 
	(c,\rho) := \lpa{c_02^{-t}e^{\frac{2\beta}{e\rho_0^{}}},
	\tfrac12\rho_0}.
\end{align*}
\end{lemma}
\begin{proof}
Expand $e^{n\beta z}$ at $z=\frac{2}{e\rho_0 n}$, the asymptotic 
saddle-point of $z^{-\frac12n}B(z)$, and estimate the error as in 
Section~\ref{S:cov}. 
\end{proof}

\begin{theorem}[$\Lambda$-Fishburn matrices with $2$ as the smallest
entries]\label{T:small2} Assume that $\Lambda$ is a multiset of
nonnegative integers satisfying \eqref{E:odd-entries} with
$\Lambda(0)=1$. If $m=1$, then the number of $\Lambda$-Fishburn
matrices of size $n$ satisfies
\begin{align*}
    [z^n]\sum_{k\ge0}
    \prod_{1\le j\le k}\lpa{1-\Lambda(z)^{-j}}
    = c e^{\beta \sqrt{n}}
    \rho^{\frac12n} n^{\frac12n+1}
    \lpa{1+O\lpa{n^{-\frac12}}},
\end{align*}
where $\beta := \frac{\lambda_3\pi}{2\sqrt{3}\,\lambda_2^{3/2}}$, and
$(c,\rho)
    := \Lpa{\tfrac{3\sqrt{6}}{\pi^2}\,
    e^{\frac{\pi^2}{6}\lpa{\frac{\lambda_4}{\lambda_2^2}-\frac12
    -\frac{7\lambda_3^2}{8\lambda_2^3}}}
    , \tfrac{3\lambda_2}{e\pi^2}}$;
and if $m\ge2$, then 
\begin{align}\label{E:fm13}
    [z^n]\sum_{k\ge0}\prod_{1\le j\le k}
	\lpa{1-\Lambda(z)^{-j}}
    =\begin{cases}
        c'e^{\beta \sqrt{n}}\rho^{\frac12n}
        n^{\frac12n+1}\lpa{1+O\lpa{n^{-\frac12}}}, 
        & \text{ if $n$ is even};\\
        c_me^{\beta \sqrt{n}}\rho^{\frac12n}
        n^{\frac12n-m+\frac52}
        \lpa{1+O\lpa{n^{-\frac12}}}, 
        & \text{ if $n$ is odd},
    \end{cases}
\end{align}
where $\rho$ and $\beta$ remain the same, 
$c' :=\tfrac{6\sqrt{6}}{\pi^2}\, 
    e^{\frac{\pi^2}{6}
    \lpa{\frac{\lambda_4}{\lambda_2^2}-\frac12}}, $ and $
	c_m := \tfrac{\sqrt{2}\pi^{2m-3} }
	{3^{m-2}}\cdot\tfrac{\lambda_{2m+1}}
	{\lambda_2^{m+1/2}} \,e^{\frac{\pi^2}{6}
    \lpa{\frac{\lambda_4}{\lambda_2^2}-\frac12}}.
$
\end{theorem}
\begin{proof} (Sketch)
When $m=1$, apply Theorem~\ref{T:ext} to the right-hand side of
\eqref{E:gf-lf}. When $m=2$, following the proof of
Theorem~\ref{T:ext}, we begin with the change of variables
$\Lambda(z)=e^{y^2}$ and then apply Theorem~\ref{T:gen} and
Lemma~\ref{L:Bz} to prove \eqref{E:fm13}. 
In particular, when $m\ge2$, by splitting $\Lambda(z)$ into odd and 
even parts, using Lagrange's inversion formula in the form
\[
    [y^k]z = \frac1k[t^{k-1}]
	\Lpa{\frac t{\log\Lambda(t)}}^k \qquad(k=1,2,\dots),
\]
we deduce \eqref{E:fm13} when $n$ is even; the expression of $c_m$
then follows from that in the even case.

\end{proof}

In particular, the number of Fishburn matrices without occurrence of 
$1$ as entries ($\Lambda=Z_{\ge0}\setminus \{1\}$) satisfies 
\begin{align*}
    [z^n]\sum_{k\ge0}
    \prod_{1\le j\le k}\llpa{1-\Lpa{\frac{1-z}{1-z+z^2}}^j}
    = c e^{\beta \sqrt{n}}
    \rho^{\frac12n} n^{\frac12n+1}
    \lpa{1+O\lpa{n^{-\frac12}}},
\end{align*}
where $\beta := \frac{\pi}{2\sqrt{3}}$, and $(c,\rho)
:= \lpa{\tfrac{3\sqrt{6}}{\pi^2}\, e^{-\frac{\pi^2}{16}}, 
\tfrac{3}{e\pi^2}}$, which marks a significant difference with that 
containing $1$ as entries, as given in \eqref{E:zag}. Similar behaviors are also exhibited in the asymptotics of row-Fishburn matrices without entry $1$.

On the other hand, asymptotics of $\Lambda$-row-Fishburn matrices can
be similarly treated, and exhibits a very similar behavior.

\subsection{Statistics on $\Lambda$-Fishburn matrices whose smallest nonzero entry is $2$}\label{S:8-4}
\label{S:130}

Based on the generating functions of Proposition~\ref{P:lfm-stat}, we 
now consider the behavior of a general random $\Lambda$-Fishburn 
matrix with $2$ being the smallest nonzero entry. 

\begin{theorem}\label{T:phase}
Assume that $\Lambda$ satisfies \eqref{E:odd-entries}. If all
$\Lambda$-Fishburn matrices of size $n$ are equally likely to be
selected, then, in a random matrix under this distribution,
the size $X_n$ of the first row (or the last column) and the diagonal size $Y_n$ are 
both asymptotically normally distributed in the following sense:
\begin{align*}
	\frac{X_n-\log n}{\sqrt{\log n}}
	&\xrightarrow{d} \mathscr{N}(0,1),
    \and
    \frac{Y_n-2\log n}{\sqrt{2\log n}}
    \xrightarrow{d} \mathscr{N}(0,1),
\end{align*}
while the limiting distribution of the number $Z_n$ of occurrences of 
$2$ depends on $m$: if $m=1$, then 
\begin{align}\label{E:root-n-clt}
    \frac{\frac13(n-2Z_n)-\tau\sqrt{n}}{\sqrt{\tau\sqrt{n}}}
    \xrightarrow{d} \mathscr{N}(0,1),
    \qquad\lpa{\tau := \tfrac{\lambda_3\pi}
    {2\sqrt{3}\lambda_2^{3/2}}},
\end{align}
and if $m\ge2$, then 
\[
    Z_n^*\xrightarrow{d}\mathrm{Poisson}
	\lpa{\tfrac{\lambda_4\pi^2}{6\lambda_2}},
    \with 
    Z_n^* := \begin{cases}
        \frac12\lpa{\frac12n-Z_n},&\text{ if $n$ is even};\\
        \frac12\lpa{\frac12(n-2m-1)-Z_n},&\text{ if $n$ is odd}.
    \end{cases}
\]
\end{theorem}
\begin{proof}


\begin{enumerate}[(i)]
   \item When $m=1$, for the first row sum, we have, by the
   generating function \eqref{E:gf-f-fr}, the approximation
   \eqref{E:log-es} and a modification of the proof of
   Theorem~\ref{T:ext},
   \begin{align*}
       [z^n] \Lambda(vz)\sum_{k\ge 0}\Lambda(z)^k
       \prod_{1\le j\le k}\lpa{\lpa{\Lambda(vz)\Lambda(z)^{j-1}-1}
      \lpa{\Lambda(z)^j-1}}
       = c(v) \rho^{\frac12n} n^{\frac12n+v}
       \lpa{1+O\lpa{n^{-\frac12}}},
   \end{align*}
   where
   $
       (c(v),\rho) :=
       \Lpa{\tfrac{\sqrt{6}}{\Gamma(v)}\lpa{\tfrac 3{\pi^2}}^v
       e^{\frac{\pi^2}6\lpa{\frac{\lambda_4}{\lambda_2^2}
       -\frac{7\lambda_3^2}{8\lambda_2^3}-\frac12}},
       \tfrac{3\lambda_2}{e\pi^2}}.
   $
   Thus
   \begin{align}\label{Esec8-vXn}
       \mathbb{E}\lpa{v^{X_n}} = \frac1{\Gamma(v)}
       \Lpa{\frac{3}{\pi^2}}^{v-1}n^{v-1}
       \lpa{1+O\lpa{n^{-\frac12}}},
   \end{align}
   uniformly for bounded $v$. Then the asymptotic normality
   (or Poisson$(\log n)$) follows from Quasi-Powers theorem.
   When $m\ge2$, by the same procedure as in the proof of
   Theorem~\ref{T:small2}, 
   we then deduce the same asymptotic approximation
   \eqref{Esec8-vXn} when $n$ is even. When $n$ is odd,
   the corresponding asymptotic approximation differs by a factor of
   $n^{-m}$ as in Theorem~\ref{T:small2} but the resulting
   normalized expression is still \eqref{Esec8-vXn}.

   \item Very similarly, for the diagonal size, by applying 
   Theorem~\ref{T:ext} to the generating function \eqref{E:gf-f-dg},
   we deduce that 
  \[
       \mathbb{E}\lpa{v^{Y_n}}
      = \frac1{\Gamma(v)^2}\Lpa{\frac{3}{\pi^2}}^{2(v-1)}
       n^{2(v-1)}\lpa{1+O\lpa{n^{-\frac12}}},
   \]
   uniformly for bounded $v$. The same expression remains true when
   $m\ge2$ and the proof proceeds along the lines of that of
   Theorem~\ref{T:small2}.

   \item Regarding the number of $2$s, it is more involved. Consider first $m=1$.
   Parallel to \eqref{E:gf-f-1} for the number of $1$s, we now have
   the generating function
   \begin{equation} \label{E:nb-2s}
   \begin{split}
       &\sum_{k\ge0}\prod_{1\le j\le k}
       \lpa{1-\lpa{\Lambda(z)+\lambda_2(v-1)z^2}^{-j}}\\
       &\qquad=\sum_{k\ge0}
       \lpa{\Lambda(z)+\lambda_2(v-1)z^2}^{k+1}
       \prod_{1\le j\le k}
       \lpa{\lpa{\Lambda(z)+\lambda_2(v-1)z^2}^j-1}^2.
   \end{split}
   \end{equation}
   Then if $\lambda_3>0$, we get, by a similar modification of the
   proof of Theorem~\ref{T:ext} (see Theorem~\ref{T:sec8-sd}), the
   Quasi-Powers approximations,
   \[
       \mathbb{E}\lpa{v^{\frac12 n-Z_n}}
       = c(v) e^{\tau\sqrt{n}(v^{\frac32}-1)}
       \lpa{1+O\lpa{n^{-\frac12}}},
   \]
   where $\tau$ is given in \eqref{E:root-n-clt} and
   $
       c(v) :=
       \,e^{-\frac{7\lambda_3^2\pi^2}{48\lambda_2^3}
       (v^3-1)+\frac{\lambda_4\pi^2}{6\lambda_2^2}
       (v^2-1)}.
   $
   The asymptotic normality then results from the Quasi-Powers
   theorems; see \cite{Flajolet2009, Hwang1998}. Indeed, $Z_n^*$ is
   asymptotically Poisson distributed with parameter $\tau\sqrt{n}$.

   When $m\ge2$, we obtain, by \eqref{E:nb-2s}, the change of
   variables $\Lambda(z)+(\lambda_2-1)vz^2 = e^{y^2}$ and modifying
   the proof of \eqref{E:fm13} (see also the proof of
   Theorem~\ref{T:sec8-sd}),
   \[
       \mathbb{E}\lpa{v^{\frac12n-Z_n}}
       =\begin{cases}
           e^{\frac{\lambda_4\pi^2}{6\lambda_2^2}(v^2-1)},
           &\text{if $n$ is even};\\
           v^{m+\frac12}e^{\frac{\lambda_4\pi^2}
           {6\lambda_2^2}(v^2-1)},
           &\text{if $n$ is odd}.
       \end{cases}
   \]
   This proves the Poisson limit law.
\end{enumerate} 
\end{proof}

\begin{figure}[!ht]
\begin{center}
\begin{tabular}{cc}
    \includegraphics[height=3.5cm]{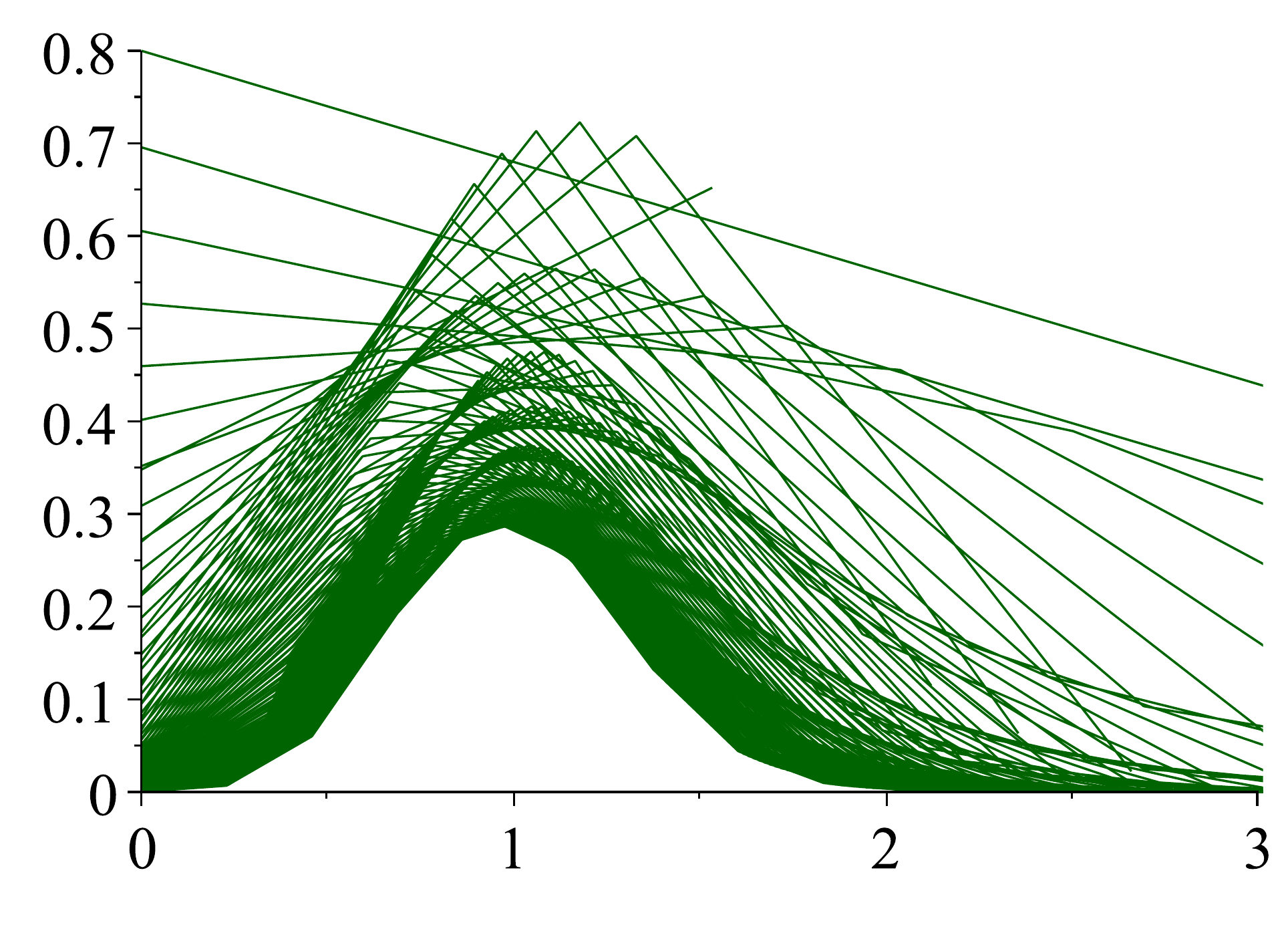} 
    & \includegraphics[height=3.5cm]{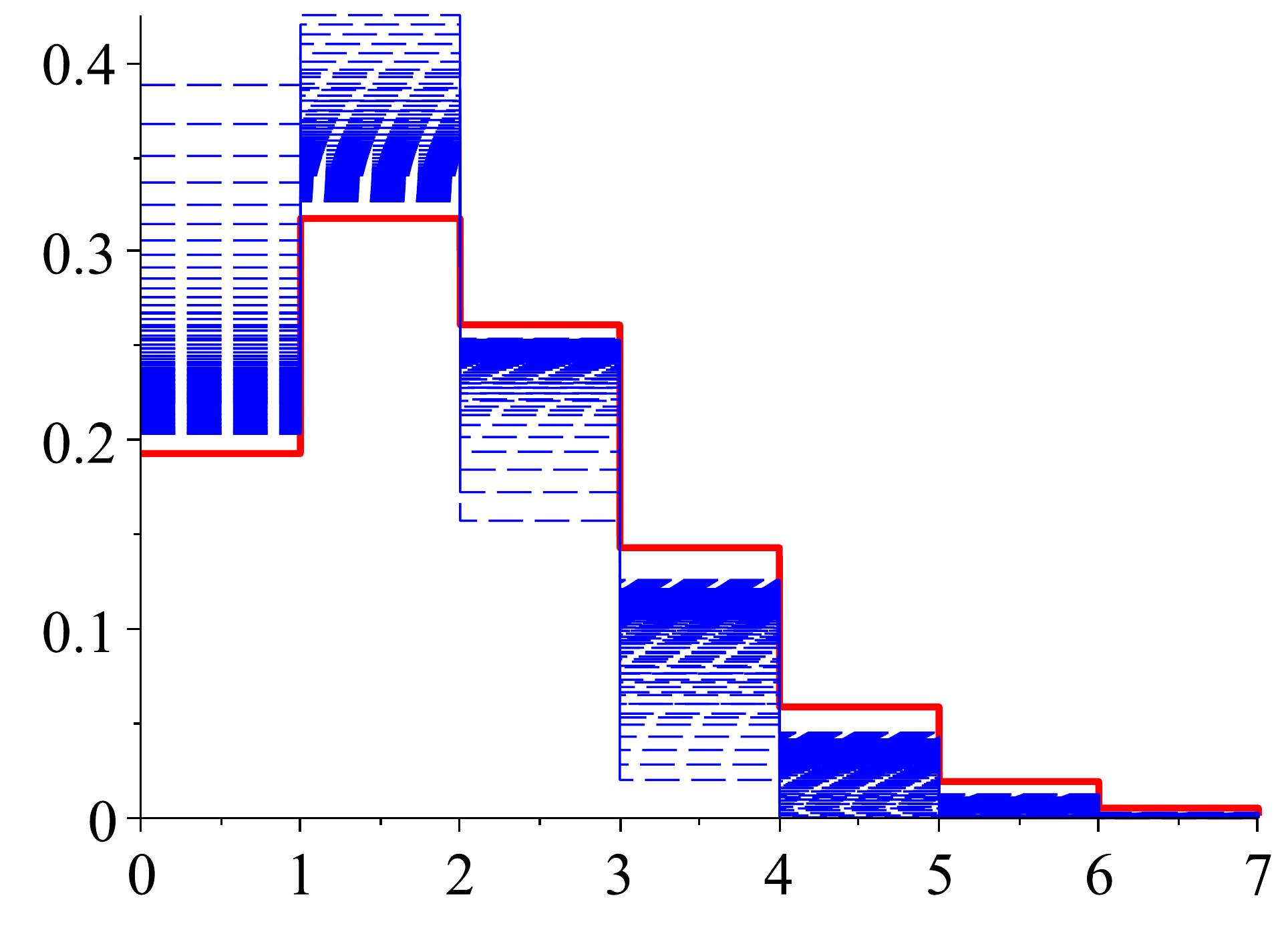}\\
    $\Lambda(z) = 1+z^2+z^3$ & $\Lambda(z) = 1+z^2+z^4+z^5$ \\
    $\mathscr{N}\lpa{\frac{\pi\sqrt{n}}{2\sqrt{3}}, 
    \frac{\pi\sqrt{n}}{2\sqrt{3}}}$ & 
    Poisson$\lpa{\frac{\pi^2}{6}}$ 
\end{tabular}    
\end{center}    
\caption{Histograms of the number of $2$s in two different
compositions of random $\Lambda$-Fishburn matrices. Left: the
distributions $\mathbb{P}\lpa{\frac13\lpa{\frac n2-Z_n}=\tr{x\mu_n}}$
with $\mu_n$ denoting the exact mean, which is asymptotic to
$\frac{\pi\sqrt{n}}{2\sqrt{3}}$; right: $\mathbb{P}\lpa{Z_n^*=k}$,
where $Z_n^* := \frac12\lpa{\frac n2-Z_n}$ when $n$ is even, and
$Z_n^* := \frac12\lpa{\frac n2-2-Z_n}$ when $n$ is odd, where the red
line represents the corresponding Poisson distribution.}
\end{figure}

For random $\Lambda$-row-Fishburn matrices, one can derive very 
similar types of results: zero-truncated Poisson with parameter 
$\frac{\lambda_1}{\lambda_2}\log 2$ for the first row size, 
$\mathscr{N}(\log n,\log n)$ for the diagonal size, and 
$\mathscr{N}\lpa{\tau\sqrt{n},\tau\sqrt{n}}$ with $\tau := 
\tfrac{\lambda_3\pi}{2\sqrt{6}\lambda_2^{3/2}}$ or  
Poisson($\frac{\lambda_4\pi^2}{12\lambda_2^2})$ limit law 
when $m=1$ or $m\ge2$, respectively, for the number of $2$s.

\section{Conclusions}
\label{S:conclusions}

Motivated by the asymptotic enumeration of and statistics on Fishburn
matrices and their variants, we developed in this paper a
saddle-point approach to computing the asymptotics of the
coefficients of generating functions with a sum-of-product form, and
applied it to several dozens of examples. The approach is not only
useful for the usual large-$n$ asymptotics but also effective in
understanding the stochastic behaviors of random Fishburn matrices,
with or without further constraints on the entries or on the
structure of the matrices. In particular, we identified a simple yet
general framework and showed its versatile usefulness in this paper.
Many new asymptotic distributions of statistics on random matrices
are derived in a systematic and unified manner, which in turn demand
further structural interpretations; for example, since the normal 
approximations we derived in this paper can indeed all be
approximated by Poisson distributions with parameters depending on 
$n$ (equal to the asymptotic mean), a natural question is why 
Poisson laws with bounded or unbounded parameters are ubiquitous in 
the random $\Lambda$-Fishburn matrices.

Other frameworks will be examined in a follow-up paper. In addition to
different sum-of-product patterns, we will also work out cases for
which our approach in this paper does not directly apply. For
example, we have not found transformations for the series 
$[z^n]\sum_{k\ge0}\prod_{1\le j\le k}\tanh(2jz)$,
a special case of general theorems in \cite{Andrews2001}, such that 
our saddle-point method works, although it is known (see 
\cite{Andrews2001}) that 
\[
		\sum_{k\ge0}\prod_{1\le j\le k}\tanh(2jz)
	    =\sum_{n\ge0}\frac{a_{2n+1}}{n!}\,z^n,\and
	    \sum_{n\ge0}\frac{a_{2n+1}}{(2n+1)!}\,z^{2n+1}
	    =\tan(z),
\]
where the $a_{2n+1}$'s are the tangent numbers. For similar pairs of
series of this type, see \cite{Andrews2001, Hikami2006a, Ono2004}.

Finally, the rank (or dimension) of a random $\Lambda$-Fishburn
matrices represents another important statistic on random matrices,
which is expected to follow a central limit theorem with linear mean
and variance. Indeed, if we replace $e^z$ in Remark~\ref{R:llt} by
$1/(1-z)$, then the local limit theorem given there still holds,
which can be interpreted as the distribution of the dimension in a
random row-Fishburn matrix of size $n$. The situation for random
Fishburn matrices is however less clear as we do not have a simple
decomposition as in row-Fishburn ones. This and related quantities
will be investigated and clarified elsewhere.

\section*{Acknowledgements}
We thank Cyril Banderier, Michael Drmota, Christian Krattenthaler, 
Yingkun Li, Shuhei Mano for helpful comments. 

\bibliographystyle{abbrv}
\bibliography{fishburn}

\end{document}